\documentclass[11pt]{article}%
\usepackage{amsfonts}
\usepackage{amssymb,amsmath,amsthm, amsfonts}
\usepackage{mathtools}

\makeatletter

\newcommand{\Rmnum}[1]{\expandafter\@slowromancap\romannumeral #1@}
\makeatother

\providecommand{\U}[1]{\protect \rule{.1in}{.1in}}
\newtheorem{theorem}{Theorem}[section]

\newtheorem{definition}{Definition}[section]

\newtheorem{lemma}{Lemma}[section]

\newtheorem{proposition}{Proposition}[section]
\newtheorem{remark}{Remark}[section]
\newtheorem{example}{Example}[section]

\def\esssup{\hbox{\rm ess$\,$\rm sup$\,$}}
\def\esssup{\mathop{\rm esssup}}
\def\essinf{\mathop{\rm essinf}}
\def\sup{\mathop{\rm sup}}
\makeatletter
   
   \@addtoreset{equation}{section}
\makeatother

\begin{document}

\title{Zero-sum and nonzero-sum differential games without Isaacs condition
\footnote{The work has been supported by the NSF of P.R.China (Nos. 11071144, 11171187, 11222110), Shandong Province (Nos. BS2011SF010, JQ201202), SRF for
 ROCS (SEM), Century Excellent Talents in University (No. NCET-12-0331), 111 Project (No. B12023).}}

\author{
 Juan Li$^{1}$, Wenqiang Li$^{1,}\footnotemark[4]$ \\
{\small$^{1}$ School of Mathematics
and Statistics, Shandong University, Weihai,}\\
{\small Weihai, 264209, P. R. China.}\\
{\small \textit{E-mails: juanli@sdu.edu.cn; wenqianglis2009@gmail.com}}\\
}
\footnotetext[4]{Corresponding author.}
\date{July 17, 2015}
\maketitle

\bigskip
\noindent \textbf{Abstract.}
In this paper we study the zero-sum and nonzero-sum differential games with not assuming Isaacs condition.
Along with the partition $\pi$ of the time interval $[0,T]$, we choose the suitable random non-anticipative strategy with delay
to study our differential games with asymmetric information.
Using Fenchel transformation, we prove that the limits of the upper value function $W^\pi$ and lower value function $V^\pi$
 coincide when the mesh of partition $\pi$ tends to 0. Moreover, we give a characterization for the Nash equilibrium payoff (NEP, for short) of our nonzero-sum differential games without Isaacs condition, then we prove the existence of the NEP of our games. Finally, by considering all the strategies along with all partitions, we give a new characterization for the value of our zero-sum differential game with asymmetric information under some equivalent Isaacs condition.
\bigskip


\section{{\protect \large {Introduction}}}
Zero-sum stochastic differential games have developed rapidly since the pioneering work \cite{FS1989} by Fleming and Souganidis after
they firstly introduced the Isaacs condition and characterized the value of these games as a viscosity solution of some Hamilton-Jacobi-Isaacs equation. Hamad\`ene and Lepeltier in \cite{HL1995} characterized the value of zero-sum stochastic differential game as a solution of some backward stochastic differential equation (BSDE, for short) under the equivalent Isaacs condition. Cardaliaguet in \cite{C2007} studied the zero-sum differential game with asymmetric information under the Isaacs condition and characterized the value of this game as a dual solution of some Hamilton-Jacobi equation.

On the other hand, nonzero-sum differential games with Isaacs condition have been studied by many authors. When playing ``control against control", Hamad\`ene, Lepeltier and  Peng in \cite{HLP1997} characterized the Nash equilibrium point for nonzero-sum stochastic differential games as the solution of some BSDE. Then Buckdahn, Cardaliaguet and Rainer in \cite{BCR2004} gave a new definition of Nash equilibrium payoffs (NEP, for short) and
characterized this NEP, then they obtained the NEP through a approximated method with playing ``strategy against strategy". Rainer in \cite{R2007} compared the two approaches used in \cite{BCR2004} and \cite{HLP1997}  for nonzero-sum stochastic differential game and obtained that the two definitions of Nash equilibrium payoffs coincide when they both exist. Lin in \cite{L2012} generalized the result of \cite{BCR2004} into the case with the nonlinear payoffs.

Recently, some authors (such as, Buckdahn, Li, Quincampoix, etc., in \cite{BLQ20131}, \cite{BLQ20132} and \cite{BQRX2014}) tried to investigate the zero-sum stochastic differential games  without the Isaacs condition. In \cite{BLQ20131} and \cite{BLQ20132}, the authors considered the zero-sum differential games and zero-sum stochastic differential games, respectively, with symmetric information and without Isaacs condition by using a suitable notion of mixed strategies and proved the existence of the value of these games. Buckdahn, Quincampoix, Rainer and Xu in \cite{BQRX2014} generalized the case without Isaacs condition into the zero-sum differential games with asymmetric information. All these papers (\cite{BLQ20131}, \cite{BLQ20132} and \cite{BQRX2014}) used the strategy along the partition $\pi$ of  the time interval $[0,T]$ and showed that the upper and lower value function $W^\pi$, $V^\pi$ defined along the partition $\pi$ coincide as the mesh of $\pi$ tends to 0. Although they have proved the upper and lower value function $W^\pi$, $V^\pi$ converge to the same function $U$ when $|\pi|$ (the mesh of $\pi$) tends to 0, one may want to know the characterization for this value $U$. For this, we try to
give a characterization for the value $U$ in Section 5. Inspired by the above papers without Isaacs condition, we want to study the NEP for the non-zero sum differential games without Isaacs condition in Section 4.

More details, we consider the following  dynamics:
\begin{equation}\label{equ 1.1}
X_s=x+\int_t^sf(X_r,u_r,v_r)dr,\ s\in[t,T],
\end{equation}
where $u$ and $v$ are stochastic processes taking value in compact $U$ and $V$, respectively, $f: \mathbb{R}^n\times U\times V
\mapsto \mathbb{R}^n$ is bounded, Lipschitz in $x$, uniformly in $(u,v)$. For any fixed partition $\pi$ of time interval $[0,T]$,
we give a generalized defintion of non-anticipative strategy with delay along this partition $\pi$ (see, Def. \ref{def 2.2}) which has the property that: for any partitions
$\pi_1$ and $\pi_2$ with``$\pi_1\subset\pi_2$'' (the partition points of $\pi_2$ contain all of the partition points of $\pi_1$), it holds the strategy set $\mathcal{A}^{\pi_1}(t,T)\subset\mathcal{A}^{\pi_2}(t,T)$ for Player \Rmnum {1}; similarly, we have that for Player \Rmnum {2}. Along with the partition $\pi$, we define the upper and lower value functions $W^\pi(t,x,p,q)$ and $V^\pi(t,x,p,q)$ (more details see Section 2) for our zero-sum differential game with asymmetric.

In Section 3, we firstly show that the upper and lower value function $W^\pi$ and $V^\pi$ defined by the strategy from $\mathcal{A}^\pi(t,T)$ and $\mathcal{B}^\pi(t,T)$
are just the upper and lower value function $W_1^\pi$ and $V_1^\pi$ defined by the strategy from $\mathcal{A}_1^\pi(t,T)$ and $\mathcal{B}_1^\pi(t,T)$ which is the subset of $\mathcal{A}^\pi(t,T)$ and $\mathcal{B}^\pi(t,T)$, respectively. Then, with the help of Fenchel transform, we prove a sub-dynamic programming principle (sub-DPP, for short) for the conjugate functions of $W^\pi$ and $V^\pi$, then  we show that $W^\pi$ and $V^\pi$ converge to the same function $U$ as the mesh of $\pi$ tends to 0 without Isaacs condition. Moreover, this value $U$ can be characterized as the unique dual viscosity solution of the following Hamilton-Jacobi-Isaacs equation
\begin{equation}\label{equ 1.2}
\left\{
\begin{array}{ll}
\frac{\partial V}{\partial t}(t,x)+H(x,D{V}(t,x))=0,&(t,x)\in[0,T]\times\mathbb{R}^n,\\
V(T,x)=\sum_{i,j}p_iq_jg_{ij}(x),& (p,q)\in\Delta(I)\times\Delta(J),
\end{array}
\right.
\end{equation}
where $H(x,\xi)=\inf_{\mu\in\mathcal{P}(U)}\sup_{\nu\in\mathcal{P}(V)}
\big(\int_{U\times V}f(x,u,v)\mu(du)\nu(dv)\cdot\xi\big).$

In Section 4, we mainly consider  the nonzero-sum differential game with symmetric information (i.e., $I=J=1$) and without Isaac condition.
Inspired by the definition of NEP used in \cite{BCR2004}, we introduce a new definition of NEP for our nonzero-sum differential games. Using the value function $U$ that we find in Section 3, we show a characterization for our NEP.
Furthermore, we prove the existence of the NEP for our nonzero-sum differential games without Isaacs condition using this characterization.

In Section 5, we give a characterization for the value function $U$ (found in Section 3) of our zero-sum differential games with asymmetric information. With
the property that $\mathcal{A}^{\pi_1}(t,T)\subset\mathcal{A}^{\pi_2}(t,T)$ if ``$\pi_1\subset\pi_2$", we can consider the strategies in $\mathcal{A}(t,T)$ where $\mathcal{A}(t,T)$ is the union of the $\mathcal{A}^{\pi}(t,T)$ with all the partitions $\pi$ for Player \Rmnum {1}, similarly that for Player \Rmnum {2}. Then we show that the upper and lower value function $W(t,x,p,q)$, $V(t,x,p,q)$ defined by the strategies from $\mathcal{A}(t,T)$ and $\mathcal{B}(t,T)$ coincide with the value function $U(t,x,p,q)$ which is the unique dual viscosity solution of Hamilton-Jacobi-Isaac equation (\ref{equ 1.2}) under some equivalent Isaacs condition. Therefore, we also provide a new numerical method for calculating the value of the zero-sum differential game with asymmetric information. At last, we give an example to illustrate that the equivalent Isaacs condition is necessary.

Our paper is organized as follows. In Section 2 we give some introduction about the dynamic and the strategies for our games. Section 3 is devoted to proving the existence of the value of our zero-sum differential game with asymmetric information and without Isaacs condition. In Section 4 we prove the existence of the Nash equilibrium payoffs of our nonzero-sum differential game with symmetric information and without Isaacs condition. Finally, we give a characterization for the value of the zero-sum differential games under some equivalent Isaacs condition in Section 5.

\section{ {\protect \large Preliminaries}}

Let $(\Omega, \mathcal{F}, P)$ be the canonical Wiener space, that is, $\Omega$ is the set of
continuous functions from $[0,T]$ to $\mathbb{R}^2$, $\mathcal{F}$ is the completed
$\sigma$-algebra on $\Omega$, $P$ is the Wiener measure. We define the canonical process
$B_t(\omega)=(B_t^1(\omega),B_t^2(\omega))=(\omega_1(t),\omega_2(t))$, $t\in[0,T]$, $\omega=(\omega_1,\omega_2)\in\Omega$.
 Then $B$ is a 2-dimensional Brownian motion on $(\Omega, \mathcal{F}, P)$ and $B^1$ is independent of $B^2$. We denote by
 $\{\mathcal{F}_{t,s}, s\geq t\}$ the filtration generated by the Brownian motion $B$, where
 $\mathcal{F}_{t,s}=\sigma\{B_r-B_t, r\in[t,s]\}\vee\mathcal{N}$, $\mathcal{N}$ is the set of null-set of $P$.

For any given partition $\pi=\{0=t_0<t_1<\cdots<t_N=T\}$ of the interval $[0,T]$,
 we define random variables $\zeta_{i,j}^\pi=\Phi_{0,1}(\frac{
{B_{t_j}^i-B_{t_{j-1}}^i}}{\sqrt{t_j-t_{j-1}}})$, $i=1,2,$ $j=1,2,\ldots,N$, where $\Phi_{0,1}(x)=\frac{1}{\sqrt{2\pi}}
\int_{-\infty}^x \exp\{-\frac{y^2}{2}\}dy$, $x\in\mathbb{R}$.
 Obviously, $\{\zeta_{i,j}^\pi\}_{
 1\leq j\leq N,\\
 }$ $i=1,2,$ is a family of independent random variables with uniform distribution on $[0,1]$.
Let $U$ and $V$ be the compact metric spaces which are the control state spaces used
by Player \Rmnum {1} and \Rmnum {2}, respectively. Let $\mathcal{P}(U)$ and $\mathcal{P}(V)$ be the space of
all probability measures over $U$ and $V$, respectively. From Skorohod's Representation Theorem, $\mathcal{P}(U)$ (resp.,
$\mathcal{P}(V)$) coincides with the set of the distributions of all $U$-valued (resp., $V$-valued) random variables.

Now we introduce the admissible controls for both players.

For any $t\in[0,T]$, the $U$-valued and Lebesgue measurable functions $(u_s)_{s\in[t,T]}$ form the set of
admissible controls for Player \Rmnum {1}, the $V$-valued and Lebesgue measurable functions
$(v_s)_{s\in[t,T]}$  that
 for Player \Rmnum {2}.
We denote by $\mathcal{U}_{t,T}$ the set of admissible controls $(u_s)_{s\in[t,T]}$ for Player \Rmnum {1} and by
$\mathcal{V}_{t,T}$ the set of admissible controls $(v_s)_{s\in[t,T]}$ for Player \Rmnum {2}.

For any given $t\in[0,T]$, $x\in\mathbb{R}^n$, we consider the following
ordinary differential equation
\begin{equation}\label{equ 2.1}
X_s=x+\int_t^sf(X_r,u_r,v_r)dr,\ s\in[t,T],
\end{equation}
where $u\in\mathcal{U}_{t,T}$ and $v\in\mathcal{V}_{t,T}$,
and the coefficient $f:\mathbb{R}^n\times
U\times V\mapsto \mathbb{R}^n$ is supposed to be bounded, continuous with respect to $(u,v)$ and Lipschitz continuous in $x$, uniformly with
respect to $u$ and $v$. Therefore, equation (\ref{equ 2.1}) has a unique solution and we denote it by $X^{t,x,u,v}$. From standard estimates we obtain that there exists a constant $C>0$ such that, for all $(t,x), (t',x')\in
[0,T]\times\mathbb{R}^n$, for all $s\in[t\vee t',T]$,
\begin{equation}\label{equ 2.1.1}
\begin{split}
&(1)\ |X_s^{t,x,u,v}-x|\leq C(s-t),\\
&(2)\ |X_s^{t,x,u,v}-X_s^{t',x',u,v}|\leq C(|t-t'|+|x-x'|).
\end{split}
\end{equation}

The cost functionals of the zero-sum differential games are defined by the $I\times J$ functionals
$g_{ij}(X_T^{t,x,u,v})$, $i=1,2,\ldots,I$, $j=1,2,\ldots,J$,
where the mappings $g_{ij}:\mathbb{R}^n\mapsto\mathbb{R}$ are Lipschitz continuous and bounded. Player \Rmnum{1} wants to minimize $g_{ij}(X_T^{t,x,u,v})$,
i.e., it is a cost functional for him/her, while
Player \Rmnum{2} wants to maximize $g_{ij}(X_T^{t,x,u,v})$ a payoff for him/her. The cost functionals of nonzero-sum differential
games are defined in Section 5.

The rules for our zero-sum differential game with asymmetric information are as follows:\\
$(1)$ At the beginning of the game, a pair $(i,j)$ is chosen randomly with the probability $(p,q)\in\Delta(I)\times\Delta(J)$,
where $\Delta(I)$ is the set of probabilities $p=(p_i)_{i=1,\ldots,I}$ on $\{1,\ldots,I\}$ and $\sum_{i=1}^Ip_i=1$; $\Delta(J)$ is defined similarly.
Both players know the probability $(p,q)$.\\
$(2)$ The choice of $i$ is only communicated with Player \Rmnum{1}, while the choice of $j$ is only communicated with Player \Rmnum{2}. But both players observe their opponent's controls.

Generally speaking, differential games with ``control against control" don't admit a dynamic programming principle and the value does, in general, not exist.
Thus, we study the game of the type ``nonanticipative strategy with delay against nonanticipative strategy with delay".
Considering the asymmetry of the information, the players want to hide a part of their private information. For this they randomize their
strategies, and the kind of randomization we choose is the key to obtain a value for our zero-sum game in a framework without Isaacs condition.

Let we consider an arbitrarily given partition $\pi=\{0=t_0<t_1<\ldots<t_N=T\}$ and assume $t\in[t_{k-1},t_k)$. We
give the definition of random non-anticipative strategies with delay for a game over the time interval $[t,T]$.
\begin{definition}\label{def 2.2}
A random  non-anticipative strategy with delay (NAD, for short) along
the partition $\pi$ for Player \Rmnum {1} is a mapping $\alpha: \Omega\times[t,T]\times\mathcal{V}_{t,T}\mapsto\mathcal{U}_{t,T}$ of the
form
$$\alpha(\omega,v)(s)=\alpha_k(\omega,\zeta_{1,k-1}^\pi(\omega),v)(s)I_{[t,t_k)}(s)+
\sum_{l=k+1}^N\alpha_l(\omega,(\zeta_{k-1}^\pi,\ldots,\zeta_{l-2}^\pi,\zeta_{1,l-1}^\pi)(\omega),v)(s)I_
{[t_{l-1},t_l)}(s),$$
where $\zeta_{l}^\pi=(\zeta_{1,l}^\pi,\zeta_{2,l}^\pi)$, $k-1\leq l\leq N-2$, and for $k\leq l\leq N$, the  $\alpha_l:$
 $\Omega\times\mathbb{R}^{2(l-k)+1}\times[t\vee t_{l-1},t_l]\times\mathcal{V}_{t,T}\mapsto\mathcal{U}_{t,T}$, are $\mathcal{F}_{0,t_{k-2}}\otimes\mathcal{B}(\mathbb{R}^{2(l-k)+1})\otimes\mathcal{B}([t\vee t_{l-1},t_l])\otimes\mathcal{B}(\mathcal{V}_{t,T})$-measurable
functions satisfying:
For all $v,v'\in\mathcal{V}_{t,T}$, it holds that, whenever $v=v'$ a.e. on $[t,t_{l-1}]$, we have for all $\omega\in\Omega$,
for all $x\in\mathbb{R}^{2(l-k)+1}$,
$\alpha_l(\omega,x,v)(s)=
\alpha_l(\omega,x,v')(s)$, a.e. on $[t\vee t_{l-1},t_l],\  k+1\leq l\leq N$.\\
\indent Similarly, a random NAD strategy along
the partition $\pi$ for Player \Rmnum {2} is a mapping $\beta: \Omega\times[t,T]\times\mathcal{U}_{t,T}\mapsto\mathcal{V}_{t,T}$ of the
form
$$\beta(\omega,u)(s)=\beta_k(\omega,\zeta_{2,k-1}^\pi(\omega),v)(s)I_{[t,t_k)}(s)+
\sum_{l=k+1}^N\beta_l(\omega,(\zeta_{k-1}^\pi,\ldots,\zeta_{l-2}^\pi,\zeta_{2,l-1}^\pi)(\omega),v)(s)I_
{[t_{l-1},t_l)}(s),$$
where $\zeta_{l}^\pi=(\zeta_{1,l}^\pi,\zeta_{2,l}^\pi)$, $k-1\leq l\leq N-2$, and for $k\leq l\leq N$, the  $\beta_l:$
 $\Omega\times\mathbb{R}^{2(l-k)+1}\times[t\vee t_{l-1},t_l]\times\mathcal{U}_{t,T}\mapsto\mathcal{V}_{t,T}$, are $\mathcal{F}_{0,t_{k-2}}\otimes\mathcal{B}(\mathbb{R}^{2(l-k)+1})\otimes\mathcal{B}([t\vee t_{l-1},t_l])\otimes\mathcal{B}(\mathcal{U}_{t,T})$-measurable
functions satisfying:
For all $u,u'\in\mathcal{U}_{t,T}$, it holds that, whenever $u=u'$ a.e. on $[t,t_{l-1}]$, we have for all $\omega\in\Omega$,
for all $x\in\mathbb{R}^{2(l-k)+1}$,
$\beta_l(\omega,x,u)(s)=
\beta_l(\omega,x,u')(s)$, a.e. on $[t\vee t_{l-1},t_l],\  k+1\leq l\leq N$.
\end{definition}

The set of all such random NAD strategies for Player \Rmnum {1} along the partition $\pi$ is denoted by $\mathcal{A}^{\pi}(t,T)$, and similarly
$\mathcal{B}^{\pi}(t,T)$ is that for Player \Rmnum {2}, $\mathcal{A}_0^\pi(t,T)$ and $\mathcal{B}_0^\pi(t,T)$ are the sets of pure (i.e. deterministic)
 strategies for player \Rmnum {1} and \Rmnum {2}. Then, we know for any partitions $\pi,\pi'$ of interval $[t,T]$ with $\pi\subset\pi'$,
it holds $\mathcal{A}^{\pi}(t,T)\subset\mathcal{A}^{\pi'}(t,T)$. Moreover we define
\begin{equation}\label{equ 2.2}
\mathcal{A}(t,T):=\bigcup_{\pi}\mathcal{A}^{\pi}(t,T),\ \  \mathcal{B}(t,T):=\bigcup_{\pi}\mathcal{B}^{\pi}(t,T).
\end{equation}
\begin{definition}\label{def 2.3}
We say that $\alpha\in\mathcal{A}_1^\pi(t,T)$, if the mapping $\alpha: \Omega\times[t,T]\times\mathcal{V}_{t,T}\mapsto\mathcal{U}_{t,T}$ has the
form
$$\alpha(\omega,v)(s)=\alpha_k(\zeta_{1,k-1}^\pi(\omega),v)(s)I_{[t,t_k)}(s)+
\sum_{l=k+1}^N\alpha_l((\zeta_{k-1}^\pi,\ldots,\zeta_{l-2}^\pi,\zeta_{1,l-1}^\pi)(\omega),v)(s)I_
{[t_{l-1},t_l)}(s),$$
where $\zeta_{l}^\pi=(\zeta_{1,l}^\pi,\zeta_{2,l}^\pi)$, $k-1\leq l\leq N-2$, and for $k\leq l\leq N$, the  $\alpha_l:$
 $\mathbb{R}^{2(l-k)+1}\times[t\vee t_{l-1},t_l]\times\mathcal{V}_{t,T}\mapsto\mathcal{U}_{t,T}$, are $\mathcal{B}(\mathbb{R}^{2(l-k)+1})\otimes\mathcal{B}([t\vee t_{l-1},t_l])\otimes\mathcal{B}(\mathcal{V}_{t,T})$-measurable
functions satisfying:
For all $v,v'\in\mathcal{V}_{t,T}$, it holds that, whenever $v=v'$ a.e. on $[t,t_{l-1}]$, we have
for all $x\in\mathbb{R}^{2(l-k)+1}$,
$\alpha_l(x,v)(s)=
\alpha_l(x,v')(s)$, a.e. on $[t\vee t_{l-1},t_l],\  k+1\leq l\leq N$.
Similarly, we have  $\beta\in\mathcal{B}_1^\pi(t,T)$.
\end{definition}
Obviously, from the Definition \ref{def 2.2} and \ref{def 2.3} we know $\mathcal{A}_0^\pi(t,T)\subset\mathcal{A}_1^\pi(t,T)
\subset\mathcal{A}^\pi(t,T)$, $\mathcal{B}_0^\pi(t,T)\subset\mathcal{B}_1^\pi(t,T)
\subset\mathcal{B}^\pi(t,T)$.

From the definition of a NAD strategy, we get the following lemma which is crucial throughout the paper. Such
 a result was established the first time by Buckdahn, Cardaliaguet and Rainer \cite{BCR2004}, Lemma 2.4.
\begin{lemma}\label{le 2.1}
For any $\alpha\in\mathcal{A}(t,T)$ and $\beta\in\mathcal{B}(t,T)$, there exists a unique measurable mapping $\Omega\ni\omega
\mapsto (u_\omega,v_\omega)\in\mathcal{U}_{t,T}\times\mathcal{V}_{t,T}$, such that, for all $\omega\in\Omega$,
$$\alpha(\omega,v_\omega)=u_\omega,\ \beta(\omega,u_\omega)=v_\omega,\ a.e.\ \text{on}\ [t,T]. $$
\end{lemma}
A proof of Lemma \ref{le 2.1} for a similar context can be found in \cite{BQRX2014}. However, since
our framework is slightly more general, for the reader's convenience we prefer to give it here.
\begin{proof}
For any $\alpha\in\mathcal{A}(t,T)$, from (\ref{equ 2.2}) we know there exist a partition $\pi_1$ of interval $[0,T]$,
such that $\alpha\in\mathcal{A}^{\pi_1}(t,T)$. Similarly, there exist a partition $\pi_2$ of
 interval $[0,T]$, such that $\beta\in\mathcal{B}^{\pi_2}(t,T)$. Define $\pi=\pi_1\cup\pi_2$ which
 combines $\pi_1$ and $\pi_2$, and notice that then $\alpha\in\mathcal{A}^\pi(t,T),$ and  $\beta\in\mathcal{B}^\pi(t,T)$.

 Indeed, if, for example, $\pi=\{0=t_0<t_1<\cdots<t_N=T\}$ and $t_{l-1},t_{l+1}\in\pi_1$, but $t_l\notin\pi_1$, then for
 $[t_{l-1},t_{l+1}]$ as $j$-th subinterval of the partition $\pi_1$,
 $\zeta_{i,j}^{\pi_1}=\Phi_{0,1}(\frac{B_{t_{l+1}}^i-B_{t_{l-1}}^i}{\sqrt{t_{l+1}-t_{l-1}}})=\Phi_{0,1}\Big(\frac{1}{\sqrt{t_{l+1}-t_{l-1}}}
 \big(\sqrt{t_{l}-t_{l-1}}\Phi_{0,1}^{-1}(\Phi_{0,1}(\frac{B_{t_{l}}^i-B_{t_{l-1}}^i}{\sqrt{t_{l}-t_{l-1}}}))+
 \sqrt{t_{l+1}-t_{l}}\Phi_{0,1}^{-1}(\Phi_{0,1}(\frac{B_{t_{l+1}}^i-B_{t_{l}}^i}{\sqrt{t_{l+1}-t_{l}}})\big)\Big),$
 i.e., $\zeta_{i,j}^{\pi_1}$ is a measurable function of $(\zeta^\pi_{i,l},\zeta^\pi_{i,l+1})$, $i=1,2.$ The above situation
 can be extended  into an obvious manner to the general case $\pi_1\subset\pi$ and  allows to show that $\mathcal{A}^{\pi_1}(t,T)\subset\mathcal{A}^\pi(t,T)$.

 Assume $\pi=\{0=t_0<t_1<\ldots<t_N=T\}$, and $t\in[t_{k-1},t_k)$, $0\leq k\leq N$. For each $\omega\in\Omega$, $\alpha(\omega,v)$ (respectively, $\beta(\omega,u)$)
 restricted to $[t,t_k]$ depends only on $v\in\mathcal{V}_{t,T}$ (respectively, $u\in\mathcal{U}_{t,T}$) restricted to $[t,t_{k-1}]$. Since $[t,t_{k-1}]$ is empty or a singleton, from the property of delay we know
$\alpha(\omega,v), \beta(\omega,u)$ restricted to $[t,t_{k}]$ do not depend on $v$ and $u$. Then we can define $u_\omega^1=\alpha(\omega,v^0)$, $v_\omega^1=\beta(\omega,u^0)$, for any $v^0\in\mathcal{V}_{t,T}$ and $u^0\in\mathcal{U}_{t,T}$, and
 the mapping $\Omega\ni\omega\mapsto (u_\omega^1,v_\omega^1)\in\mathcal{U}_{t,T}\times\mathcal{V}_{t,T}$ is measurable. Then we have
$$\alpha(\omega,v^1)=u^1,\ \beta(\omega,u^1)=v^1\ a.e.\ \text{on}\ [t,t_k].$$
Now we assume that for $j\geq 2$, $\alpha(\omega,v_\omega^{j-1})=u_\omega^{j-1},\ \beta(\omega,u_\omega^{j-1})=v_\omega^{j-1}$, a.e. on $[t,t_{j+k-2}]$,
and $\omega\mapsto(u_\omega^{j-1},v_\omega^{j-1})$ is measurable.\\
Then we define $u_\omega^j=\alpha(\omega,v_\omega^{j-1})$, $v_\omega^j=\beta(\omega,u_\omega^{j-1})$. Obviously, $u_\omega^j=u_\omega^{j-1},\ v_\omega^j=v_\omega^{j-1}$, a.e. on $[t,t_{j+k-2}]$. From the property of delay, we have $\alpha(\omega,v_\omega^j)=\alpha(\omega,v_\omega^{j-1})=u_\omega^j,\ \beta(\omega,u_\omega^j)=\beta(\omega,u_\omega^{j-1})=v_\omega^j$, a.e. on $[t,t_{j+k-1}]$, and $\omega\mapsto(u_\omega^{j},v_\omega^{j})$ is measurable.
Consequently, we get the existence of the measurable mapping $\Omega\ni\omega
\mapsto (u_\omega,v_\omega)\in\mathcal{U}_{t,T}\times\mathcal{V}_{t,T}$ satisfying this lemma and the uniqueness is obvious from the above construction.
\end{proof}
\begin{remark}\label{re 2.1}
This lemma implies that, for any partition $\pi$ of $[0,T]$:\\
For any $\alpha\in\mathcal{A}^\pi(t,T)$, $\beta\in\mathcal{B}^\pi(t,T)$, but also
 for any $\alpha\in\mathcal{A}(t,T)$, $\beta\in\mathcal{B}^\pi(t,T)$, and
 for any $\alpha\in\mathcal{A}^\pi(t,T)$, $\beta\in\mathcal{B}(t,T)$, there exists the
unique mapping $\Omega\ni\omega\mapsto(u_\omega,v_\omega)\in\mathcal{U}_{t,T}\times\mathcal{V}_{t,T}$, such that
 for all $\omega\in\Omega$,
$$\alpha(\omega,v_\omega)=u_\omega,\ \beta(\omega,u_\omega)=v_\omega,\ a.e.\ \text{on}\ [t,T]. $$
\end{remark}
\begin{remark}\label{re 2.2}
The control processes $u$ and $v$ along the partition $\pi$ satisfying Lemma \ref{le 2.1} have the following form:
$$
\left\{
\begin{array}{c}
u(\omega,s)=u^k(\omega,\zeta^\pi_{1,k-1},s)\cdot I_{[t,t_k)}(s)+\sum\limits_{l=k+1}^N u^l(\omega,\zeta^\pi_{k-1},\ldots,\zeta^\pi_{l-2},\zeta^\pi_{1,l-1},s)\cdot I_{[t_{l-1},t_l)}(s),\\
v(\omega,s)=v^k(\omega,\zeta^\pi_{2,k-1},s)\cdot I_{[t,t_k)}(s)+\sum\limits_{l=k+1}^N v^l(\omega,\zeta^\pi_{k-1},\ldots,\zeta^\pi_{l-2},\zeta^\pi_{2,l-1},s)\cdot I_{[t_{l-1},t_l)}(s),
\end{array}
\right.
$$
where $u^l,v^l$  are $\mathcal{F}_{0,t_{k-2}}\otimes\mathcal{B}(\mathbb{R}^{2(l-k)+1})\otimes\mathcal{B}([t\vee t_{l-1},t_l])$-measurable
functions, $k\leq l\leq N$. We denoted by $\mathcal{U}_{t,T}^{\pi}$ and $\mathcal{V}_{t,T}^{\pi}$ the set of the processes $u$ and $v$, respectively, which have the
above forms. The corresponding controls set constructed by $\mathcal{A}_1^{\pi}(t,T)$ and $\mathcal{B}_1^{\pi}(t,T)$ we denoted by $\mathcal{U}_{t,T}^{\pi,1}$ and $\mathcal{V}_{t,T}^{\pi,1}$. The only difference between $\mathcal{U}_{t,T}^{\pi}$ and $\mathcal{U}_{t,T}^{\pi,1}$ is that, if $u\in\mathcal{U}_{t,T}^{\pi,1}$, then $u_l$, $k\leq l\leq N$ is just $\mathcal{B}(\mathbb{R}^{2(l-k)+1})\otimes\mathcal{B}([t\vee t_{l-1},t_l])$-measurable.
\end{remark}
\begin{remark}\label{re 2.2.1}
We write $\hat{\alpha}\in(\mathcal{A}^\pi(t,T))^I$, if $\hat{\alpha}=(\alpha_1,\ldots,\alpha_I)$ and
$\alpha_i\in\mathcal{A}^\pi(t,T)$, $i=1,\ldots,I$, and  $\hat{\beta}\in(\mathcal{B}^\pi(t,T))^J$, if $\hat{\beta}=(\beta_1,\ldots,\beta_J)$ and
$\beta_j\in\mathcal{B}^\pi(t,T)$, $j=1,\ldots,J$. Similarly, we  have
$\hat{\alpha}\in(\mathcal{A}(t,T))^I$, $\hat{\beta}\in(\mathcal{B}(t,T))^J$.
\end{remark}
Let $(p,q)\in\Delta(I)\times\Delta(J),\ (t,x)\in[0,T]\times\mathbb{R}^n$, $\pi=\{0=t_0<t_1<\ldots<t_N=T\}$ and
$t\in[t_{k-1},t_k)$, we define the payoff functionals  $$J(t,x,\hat{\alpha},\hat{\beta},p,q)=\sum_{i=1}^{I}\sum_{j=1}^Jp_iq_jE[g_{ij}(X_T^{t,x,\alpha_i,\beta_j})].$$
Now we define the following upper value functions and lower value functions, respectively,
\begin{eqnarray}
W^\pi(t,x,p,q)&=&\inf_{\hat{\alpha}\in(\mathcal{A}^\pi(t,T))^I}\sup_{\hat{\beta}\in(\mathcal{B}^\pi(t,T))^J}
J(t,x,\hat{\alpha},\hat{\beta},p,q),\\
V^\pi(t,x,p,q)&=&\sup_{\hat{\beta}\in(\mathcal{B}^\pi(t,T))^J}\inf_{\hat{\alpha}\in(\mathcal{A}^\pi(t,T))^I}
J(t,x,\hat{\alpha},\hat{\beta},p,q),\\
W(t,x,p,q)&=&\inf_{\hat{\alpha}\in(\mathcal{A}(t,T))^I}\sup_{\hat{\beta}\in(\mathcal{B}(t,T))^J}
J(t,x,\hat{\alpha},\hat{\beta},p,q),\\
V(t,x,p,q)&=&\sup_{\hat{\beta}\in(\mathcal{B}(t,T))^J}\inf_{\hat{\alpha}\in(\mathcal{A}(t,T))^I}
J(t,x,\hat{\alpha},\hat{\beta},p,q).
\end{eqnarray}
\begin{definition}\label{re 2.2.2}
Let $\varepsilon>0$, we say that $\hat{\alpha}\in(\mathcal{A}^\pi(t,T))^I$ is an $\varepsilon$-optimal randomized strategy for
$W^\pi(t,x,p,q)$, if for all $(t,x,p,q)\in[0,T]\times\mathbb{R}^n\times\Delta(I)\times\Delta(J)$, it holds
\begin{equation}\label{equ 2.12}
|W^\pi(t,x,p,q)-\sup_{\hat{\beta}\in(\mathcal{B}^\pi(t,T))^J}
J(t,x,\hat{\alpha},\hat{\beta},p,q)|\leq \varepsilon.
\end{equation}
We say that $\hat{\beta}\in(\mathcal{B}^\pi(t,T))^J$ is an $\varepsilon$-optimal randomized strategy for
$V^\pi(t,x,p,q)$, if for all $(t,x,p,q)\in[0,T]\times\mathbb{R}^n\times\Delta(I)\times\Delta(J)$, it holds
\begin{equation}\label{equ 2.13}
|V^\pi(t,x,p,q)-\inf_{\hat{\alpha}\in(\mathcal{A}^\pi(t,T))^I}
J(t,x,\hat{\alpha},\hat{\beta},p,q)|\leq \varepsilon.
\end{equation}
Similarly, we define $\varepsilon$-optimal strategies for the other upper and lower value functions.
\end{definition}

\section{{\protect \large {The functions $W^\pi(t,x,p,q)$ and
 $V^\pi(t,x,p,q)$ without Isaacs condition}}}

In this section we mainly prove that when the mesh of the partition $\pi$ tends to 0,
the functions  $W^\pi$ and $V^\pi$ converge uniformly to the
same function which is the unique dual solution of some Hamilton-Jacobi-Isaacs (HJI, for short) equation.
For this, we introduce the following functions:
\begin{eqnarray}
W^\pi_1(t,x,p,q)&=&\inf_{\hat{\alpha}\in(\mathcal{A}_1^\pi(t,T))^I}\sup_{\hat{\beta}\in(\mathcal{B}_1^\pi(t,T))^J}
J(t,x,\hat{\alpha},\hat{\beta},p,q),\\
V^\pi_1(t,x,p,q)&=&\sup_{\hat{\beta}\in(\mathcal{B}_1^\pi(t,T))^J}\inf_{\hat{\alpha}\in(\mathcal{A}_1^\pi(t,T))^I}
J(t,x,\hat{\alpha},\hat{\beta},p,q).
\end{eqnarray}
\begin{theorem}\label{le 3.1.19}
For any $(t,x,p,q)\in[0,T]\times\mathbb{R}^n\times\Delta(I)\times\Delta(J)$,  it holds
$V^\pi(t,x,p,q)=V^\pi_1(t,x,p,q),$ $W^\pi(t,x,p,q)=W^\pi_1(t,x,p,q).$
\end{theorem}
We only give the proof for $V^\pi(t,x,p,q)=V^\pi_1(t,x,p,q)$, the proof for $W^\pi(t,x,p,q)=W^\pi_1(t,x,p,q)$ is similar.
In order to show that,  we need the following auxiliary lower value function:
\begin{equation}\label{equ 3.1.20}
\tilde{V}^\pi(t,x,p,q)=\esssup\limits_{\hat{\beta}\in(\mathcal{B}^\pi(t,T))^J}\essinf\limits_{\hat{\alpha}\in(\mathcal{A}^\pi(t,T))^I}
\sum_{i=1}^{I}\sum_{j=1}^Jp_iq_jE[g_{ij}(X_T^{t,x,\alpha_i,\beta_j})|\mathcal{F}_{0,t_{k-2}}].
\end{equation}
\begin{lemma}\label{le 3.1.20.1}
For all $(t,x,p,q)\in[0,T]\times\mathbb{R}^n\times\Delta(I)\times\Delta(J)$, the function $\tilde{V}^\pi(t,x,p,q)$ is deterministic, i.e., independent of $\mathcal{F}_{0,t_{k-2}}$, then we have $\tilde{V}^\pi(t,x,p,q)=E[\tilde{V}^\pi(t,x,p,q)]$, $P$-a.s.
\end{lemma}
\begin{proof}
For $\Omega=C([0,T];\mathbb{R}^2)$,
we assume\\
$H=\{h\in\Omega: \exists\text{ Radon-Nikodym\
derivative}\ \dot{h}\in L^2([0,T];\mathbb{R}^2), h(s)=h(s\wedge t_{k-2}), s\in[0,T]\},$ then we know  $H$ is the Cameron-Martin space.
For any $h\in H$, we define the mapping $\tau_h:\Omega\mapsto\Omega$ by $\tau_h(\omega):=\omega+h$, $\omega\in\Omega$. Then, we know
$\tau_h$ is a bijection and $\tau_h^{-1}=\tau_{-h}$. \\
For any $\alpha\in\mathcal{A}^{\pi}(t,T)$, we know $\alpha$ has the form of
$$\alpha(\omega,v)(s)=\alpha_k(\omega,\zeta_{1,k-1}^\pi(\omega),v)(s)I_{[t,t_k)}(s)+
\sum_{l=k+1}^N\alpha_l(\omega,(\zeta_{k-1}^\pi,\ldots,\zeta_{l-2}^\pi,\zeta_{1,l-1}^\pi)(\omega),v)(s)I_
{[t_{l-1},t_l)}(s).$$
Then, for any $h\in H$, we define
\begin{eqnarray*}
\begin{split}
&\alpha^h(\omega,v)(s)\\
&:=\alpha_k(\tau_h(\omega),\zeta_{1,k-1}^\pi(\omega),v)(s)I_{[t,t_k)}(s)+
\sum_{l=k+1}^N\alpha_l(\tau_h(\omega),(\zeta_{k-1}^\pi,\ldots,\zeta_{l-2}^\pi,\zeta_{1,l-1}^\pi)(\omega),v)(s)I_
{[t_{l-1},t_l)}(s).
\end{split}
\end{eqnarray*}
Obviously, we know $\alpha^h\in\mathcal{A}^{\pi}(t,T)$, and the mapping $\alpha\mapsto\alpha^h$ is a bijection on $\mathcal{A}^{\pi}(t,T)$.
For any $h\in H$, $\beta\in\mathcal{B}^\pi(t,T)$,  $\beta^h$ is similarly defined and
$\beta\mapsto\beta^h$ is a bijection on $\mathcal{B}^{\pi}(t,T)$. Then we get
\begin{equation}\label{equ 3.1.20.1}
E[g_{ij}(X_T^{t,x,\alpha_i,\beta_j})|\mathcal{F}_{0,t_{k-2}}]\circ\tau_h=E[g_{ij}(X_T^{t,x,\alpha^h_i,\beta^h_j})|\mathcal{F}_{0,t_{k-2}}],\ P\text{-a.s.}
\end{equation}
We now define $I(t,x,p,q,\hat{\beta}):=\essinf\limits_{\hat{\alpha}\in(\mathcal{A}^\pi(t,T))^I}
\sum\limits_{i=1}^{I}\sum\limits_{j=1}^Jp_iq_jE[g_{ij}(X_T^{t,x,\alpha_i,\beta_j})|\mathcal{F}_{0,t_{k-2}}]$,  $\hat{\beta}\in(\mathcal{B}^\pi(t,T))^J$.
Since $I(t,x,p,q,\hat{\beta})\leq
\sum\limits_{i=1}^{I}\sum\limits_{j=1}^Jp_iq_jE[g_{ij}(X_T^{t,x,\alpha_i,\beta_j})|\mathcal{F}_{0,t_{k-2}}]$, $P$-a.s., from (\ref{equ 3.1.20.1}) we get
\begin{equation}\label{equ 3.1.20.2}
I(t,x,p,q,\hat{\beta})\circ\tau_h\leq
\sum\limits_{i=1}^{I}\sum\limits_{j=1}^Jp_iq_jE[g_{ij}(X_T^{t,x,\alpha_i^h,\beta_j^h})|\mathcal{F}_{0,t_{k-2}}],\ P\text{-a.s.}
\end{equation}
On the other hand, for any random variable $\xi$, such that $\xi\leq \sum\limits_{i=1}^{I}\sum\limits_{j=1}^Jp_iq_jE[g_{ij}(X_T^{t,x,\alpha_i^h,\beta_j^h})|\mathcal{F}_{0,t_{k-2}}]$, $P$-a.s., we have that
$\xi\circ\tau_{-h}\leq \sum\limits_{i=1}^{I}\sum\limits_{j=1}^Jp_iq_jE[g_{ij}(X_T^{t,x,\alpha_i,\beta_j})|\mathcal{F}_{0,t_{k-2}}],$ $P$-a.s., for all
$\hat{\alpha}\in\mathcal{A}^\pi(t,T)$, then we know $\xi\circ\tau_{-h}\leq I(t,x,p,q,\hat{\beta})$, $P$-a.s., which means that $\xi\leq I(t,x,p,q,\hat{\beta})\circ\tau_h$. Thus we have
\begin{equation}\label{equ 3.1.20.3}
\begin{split}
I(t,x,p,q,\hat{\beta})\circ\tau_h=&\essinf\limits_{\hat{\alpha}\in(\mathcal{A}^\pi(t,T))^I}
\sum\limits_{i=1}^{I}\sum\limits_{j=1}^Jp_iq_jE[g_{ij}(X_T^{t,x,\alpha_i^h,\beta_j^h})|\mathcal{F}_{0,t_{k-2}}],\ P\text{-a.s.}
\end{split}
\end{equation}
Using the similar method, we obtain
\begin{equation}\label{equ 3.1.20.4}
\Big(\esssup\limits_{\hat{\beta}\in(\mathcal{B}^\pi(t,T))^J}I(t,x,p,q,\hat{\beta})\Big)\circ\tau_h=
\esssup\limits_{\hat{\beta}\in(\mathcal{B}^\pi(t,T))^J}\Big(I(t,x,p,q,\hat{\beta})\circ\tau_h\Big),\ P\text{-a.s.}
\end{equation}
Therefore, for all $h\in H$, from (\ref{equ 3.1.20.4}) and (\ref{equ 3.1.20.3})  we get, $P$-a.s.,
\begin{equation}\label{equ 3.1.20.5}
\begin{split}
&\tilde{V}^\pi(t,x,p,q)\circ\tau_h=\Big(\esssup\limits_{\hat{\beta}\in(\mathcal{B}^\pi(t,T))^J}I(t,x,p,q,\hat{\beta})\Big)\circ\tau_h\\
&=\esssup\limits_{\hat{\beta}\in(\mathcal{B}^\pi(t,T))^J}\essinf\limits_{\hat{\alpha}\in(\mathcal{A}^\pi(t,T))^I}
\sum\limits_{i=1}^{I}\sum\limits_{j=1}^Jp_iq_jE[g_{ij}(X_T^{t,x,\alpha_i^h,\beta_j^h})|\mathcal{F}_{0,t_{k-2}}]\\
&=\esssup\limits_{\hat{\beta}\in(\mathcal{B}^\pi(t,T))^J}\essinf\limits_{\hat{\alpha}\in(\mathcal{A}^\pi(t,T))^I}
\sum\limits_{i=1}^{I}\sum\limits_{j=1}^Jp_iq_jE[g_{ij}(X_T^{t,x,\alpha_i,\beta_j})|\mathcal{F}_{0,t_{k-2}}]
=\tilde{V}^\pi(t,x,p,q).
\end{split}
\end{equation}
Then combined with Lemma 4.1 in \cite{BL2008}, we obtain our desired results.
\end{proof}
Now we give the proof of Theorem \ref{le 3.1.19}.
\begin{proof}
Step 1:  We prove $\tilde{V}^\pi(t,x,p,q)=V_1^\pi(t,x,p,q)$, for all $(t,x,p,q)\in[0,T]\times\mathbb{R}^n\times\Delta(I)\times\Delta(J)$.

For any $\hat{\beta}\in(\mathcal{B}_1^\pi(t,T))^J$ (independent of $\mathcal{F}_{t_{k-2}}$), we have
$$\tilde{V}^\pi(t,x,p,q)\geq \essinf\limits_{\hat{\alpha}\in(\mathcal{A}^\pi(t,T))^I}
\sum\limits_{i=1}^{I}\sum\limits_{j=1}^Jp_iq_jE[g_{ij}(X_T^{t,x,\alpha_i,\beta_j})|\mathcal{F}_{0,t_{k-2}}],\ P\text{-a.s.}$$
For any $\varepsilon>0$, there exists $\hat{\alpha}\in(\mathcal{A}^\pi(t,T))^I$ (depending on $\varepsilon$, $\hat{\beta}$), such that
\begin{equation}\label{equ 3.1.20.6}
\tilde{V}^\pi(t,x,p,q)\geq
\sum\limits_{i=1}^{I}\sum\limits_{j=1}^Jp_iq_jE[g_{ij}(X_T^{t,x,\alpha_i,\beta_j})|\mathcal{F}_{0,t_{k-2}}]-\varepsilon,\ P\text{-a.s.}
\end{equation}
From Lemma \ref{le 3.1.20.1} and (\ref{equ 3.1.20.6}), we have
\begin{equation}\label{equ 3.1.20.7}
\begin{split}
&\tilde{V}^\pi(t,x,p,q)=E[\tilde{V}^\pi(t,x,p,q)]\geq
\sum\limits_{i=1}^{I}\sum\limits_{j=1}^Jp_iq_jE[g_{ij}(X_T^{t,x,\alpha_i,\beta_j})]-\varepsilon\\
&\geq \inf_{\hat{\alpha}\in(\mathcal{A}^\pi_1(t,T))^I}
\sum\limits_{i=1}^{I}\sum\limits_{j=1}^Jp_iq_jE[g_{ij}(X_T^{t,x,\alpha_i,\beta_j})]-\varepsilon.
\end{split}
\end{equation}
Since (\ref{equ 3.1.20.7}) holds for any $\hat{\beta}\in(\mathcal{B}_1^\pi(t,T))^J$, we get
\begin{equation}\label{equ 3.1.20.8}
\tilde{V}^\pi(t,x,p,q)\geq \sup_{\hat{\beta}\in(\mathcal{B}^\pi_1(t,T))^J}\inf_{\hat{\alpha}\in(\mathcal{A}^\pi_1(t,T))^I}
\sum\limits_{i=1}^{I}\sum\limits_{j=1}^Jp_iq_jE[g_{ij}(X_T^{t,x,\alpha_i,\beta_j})]-\varepsilon=V_1^\pi(t,x,p,q)-\varepsilon.
\end{equation}
From the arbitrariness of $\varepsilon$, we obtain $\tilde{V}^\pi(t,x,p,q)\geq V_1^\pi(t,x,p,q)$.\\
On the other hand, for any $\varepsilon>0$, there exists $\hat{\beta}\in(\mathcal{B}^\pi(t,T))^J$, such that, $P$-a.s.,
\begin{equation}\label{equ 3.1.20.9}
\begin{split}
\tilde{V}^\pi(t,x,p,q)&\leq \essinf\limits_{\hat{\alpha}\in(\mathcal{A}^\pi(t,T))^I}
\sum\limits_{i=1}^{I}\sum\limits_{j=1}^Jp_iq_jE[g_{ij}(X_T^{t,x,\alpha_i,\beta_j})|\mathcal{F}_{0,t_{k-2}}]+\varepsilon\\
&\leq \essinf\limits_{\hat{\alpha}\in(\mathcal{A}_1^\pi(t,T))^I}
\sum\limits_{i=1}^{I}\sum\limits_{j=1}^Jp_iq_jE[g_{ij}(X_T^{t,x,\alpha_i,\beta_j})|\mathcal{F}_{0,t_{k-2}}]+\varepsilon.
\end{split}
\end{equation}
Notice that $E[g_{ij}(X_T^{t,x,\alpha_i,\beta_j})|\mathcal{F}_{0,t_{k-2}}](\omega)=E[g_{ij}(X_T^{t,x,\alpha_i,\beta_j^{\bar{\omega}}})]$,
$P(d\omega)$-a.s., where $\bar{\omega}(s)=\omega(s\wedge t_{k_2})$, $s\in[0,T]$. Thus, from (\ref{equ 3.1.20.9}) we have
\begin{equation}\label{equ 3.1.20.10}
\begin{split}
&\tilde{V}^\pi(t,x,p,q)\leq \essinf\limits_{\hat{\alpha}\in(\mathcal{A}_1^\pi(t,T))^I}
\sum\limits_{i=1}^{I}\sum\limits_{j=1}^Jp_iq_jE[g_{ij}(X_T^{t,x,\alpha_i,\beta_j^{\bar{\omega}}})]+\varepsilon\\
&\leq \esssup\limits_{\hat{\beta}\in(\mathcal{B}_1^\pi(t,T))^J}\essinf\limits_{\hat{\alpha}\in(\mathcal{A}_1^\pi(t,T))^I}
\sum\limits_{i=1}^{I}\sum\limits_{j=1}^Jp_iq_jE[g_{ij}(X_T^{t,x,\alpha_i,\beta_j^{\bar{\omega}}})]+\varepsilon
=V_1^\pi(t,x,p,q)+\varepsilon.
\end{split}
\end{equation}
From the arbitrariness of $\varepsilon$, we obtain $\tilde{V}^\pi(t,x,p,q)\leq V_1^\pi(t,x,p,q)$.\\
Step 2:  We prove $\tilde{V}^\pi(t,x,p,q)=V^\pi(t,x,p,q)$, for all $(t,x,p,q)\in[0,T]\times\mathbb{R}^n\times\Delta(I)\times\Delta(J)$.\\
For any $\varepsilon>0$, there exists $\hat{\alpha}\in(\mathcal{A}^\pi(t,T))^I$, such that, $P$-a.s.,
\begin{equation}\label{equ 3.1.20.11}
\begin{split}
&\tilde{V}^\pi(t,x,p,q)\geq \essinf\limits_{\hat{\alpha}\in(\mathcal{A}^\pi(t,T))^I}
\sum\limits_{i=1}^{I}\sum\limits_{j=1}^Jp_iq_jE[g_{ij}(X_T^{t,x,\alpha_i,\beta_j})|\mathcal{F}_{0,t_{k-2}}]\\
&\geq
\sum\limits_{i=1}^{I}\sum\limits_{j=1}^Jp_iq_jE[g_{ij}(X_T^{t,x,\alpha_i,\beta_j})|\mathcal{F}_{0,t_{k-2}}]-\varepsilon.
\end{split}
\end{equation}
From Lemma \ref{le 3.1.20.1} and (\ref{equ 3.1.20.11}), we have
\begin{equation}\label{equ 3.1.20.12}
\begin{split}
&\tilde{V}^\pi(t,x,p,q)=E[\tilde{V}^\pi(t,x,p,q)]\geq
\sum\limits_{i=1}^{I}\sum\limits_{j=1}^Jp_iq_jE[g_{ij}(X_T^{t,x,\alpha_i,\beta_j})]-\varepsilon\\
&\geq \inf\limits_{\hat{\alpha}\in(\mathcal{A}^\pi(t,T))^I}\sum\limits_{i=1}^{I}\sum\limits_{j=1}^Jp_iq_jE[g_{ij}(X_T^{t,x,\alpha_i,\beta_j})]-\varepsilon.
\end{split}
\end{equation}
Thanks to (\ref{equ 3.1.20.12}) holds for any $\hat{\beta}\in(\mathcal{B}^\pi(t,T))^J$ and from the arbitrariness of $\varepsilon$, we have $\tilde{V}^\pi(t,x,p,q)\geq V^\pi(t,x,p,q)$.
On the other hand, for any $\varepsilon>0$, there exists $\hat{\beta}\in(\mathcal{B}^\pi(t,T))^J$, such that, $P$-a.s.,
\begin{equation}\label{equ 3.1.20.13}
\begin{split}
&\tilde{V}^\pi(t,x,p,q)\leq \essinf\limits_{\hat{\alpha}\in(\mathcal{A}^\pi(t,T))^I}
\sum\limits_{i=1}^{I}\sum\limits_{j=1}^Jp_iq_jE[g_{ij}(X_T^{t,x,\alpha_i,\beta_j})|\mathcal{F}_{0,t_{k-2}}]+\varepsilon\\
&\leq
\sum\limits_{i=1}^{I}\sum\limits_{j=1}^Jp_iq_jE[g_{ij}(X_T^{t,x,\alpha_i,\beta_j})|\mathcal{F}_{0,t_{k-2}}]+\varepsilon.
\end{split}
\end{equation}
From Lemma \ref{le 3.1.20.1}, thanks to (\ref{equ 3.1.20.13}) holds for every $\hat{\alpha}\in(\mathcal{A}^\pi(t,T))^I$, we have
\begin{equation}\label{equ 3.1.20.14}
\begin{split}
&\tilde{V}^\pi(t,x,p,q)=E[\tilde{V}^\pi(t,x,p,q)]\leq
\sum\limits_{i=1}^{I}\sum\limits_{j=1}^Jp_iq_jE[g_{ij}(X_T^{t,x,\alpha_i,\beta_j})]+\varepsilon\\
&\leq \inf\limits_{\hat{\alpha}\in(\mathcal{A}^\pi(t,T))^I}\sum\limits_{i=1}^{I}\sum\limits_{j=1}^Jp_iq_jE[g_{ij}(X_T^{t,x,\alpha_i,\beta_j})]+\varepsilon\\
&\leq\sup\limits_{\hat{\beta}\in(\mathcal{B}^\pi(t,T))^J}\inf\limits_{\hat{\alpha}\in(\mathcal{A}^\pi(t,T))^I}
\sum\limits_{i=1}^{I}\sum\limits_{j=1}^Jp_iq_jE[g_{ij}(X_T^{t,x,\alpha_i,\beta_j})]+\varepsilon
=V^\pi(t,x,p,q)+\varepsilon.
\end{split}
\end{equation}
Thus, we obtain  $\tilde{V}^\pi(t,x,p,q)\leq V^\pi(t,x,p,q)$. Finally, from Step 1 and Step 2, we have $V^\pi(t,x,p,q)=\tilde{V}^\pi(t,x,p,q)= V_1^\pi(t,x,p,q)$.
\end{proof}
We now prove that when the mesh of the partition $\pi$ tends to 0,
the functions  $W_1^\pi$ and $V_1^\pi$ converge uniformly to the
same function which is the unique dual solution of some HJI equation.
\begin{lemma}\label{le 3.1}
The functions $W_1^\pi$ and $V_1^\pi$ are Lipschitz continuous with respect to $(t,x,p,q)$, uniformly with respect to $\pi$.
\end{lemma}
\begin{proof}
We just give the proof for $V_1^\pi$, the proof of $W_1^\pi$ is similar.

Since the cost functionals $g_{ij}$ are bounded, from the definition of $V_1^\pi$, we obviously have that
 $V_1^\pi$ is Lipschitz with respect to $p$ and $q$.
 For any $t\in[0,T]$, $(u,v)\in\mathcal{U}_{t,T}\times\mathcal{V}_{t,T}$, from (\ref{equ 2.1.1}), the functional
$g_{ij}(X_T^{t,x,u,v})$ is Lipschitz continuous with respect to $x$, then for any $(\hat{\alpha},\hat{\beta})\in(\mathcal{A}_1^\pi(t,T))^I\times(\mathcal{B}_1^\pi(t,T))^J$,
we have that $J(t,x,\hat{\alpha},\hat{\beta},p,q)$
is Lipschitz continuous with respect to $x$. Moreover, the Lipschitz constant only depends on the Lipschitz constants of $g_{ij}$
and the bound of $f$. Thus we
have $V_1^{\pi}$ is Lipschitz with respect to $x$.

Now we only need to show $V_1^\pi$ is Lipschitz with respect to $t$. Let $x\in\mathbb{R}^n$, $(p,q)\in\Delta(I)\times\Delta(J)$, and $t<t'<T$ be arbitrarily fixed.
Let $\hat{\beta}=(\beta_j)_{j=1,2,\ldots,J}\in(\mathcal{B}_1^\pi(t,T))^J$ be an $\varepsilon$-optimal strategy for $V_1^\pi(t,x,p,q)$. We
define a strategy $\beta'_j\in\mathcal{B}_1^\pi(t',T)$ associated with $\beta_j$.
For this end, we put for all $u\in\mathcal{U}_{t',T}$,
 $$\tilde{\beta_j}(\omega,u)=\beta_j(\omega,\tilde{u}),\ \text{where}\ \tilde{u}(s)=
\left\{
\begin{array}{ll}
\bar{u},&s\in[t,t'),\\
u(s),&s\in[t',T],
\end{array}
\right.
$$
and $\bar{u}\in U$ is an arbitrarily given constant control.

If $t'<t_k$, then $\tilde{\beta}_j\in\mathcal{B}_1^\pi(t',T)$ and we define $\beta'_j=\tilde{\beta}_j.$
Otherwise, we let $l\geq k+1$ be such that $t_{l-1}\leq t'< t_{l}$. We now consider $2(l-k)+1$ random variables
$\eta_{k-1}^i,\ldots,\eta_{l-2}^i,\eta_{l-1}^1$, $i=1,2$, defined on $([0,1],\mathcal{B}([0,1]),dx)$ with $\eta_{l-1}^1(x)=x,\ x\in[0,1]$, which are mutually independent, independent
of $\zeta_{i,j}^\pi$, $(i,j)\neq(2,l-1)$, and uniformly
distributed on $[0,1]$. Then the composed random variables $\eta_{k-1}^1\circ\zeta_{2,l-1}^\pi,
\eta_{k-1}^2\circ\zeta_{2,l-1}^\pi,\ldots
,\eta_{l-1}^1\circ\zeta_{2,l-1}^\pi,$ are mutually independent,
independent of all $\zeta_{i,j}^\pi,\ (i,j)\neq(2,l-1)$,
uniformly distributed random variables.\\
For any $u\in\mathcal{U}_{t',T}$, $s\in[t',T]$, we define
 \begin{eqnarray*}
&&\beta'_j(\omega,u)(s)\\
&&=\sum_{m=l}^N\tilde{\beta}_{j,m}\big((\eta_{k-1}^1\circ\zeta_{2,l-1}^\pi,\eta_{k-1}^2\circ\zeta_{2,l-1}^\pi,\ldots,
\eta_{l-1}^1\circ\zeta_{2,l-1}^\pi,\zeta_{1,l-1}^\pi,\zeta_{l}^\pi,\ldots,
\zeta_{m-2}^\pi,\zeta_{2,m-1}^\pi)
(\omega),u\big)(s)\cdot\\
&&\ \ \ I_{[t'\vee t_{m-1},t_m)}(s).
\end{eqnarray*}
where $\tilde{\beta}_{j,m}(\omega,u)(s)=\tilde{\beta}_{j}(\omega,u)(s) I_{[t\vee t_{m-1},t_m)}(s).$
Then we have $\beta'_j\in\mathcal{B}_1^\pi(t',T)$. Notice that for all $u\in\mathcal{U}_{t',T}$,
$\beta'_j(u)$ and $\tilde{\beta}_j(u)$ obey the same law knowing $\zeta_{l-1}^\pi,\ldots,\zeta_{N-1}^\pi$. Therefore,
$E[g_{ij}(X_T^{t',x,u,\beta'_j(u)})]=E[g_{ij}(X_T^{t',x,u,\tilde{\beta}_j(u)})].$
Then for all $\hat{\alpha}\in(\mathcal{A}_1^\pi(t',T))^I$,
\begin{equation}\label{equ 3.1}
J(t',x,\hat{\alpha},(\beta'_j),p,q)=J(t',x,\hat{\alpha},(\tilde{\beta}_j),p,q).
\end{equation}
Now for any $\alpha\in\mathcal{A}_1^\pi(t',T)$,
we define a strategy $\alpha'\in\mathcal{A}_1^\pi(t,T)$ associated with $\alpha$ as follows, for all $v\in\mathcal{V}_{t,T}$,
$$\alpha'(\omega,v)(s)=
\left\{
\begin{array}{ll}
\bar{u}(s),&s\in[t,t'),\\
\alpha(\omega,v|_{[t',T]})(s),&s\in[t',T].
\end{array}
\right.$$
Through the above construction and from Lemma \ref{le 2.1}, the couples of admissible controls related to the couples of strategies
$(\alpha',\beta_j)$ and $(\alpha,\tilde{\beta}_j)$ coincide on the interval $[t',T]$.
Hence, using the standard estimate and Gronwall inequality we have
\begin{equation}\label{equ 3.2}
E[|X_s^{t,x,\alpha',\beta_j}-X_s^{t',x,\alpha,\tilde{\beta}_j}|]\leq M|t'-t|,\ s\in[t',T],
\end{equation}
where the constant $M$ only depends on the bound of $f$ as well as the Lipschitz constant of $f$.
Thus, for any $\hat{\alpha}\in(\mathcal{A}_1^\pi(t',T))^I$, from (\ref{equ 3.1}), (\ref{equ 3.2}) and (\ref{equ 2.13}), we have
\begin{eqnarray*}
&&J(t',x,\hat{\alpha},(\beta'_j),p,q)=J(t',x,\hat{\alpha},(\tilde{\beta}_j),p,q)
\geq J(t,x,\hat{\alpha}',\hat{\beta},p,q)-C|t'-t| \\
&&\geq\inf_{\hat{\alpha}^{''}\in(\mathcal{A}_1^\pi(t,T))^I}J(t,x,\hat{\alpha}^{''},\hat{\beta},p,q)-C|t'-t|
\geq V_1^\pi(t,x,p,q)-\varepsilon-C|t'-t|,
\end{eqnarray*}
Therefore,
\begin{equation}\label{equ 3.3}
V_1^\pi(t',x,p,q)\geq V_1^\pi(t,x,p,q)-\varepsilon-C|t'-t|.
\end{equation}
Similarly, if we assume that $\hat{\beta}\in(\mathcal{B}_1^\pi(t',T))^J$  is $\varepsilon$-optimal for $V_1^\pi(t',x,p,q)$, then we can get
\begin{equation}\label{equ 3.4}
V_1^\pi(t,x,p,q)\geq V_1^\pi(t',x,p,q)-\varepsilon-C|t'-t|.
\end{equation}
Moreover, from the arbitrariness of $\varepsilon>0$, we obtain $V_1^\pi$ is Lipschitz continuous in $t$.
\end{proof}
\begin{lemma}\label{le 3.2}
For any $(t,x)\in[0,T]\times\mathbb{R}^n$, the functions $W_1^\pi(t,x,p,q)$ and $V_1^\pi(t,x,p,q)$ both are convex in $p$ and concave in $q$
on $\Delta(I)$ and $\Delta(J)$.
\end{lemma}

\begin{proof}
We just give the proof for $V_1^\pi$, the proof of $W_1^\pi$ is similar.\\
It is obvious that
\begin{equation}\label{equ 3.5}
V_1^\pi(t,x,p,q)=\sup_{(\beta_j)\in(\mathcal{B}_1^\pi(t,T))^J}\sum^I_{i=1}p_i
\inf_{\alpha\in\mathcal{A}_1^\pi(t,T)}\sum^J_{j=1}q_jE[g_{ij}(X_T^{t,x,\alpha,\beta_j})].
\end{equation}
then we know $V_1^\pi(t,x,p,q)$ is convex in $p$.

Now we prove that $V_1^\pi(t,x,p,q)$ is concave in $q$. Let $(t,x,p)\in[0,T]\times\mathbb{R}^n\times\Delta(I)$, $q^0,q^1\in\Delta(J), \lambda\in(0,1)$, and let $\hat{\beta}^0=(\beta^0_j)_{j=1,...,J}\in(\mathcal{B}_1^\pi(t,T))^J$ and $\hat{\beta}^1=(\beta^1_j)_{j=1,...,J}\in(\mathcal{B}_1^\pi(t,T))^J$ be $\varepsilon$-optimal for $V_1^\pi(t,x,p,q^0)$ and $V_1^\pi(t,x,p,q^1)$, respectively. For $q^0=(q^0_1,\ldots,q^0_J)$ and $q^1=(q^1_1,\ldots,q^1_J)$, we define $q^\lambda_j=(1-\lambda)q^0_j+\lambda q^1_j$ and  $q^\lambda=(q^\lambda_1,\ldots,q^\lambda_J)\in\Delta(J)$. Without loss of generality, we assume $q^\lambda_j>0$, $j=1,\ldots,J$, then we define $c_j=\frac{(1-\lambda)q^0_j}{q^\lambda_j}$, $j=1,\ldots,J$. For $\omega\in \Omega$, $u\in\mathcal{U}_{t,T}$, $s\in[t,T)$, we define the strategy $\hat{\beta}^\lambda=(\beta_j^\lambda)_{j=1,\ldots,J}$ and
$\beta^\lambda_j(y_1,\ldots,y_{2(N-k)+1},u)(s)=\\
\beta_{j}^0(y_1,y_2,\ldots,y_{2(N-k)},\frac{1}{c_j}y_{2(N-k)+1},u)(s)
+\beta_{j}^1(y_1,y_2,\ldots,y_{2(N-k)},\frac{1}{1-c_j}(y_{2(N-k)+1}-c_j),u)(s),$\\
where
 $\beta^i_j((\zeta_{k-1}^\pi,\ldots,\zeta_{N-2}^\pi,\zeta_{2,N-1}^\pi)(\omega),u)(s)
=\sum\limits_{l=k}^N\beta^i_{lj}((\zeta_{k-1}^\pi,\ldots,\zeta_{l-2}^\pi,\zeta_{2,l-1}^\pi)(\omega),u)I_{[t\vee t_{l-1},t_l]}(s)$, $i=0,1$, respectively.
Then we have $(\beta^\lambda_j)\in(\mathcal{B}_1^\pi(t,T))^J$. Therefore, we have
\begin{eqnarray*}
&& \inf_{\alpha\in(\mathcal{A}_1^\pi(t,T))^I}J(t,x,\hat{\alpha},\hat{\beta}^\lambda,p,q^\lambda)=\sum_{i=1}^Ip_i
\inf_{\alpha\in\mathcal{A}_1^\pi(t,T)}\sum_{j=1}^J
q_j^\lambda E[g_{ij}(X_T^{t,x,\alpha,\beta^\lambda_j})]\\ &=&\sum^{I}_{i=1}p_i\inf_{\alpha\in\mathcal{A}_1^\pi(t,T)}\sum_{j=1}^Jq^\lambda_j\Big(\int_{[0,c_j]}E[g_{ij}(X_T^{t,x,\alpha,\beta^0_j
((\zeta_{2,k-1}^\pi,\zeta_{1,k-1}^\pi,\zeta_{k}^\pi,...,\zeta_{N-2}^\pi,\frac{1}{c_j}y_{2(N-k)+1})(\omega),\cdot)})]dy_{2(N-k)+1}\\
&&+\int_{[c_j,1]}E[g_{ij}(X_T^{t,x,\alpha,\beta^1_j((\zeta_{2,k-1}^\pi,\zeta_{1,k-1}^\pi,\zeta_{k}^\pi,...,\zeta_{N-2}^\pi,\frac{1}{1-c_j}(y_{2(N-k)+1}-c_j))
(\omega),\cdot)})]dy_{2(N-k)+1}\Big)\\
&=& \sum_{i=1}^{I}p_i\inf_{\alpha\in\mathcal{A}_1^\pi(t,T)}\sum_{j=1}^Jq^\lambda_j\Big[\frac{(1-\lambda)q_j^0}{q^\lambda_j}E[g_{ij}
(X_T^{t,x,\alpha,\beta^0_j})]+
\frac{\lambda q^1_j}{q^\lambda_j}E[g_{ij}(X_T^{t,x,\alpha,\beta^1_j})]\Big]\\
&\geq&(1-\lambda)\sum_{i=1}^{I}p_i\inf_{\alpha\in\mathcal{A}_1^\pi(t,T)}\sum_{j=1}^Jq^0_jE[g_{ij}(X_T^{t,x,\alpha,\beta^0_j})]
+\lambda\sum_{i=1}^{I}p_i\inf_{\alpha\in\mathcal{A}_1^\pi(t,T)}\sum_{j=1}^Jq^1_jE[g_{ij}(X_T^{t,x,\alpha,\beta^1_j})]\\
&\geq&(1-\lambda)V_1^\pi(t,x,p,q^0)+\lambda V_1^\pi(t,x,p,q^1)-2\varepsilon,
\end{eqnarray*}
since $\hat{\beta}^0$ and $\hat{\beta}^1$ are $\varepsilon$-optimal strategies for $V_1^\pi(t,x,p,q^0)$ and $V_1^\pi(t,x,p,q^1)$, respectively.
Thus,
\begin{equation}\label{equ 3.6}
V_1^\pi(t,x,p,q^\lambda)\geq\inf_{\alpha\in(\mathcal{A}_1^\pi(t,T))^I}J(t,x,\hat{\alpha},\hat{\beta}^\lambda,p,q^\lambda)
\geq(1-\lambda)V_1^\pi(t,x,p,q^0)+\lambda V_1^\pi(t,x,p,q^1)-\varepsilon.
\end{equation}
Thanks to the arbitrariness of $\varepsilon$, we obtain the desired result.
\end{proof}

Now we introduce the Fenchel transforms (refer to \cite{C2007}). Assume a mapping $\psi:[0,T]\times\mathbb{R}^n\times\Delta(I)\times\Delta(J)\rightarrow\mathbb{R}$ convex in $p$ and concave in $q$ on $\Delta(I)$ and $\Delta(J)$, respectively, then we define its convex conjugate (with respect to variable $p$) $\psi^*$ by
\begin{equation}\label{equ 3.7}
\psi^*(t,x,\bar{p},q)=\sup_{p\in\Delta(I)}\{\bar{p}\cdot p -\psi(t,x,p,q)\},\
(t,x,\bar{p},q)\in[0,T]\times\mathbb{R}^n\times\mathbb{R}^I\times\Delta(J),
\end{equation}
and its concave conjugate (with respect to variable $q$) $\psi^\#$ by
\begin{equation}\label{equ 3.8}
\psi^\#(t,x,p,\bar{q})=\inf_{q\in\Delta(J)}\{\bar{q}\cdot q-\psi(t,x,p,q)\},\
(t,x,p,\bar{q})\in[0,T]\times\mathbb{R}^n\times\Delta(I)\times\mathbb{R}^J.
\end{equation}
Using these notations we denote by $V_1^{\pi*}(W_1^{\pi\#})$ for the convex (respectively, concave) conjugate of $V_1^\pi$ (respectively, $W_1^\pi$) with respect to $p$ (respectively, $q$).
\begin{lemma}\label{le 3.3}
For all $(t,x,\bar{p},q)\in[0,T]\times\mathbb{R}^n\times\mathbb{R}^I\times\Delta(J)$, we have
\begin{equation}\label{equ 3.9}
V_1^{\pi*}(t,x,\bar{p},q)=\inf_{(\beta_j)\in(\mathcal{B}_1^\pi(t,T))^J}\sup_{\alpha\in\mathcal{A}_1^\pi(t,T)}
\max_{i\in\{1,...,I\}}\{\bar{p}_i
-\sum_{j=1}^Jq_jE[g_{ij}(X_T^{t,x,\alpha,\beta_j})]\}.
\end{equation}
\end{lemma}
\begin{proof}
We define
\begin{equation}\label{equ 3.11}
F(t,x,\bar{p},q)=\inf_{(\beta_j)\in(\mathcal{B}_1^\pi(t,T))^J}\sup_{\alpha\in\mathcal{A}_1^\pi(t,T)}
\max_{i\in\{1,...,I\}}\{\bar{p}_i
-\sum_{j=1}^Jq_jE[g_{ij}(X_T^{t,x,\alpha,\beta_j})]\}.
\end{equation}
It is obviously that $F(t,x,\bar{p},q)$ is convex with
respect to $\bar{p}$.
From (\ref{equ 3.7}) and (\ref{equ 3.11}), we have
\begin{eqnarray}\label{equ 3.12}
\begin{split}
F^*(t,x,p,q)&=\sup_{\bar{p}\in\mathbb{R}^I}\{p\cdot \bar{p}-\inf_{(\beta_j)}\max_{i\in\{1,...,I\}}\{\bar{p}_i-\inf_{\alpha}
\sum_{j=1}^Jq_jE[g_{ij}(X_T^{t,x,\alpha,\beta_j})]\}\}\\
&=\sup_{(\beta_j)}\sup_{\bar{p}\in\mathbb{R}^I}\min_{i\in\{1,...,I\}}\{p\cdot \bar{p}-\bar{p}_i+\inf_{\alpha}\sum_{j=1}^Jq_jE[g_{ij}(X_T^{t,x,\alpha,\beta_j})]\}\\
&=\sup_{(\beta_j)}\sup_{\bar{p}\in\mathbb{R}^I}\min_{i\in\{1,...,I\}}\{p\cdot\bar{p}-\bar{p}_i+h_i\},
\end{split}
\end{eqnarray}
where we define $h_i:=\inf\limits_{\alpha}\sum\limits_{j=1}^Jq_jE[g_{ij}(X_T^{t,x,\alpha,\beta_j})]$,
$1\leq i\leq I$.\\
On the other hand,
\begin{eqnarray}\label{equ 3.13}
\begin{split}
\sup_{\bar{p}\in\mathbb{R}^I}\min_{i\in\{1,...,I\}}\{p\cdot\bar{p}-\bar{p}_i+h_i\}
&=\sup_{\bar{p}\in\mathbb{R}^I}\{p\cdot \bar{p}+\min_{i\in\{1,...,I\}}\{h_i-\bar{p}_i\}\}
=\sup_{\bar{p}\in\mathbb{R}^I}\{p\cdot\bar{p}+\inf_{\bar{\bar{p}}\in\Delta(I)}(h-\bar{p})\bar{\bar{p}}\}\\
&=\sup_{\bar{p}\in\mathbb{R}^I}\inf_{\bar{\bar{p}}\in\Delta(I)}\{(h-\bar{p})\bar{\bar{p}}+p\cdot\bar{p}\}
=\inf_{\bar{\bar{p}}\in\Delta(I)}\sup_{\bar{p}\in\mathbb{R}^I}\{(p-\bar{\bar{p}})\bar{p}+h\cdot\bar{\bar{p}}\}\\
&=h\cdot p.
\end{split}
\end{eqnarray}
From (\ref{equ 3.12}), (\ref{equ 3.13}) and (\ref{equ 3.5}), we get
\begin{equation}\label{equ 3.14}
F^*(t,x,p,q)=\sup_{(\beta_j)}\sum_{i=1}^Ip_i\inf_{\alpha}\sum_{j=1}^Jq_jE[g_{ij}(X_T^{t,x,\alpha,\beta_j})]=V_1^\pi(t,x,p,q).
\end{equation}
Since $F$ is convex in $\bar{p}$, we have $V_1^{\pi*}=F^{**}=F$.
 \end{proof}
 Using the definition of $V_1^{\pi*}$ and $W_1^{\pi\#}$, from Lemma \ref{le 3.1} we have the following lemma.
 \begin{lemma}\label{le 3.4}
 For all the partition $\pi$ of the interval $[0,T]$,
 the convex conjugate function $V_1^{\pi*}(t,x,\bar{p},q)$ is Lipschitz with respect to $(t,x,\bar{p},q)$,
 the concave conjugate function $W_1^{\pi\#}(t,x,p,\bar{q})$ is Lipschitz with respect to $(t,x,p,\bar{q})$.
 \end{lemma}
Generally speaking, the game with asymmetric information does not have the dynamic programming principle,
but it has s sub-dynamic programming principle.
\begin{lemma}\label{le 3.5}
 For any $(t,x,\bar{p},q)\in[t_{k-1},t_k)\times\mathbb{R}^n\times\mathbb{R}^I\times\Delta(J)$, and for all $l\ (k\leq l \leq N)$,
we have
\begin{equation}\label{equ 3.16}
\begin{split}
V_1^{\pi*}(t,x,\bar{p},q)\leq \inf_{\beta\in\mathcal{B}_1^\pi(t,t_l)}\sup_{\alpha\in\mathcal{A}_1^\pi(t,t_l)}
E[V_1^{\pi*}(t_l,X_{t_l}^{t,x,\alpha,\beta},\bar{p},q)].
\end{split}
\end{equation}
\end{lemma}
\begin{proof}
We define
\begin{equation}\label{equ 3.17}
G(t,t_l,x,\bar{p},q)=\inf_{\beta\in\mathcal{B}_1^\pi(t,t_l)}\sup_{\alpha\in\mathcal{A}_1^\pi(t,t_l)}
E[V_1^{\pi*}(t_l,X_{t_l}^{t,x,\alpha,\beta},\bar{p},q)].
\end{equation}
For any given $\varepsilon>0$, let $\beta^0\in\mathcal{B}_1^\pi(t,t_l)$ be an $\varepsilon$-optimal
strategy for $G(t,t_l,x,\bar{p},q)$, i.e.,
\begin{equation}\label{equ 3.17.1}
|G(t,t_l,x,\bar{p},q)-\sup_{\alpha\in\mathcal{A}_1^\pi(t,t_l)}
E[V_1^{\pi*}(t_l,X_{t_l}^{t,x,\alpha,\beta^0},\bar{p},q)]|\leq \varepsilon.
\end{equation}
For any $y\in\mathbb{R}^n$, there exists an $\varepsilon$-optimal
strategy $\hat{\beta}^y=(\beta_j^y)_{j=1,\ldots,J}\in(\mathcal{B}_1^\pi(t_l,T))^J$ for
$V_1^{\pi*}(t_l,y,\bar{p},q)$ for Player \Rmnum {2}, i.e.,
\begin{equation}\label{equ 3.17.2}
|V_1^{\pi*}(t_l,y,\bar{p},q)-\sup_{\alpha\in\mathcal{A}_1^\pi(t_l,T)}
\max_{i\in\{1,...,I\}}\{\bar{p}_i
-\sum_{j=1}^Jq_jE[g_{ij}(X_T^{t_l,y,\alpha,\beta^y_j})]\}|\leq \varepsilon.
\end{equation}
 Because
$\sup_{\alpha\in\mathcal{A}_1^\pi(t_l,T)}\max_{i\in\{1,\ldots,I\}}\big\{\bar{p}_i-
\sum_{j=1}^Jq_jE[g_{ij}(X_T^{t_l,y,\alpha,\beta_j^y})] \big\}$ and $V_1^{\pi*}(t_l,y,\bar{p},q)$ are Lipschitz continuous
with respect to $y$, $\hat{\beta}^y$ is a $(2\varepsilon)$-optimal strategies for $V_1^{\pi*}(t_l,z,\bar{p},q)$, if
$z\in B_r(y)$, where $B_r(y)$ is the ball with small enough radius $r$.

Since the coefficient $f$ is bounded, there exists some $R>0$ large enough such that all the value of $X_{t_l}^{t,x,\alpha,\beta}$ belong to the ball
$B_R(0)$. Then we assume $(O_n)$, $n=1,\ldots,n_0$, is a finite Borel partition of $B_R(0)$. For any $x_n\in O_n$, we
denote
$\beta_j^n=\beta_j^{x_n}$, $n=1,\ldots,n_0$, the strategy $(\beta_j^n)$ is $(2\varepsilon)$-optimal for $V_1^{\pi*}
(t_l,z,\bar{p},q)$, for any $z\in O_n$, i.e.,
\begin{equation}\label{equ 3.17.3}
|V_1^{\pi*}
(t_l,z,\bar{p},q)-\sup_{\alpha\in\mathcal{A}_1^\pi(t_l,T)}
\max_{i\in\{1,...,I\}}\{\bar{p}_i
-\sum_{j=1}^Jq_jE[g_{ij}(X_T^{t_l,z,\alpha,\beta^n_j})]\}|\leq 2\varepsilon.
\end{equation}
For any $\omega\in\Omega$ and $u\in\mathcal{U}_{t,T}$, we define
$$\beta_j(\omega,u)(s)=\left\{
\begin{array}{ll}
\beta^0(\omega,u)(s),&s\in[t,t_l),\\
\sum\limits_{n=1}^{n_0}\beta_j^n(\omega,u|_{[t_l,T)})\cdot I_{\{X_{t_l}^{t,x,u,\beta^0(u)}\in O_n\}},& s\in [t_l,T].
\end{array}
\right.
$$
Then, we have $\beta_j\in\mathcal{B}_1^\pi(t,T)$. For any $\alpha\in\mathcal{A}_1^\pi(t,T)$, we know $\alpha$ has the following form:
\begin{eqnarray*}
&&\alpha(\omega,v)(s)=\sum_{m=k}^l\alpha_m((\zeta^\pi_{1,k-1},\zeta^\pi_{2,k-1},\ldots,\zeta^\pi_{1,m-1})(\omega),v)(s)I_{[t\vee t_{m-1},t_m)}(s)\\
&&+\sum_{m=l+1}^N\alpha_m((\zeta^\pi_{1,k-1},\zeta^\pi_{2,k-1},\ldots,\zeta^\pi_{1,l-1},\zeta^\pi_{2,l-1},\zeta^\pi_{1,l},\zeta^\pi_{2,l},\zeta^\pi_{l+1},\ldots,
\zeta^\pi_{m-2},
\zeta^\pi_{1,m-1})(\omega),v)(s)I_{[t_{m-1},t_m)}(s).
\end{eqnarray*}
For $s\in[t_l,T]$, we define $\tilde{\alpha}(\omega,Q,v)(s)
=\sum_{m=l+1}^N\alpha_m(Q,(\zeta^\pi_{1,l},\zeta^\pi_{2,l},\ldots,
\zeta^\pi_{1,m-1})(\omega),v)(s)I_{[t_{m-1},t_m)}(s)$, where
$Q$ is a $2(l-k)+2$-dimensional constant vector. Obviously, we know $\tilde{\alpha}(Q)\in\mathcal{A}_1^\pi(t_l,T)$. Then for any $\alpha\in\mathcal{A}_1^\pi(t,T)$, due to $X_{t_l}^{t,x,\alpha,\beta^0}$ and
$Q_0=(\zeta^\pi_{1,k-1},\zeta^\pi_{2,k-1},\ldots,\zeta^\pi_{1,l-1},\zeta^\pi_{2,l-1})$ are $\mathcal{F}_{t_{k-2},t_{l-1}}$-measurable,
$\beta_j^n$ and $\tilde{\alpha}$ are $\mathcal{F}_{t_{l-1},T}$-measurable, we have
\begin{equation}\label{equ 3.18}
\begin{split}
&E[g_{ij}(X_T^{t,x,\alpha,\beta_j})]=E[\sum_{n=1}^{n_0}g_{ij}(X_T^{t_l,X_{t_l}^{t,x,\alpha,\beta^0},
\tilde{\alpha}(Q_0),\beta_j^n})
\cdot I_{\{X_{t_l}^{t,x,\alpha,\beta^0}\in O_n\}}]\\
&=E[\sum_{n=1}^{n_0}E[g_{ij}(X_T^{t_l,y,
\tilde{\alpha}(Q),\beta_j^n})]_{y=X_{t_l}^{t,x,\alpha,\beta^0},Q=Q_0}
\cdot I_{\{X_{t_l}^{t,x,\alpha,\beta^0}\in O_n\}}].
\end{split}
\end{equation}
 From (\ref{equ 3.18}),
(\ref{equ 3.17.3}) and (\ref{equ 3.17.1}), we have
\begin{eqnarray}\label{equ 3.19}
\begin{split}
&\max_{i\in\{1,\ldots,I\}}\big\{\bar{p}_i-\sum_{j=1}^Jq_j E[g_{ij}(X_T^{t,x,\alpha,\beta_j})]\big\}\\
=&\max_{i\in\{1,\ldots,I\}}\Big\{\bar{p}_i-\sum_{j=1}^Jq_j E\big[\sum_{n=1}^{n_0}E[g_{ij}(X_T^{t_l,y,
\tilde{\alpha}(Q),\beta_j^n})]_{y=X_{t_l}^{t,x,\alpha,\beta^0},Q=Q_0}
\cdot I_{\{X_{t_l}^{t,x,\alpha,\beta^0}\in O_n\}}\big]\Big\}\\
\leq& E\big[\sum_{n=1}^{n_0}\max_{i\in\{1,\ldots,I\}}\{\bar{p}_i-\sum_{j=1}^Jq_j
E[g_{ij}(X_T^{t_l,y,\tilde{\alpha}(Q),\beta_j^n})]\}_{y=X_{t_l}^{t,x,\alpha,\beta^0},Q=Q_0}\cdot I_{\{X_{t_l}^{t,x,\alpha,\beta^0}\in O_n\}}\big]\\
\leq& E\big[\sum_{n=1}^{n_0}\sup_{\alpha'\in\mathcal{A}_1^\pi(t_l,T)}\max_{i\in\{1,\ldots,I\}}\{\bar{p}_i-\sum_{j=1}^Jq_j
E[g_{ij}(X_T^{t_l,y,\alpha',\beta_j^n})]\}_{y=X_{t_l}^{t,x,\alpha,\beta^0}}\cdot I_{\{X_{t_l}^{t,x,\alpha,\beta^0}\in O_n\}}\big]\\
\leq& E[\sum_{n=1}^{n_0}V_1^{\pi*}(t_l,X_{t_l}^{t,x,\alpha,\beta^0},\bar{p},q)\cdot I_{\{X_{t_l}^{t,x,\alpha,\beta^0}\in O_n\}}]+2\varepsilon\\
\leq& G(t,t_l,x,\bar{p},q)+3\varepsilon,
\end{split}
\end{eqnarray}
which means that $V_1^{\pi*}(t,x,\bar{p},q)\leq G(t,t_l,x,\bar{p},q)$.
\end{proof}
We assume $(\pi_n)_{n\geq 1}$ is a  sequence  partitions of the interval $[0,T]$ satisfying that when $n\rightarrow\infty$, the mesh of
the partition $|\pi_n|$ tends to zero.
From Lemma \ref{le 3.4}, applying the Arzel\`a-Ascoli Theorem to $V_1^{\pi_n*}(t,x,\bar{p},q)$ and
$W_1^{\pi_n\#}(t,x,p,\bar{q})$, we have the following lemma.
\begin{lemma}\label{le 3.6}
There exists a subsequence of partitions $(\pi_n)_{n\geq 1}$, still denoted by $(\pi_n)_{n\geq 1}$ and two
functions $\tilde{V}:[0,T]\times\mathbb{R}^n\times\mathbb{R}^I\times\Delta(J)\mapsto\mathbb{R}$
and $\tilde{W}:[0,T]\times\mathbb{R}^n\times\Delta(I)\times\mathbb{R}^J\mapsto\mathbb{R}$ such that $(V_1^{\pi_n*},
W_1^{\pi_n\#})\rightarrow(\tilde{V},\tilde{W})$ uniformly on compacts in $[0,T]\times\mathbb{R}^n\times\Delta(I)
\times\Delta(J)\times\mathbb{R}^I\times\mathbb{R}^J$.
\end{lemma}
\begin{remark}\label{re 3.1}
Notice that from Lemma \ref{le 3.4}, the limit functions $\tilde{V}$ and $\tilde{W}$ are Lipschitz continuous with respect to all
their variables.
\end{remark}
Now we prove that the limit functions $\tilde{V}$ and $\tilde{W}$ are a viscosity subsolution  and a viscosity
supersolution of some HJI equation, respectively. For more details on viscosity solutions, the reader
is referred to \cite{CIL1992}.
\begin{lemma}\label{th 3.1}
The limit function $\tilde{V}(t,x,\bar{p},q)$ is
a viscosity subsolution of the following HJI equation
\begin{equation}\label{equ 3.20}
\left\{
\begin{array}{ll}
\frac{\partial\tilde{V}}{\partial t}(t,x)+H^*(x,D\tilde{V}(t,x))=0,&(t,x)\in[0,T]\times\mathbb{R}^n,\\
\tilde{V}(T,x)=\max\limits_{i\in\{1,\ldots,I\}}\{\bar{p}_i-\sum\limits_{j=1}^Jq_jg_{ij}(x)\},& (\bar{p},q)
\in\mathbb{R}^I\times\Delta(J),
\end{array}
\right.
\end{equation}
where
\begin{eqnarray*}
H^*(x,\xi)=-H(x,-\xi)&=&\inf_{\nu\in\mathcal{P}(V)}\sup_{\mu\in\mathcal{P}(U)}
\big(\int_{U\times V}f(x,u,v)\mu(du)\nu(dv)\cdot\xi\big)\\
&=&\sup_{\mu\in\mathcal{P}(U)}\inf_{\nu\in\mathcal{P}(V)}\big(\int_{U\times V}f(x,u,v)\mu(du)\nu(dv)\cdot\xi\big).
\end{eqnarray*}
\end{lemma}
\begin{proof}
For simplicity, we denote $\tilde{V}(t,x,\bar{p},q)$ by $\tilde{V}(t,x)$, for fixed $(\bar{p},q)\in\mathbb{R}^I\times\Delta(J)$.
For any fixed $(t,x)\in[0,T]\times\mathbb{R}^n$, since the coefficient $f$ is bounded, there is some $M>0$ such that,
$\bar{B}_M(x)\supset\{X_r^{s,y,\alpha,\beta},\ (s,y)\in[0,T]\times\bar{B}_1(x), (\alpha,\beta)\in\mathcal{A}_1^\pi(s,T)\times\mathcal{B}_1^\pi(s,T), r\in[s,T] \}$, where $\bar{B}_M(x)$ is the closed ball with the center $x$ and the radius $M$. From Lemma \ref{le 3.6}, we know $V_1^{\pi_n*}$ converge to $\tilde{V}$ over $[0,T]\times\bar{B}_M(x)$.
Let $\varphi\in C_b^1([0,T]\times\mathbb{R}^n)$
(the set of  bounded continuous functions where
the first order partial derivate is bounded and continuous) be a test function such that
\begin{equation}\label{equ 3.1.21.1}
(\tilde{V}-\varphi)(t,x)>(\tilde{V}-\varphi)(s,y),\ \text{for\ all}\ (s,y)\in[0,T]\times\bar{B}_M(x)\setminus\{(t,x)\}.
\end{equation}
Let $(s_n,x_n)\in[0,T]\times\bar{B}_M(x)$ be the maximum point of $V_1^{\pi_n*}-\varphi$ over $[0,T]\times\bar{B}_M(x)$, then
there exists a subsequence of $(s_n,x_n)$ still denoted by $(s_n,x_n)$, such that $(s_n,x_n)$ converges to $(t,x)$.

Indeed, since $[0,T]\times\bar{B}_M(x)$ is a compact set, there exists a subsequence $(s_n,x_n)$ and $(\bar{s},\bar{x})\in[0,T]\times\bar{B}_M(x)$ such that
$(s_n,x_n)\rightarrow(\bar{s},\bar{x})$. Due to $(V_1^{\pi_n*}-\varphi)(s_n,x_n)\geq (V_1^{\pi_n*}-\varphi)(t,x)$, for $n\geq 1$, we have
\begin{equation}\label{equ 3.1.21.2}
(\tilde{V}-\varphi)(\bar{s},\bar{x})\geq (\tilde{V}-\varphi)(t,x).
\end{equation}
From (\ref{equ 3.1.21.1}) and (\ref{equ 3.1.21.2}), we have $(\bar{s},\bar{x})=(t,x)$.\\
For the partition $\pi_n$, we assume $t^n_{k_{n-1}}\leq s_n<t^n_{k_n}$, for simplicity, we write $t^n_{k-1}\leq s_n<t^n_{k}$. Since $x_n\rightarrow x$, there
is a positive integer $N$ such that for all $n\geq N$, we have $|x_n-x|\leq 1$. Then from Lemma \ref{le 3.5}, we get
\begin{eqnarray}\label{equ 3.22}
\begin{split}
\varphi(s_n,x_n)= V_1^{\pi_n*}(s_n,x_n)
&\leq \inf_{\beta\in\mathcal{B}_1^{\pi_n}(s_n,t^n_{k})}\sup_{\alpha\in\mathcal{A}_1^{\pi_n}(s_n,t^n_{k})}
E[V_1^{\pi_n*}(t^n_{k},X^{s_n,x_n,\alpha,\beta}_{t_{k}^n})]\\
&\leq \inf_{\beta\in\mathcal{B}_1^{\pi_n}(s_n,t^n_{k})}\sup_{\alpha\in\mathcal{A}_1^{\pi_n}(s_n,t^n_{k})}
E[\varphi(t^n_{k},X^{s_n,x_n,\alpha,\beta}_{t_{k}^n})].
\end{split}
\end{eqnarray}
Thus we get
\begin{eqnarray}\label{equ 3.23}
\begin{split}
0\leq&\inf_{\beta\in\mathcal{B}_1^{\pi_n}(s_n,t^n_{k})}\sup_{\alpha\in\mathcal{A}_1^{\pi_n}(s_n,t^n_{k})}
E[\varphi(t^n_{k},X^{s_n,x_n,\alpha,\beta}_{t_{k}^n})-\varphi(s_n,x_n)]\\
=& \inf_{\beta\in\mathcal{B}_1^{\pi_n}(s_n,t^n_{k})}\sup_{\alpha\in\mathcal{A}_1^{\pi_n}(s_n,t^n_{k})}
E[\int_{s_n}^{t_{k}^n}(\frac{\partial \varphi}{\partial r}(r,X_r^{s_n,x_n,\alpha,\beta})+
f(X_r^{s_n,x_n,\alpha,\beta},\alpha_r,\beta_r)\cdot D\varphi(r,X_r^{s_n,x_n,\alpha,\beta}))dr].
\end{split}
\end{eqnarray}
For $(u,v)\in\mathcal{U}_{t,T}\times\mathcal{V}_{t,T}$, we introduce the following continuity modulus,
\begin{equation}\label{equ 3.24}
m(\delta):=\sup_{\mbox{\tiny
$\begin{array}{c}
|r-s|+|y-\bar{x}|\leq \delta,\\
u\in U,v\in V, \bar{x},y\in\bar{B}_M(x)
\end{array}$}
}\big|(\frac{\partial \varphi}{\partial r}(r,y)+
f(y,u,v)\cdot D\varphi(r,y))-(\frac{\partial \varphi}{\partial r}(s,\bar{x})+
f(\bar{x},u,v)\cdot D\varphi(s,\bar{x}))\big|.
\end{equation}
Obviously, $m(\delta)$ is nondecreasing in $\delta$ and $m(\delta)\rightarrow 0$, as $\delta\downarrow 0$. From (\ref{equ 2.1.1}), considering
that $|X_r^{s_n,x_n,\alpha,\beta}-x_n|\leq C|r-s_n|\leq C|t^n_{k}-s_n|,\ r\in[s_n,t^n_{k}]$ and
from $(\ref{equ 3.24})$
we know that
\begin{eqnarray}\label{equ 3.25}
\begin{array}{l}
|(\frac{\partial \varphi}{\partial r}(r,X_r^{s_n,x_n,\alpha,\beta})+
f(X_r^{s_n,x_n,\alpha,\beta},\alpha_r,\beta_r)\cdot D\varphi(r,X_r^{s_n,x_n,\alpha,\beta}))-\\
(\frac{\partial \varphi}{\partial r}(s_n,x_n)+
f(x_n,\alpha_r,\beta_r)\cdot D\varphi(s_n,x_n))|
\leq m(C|t_{k}^n-s_n|),\ r\in[s_n,t^n_{k}].
\end{array}
\end{eqnarray}
It follows from (\ref{equ 3.23}) and (\ref{equ 3.25}) that
\begin{equation}\label{equ 3.26}
\begin{split}
&-(t^n_{k}-s_n)\big(\frac{\partial \varphi}{\partial r}(s_n,x_n)+m(C|t_{k}^n-s_n|)\big)\\
\leq&
\inf_{\beta\in\mathcal{B}_1^{\pi_n}(s_n,t^n_{k})}\sup_{\alpha\in\mathcal{A}_1^{\pi_n}(s_n,t^n_{k})}
E[\int_{s_n}^{t_{k}^n}f(x_n,\alpha_r,\beta_r)\cdot D\varphi(s_n,x_n)dr]\\
\leq& \sup_{\alpha\in\mathcal{A}_1^{\pi_n}(s_n,t^n_{k})}
E[\int_{s_n}^{t_{k}^n}f(x_n,\alpha_r,\tilde{\beta}_r)\cdot D\varphi(s_n,x_n)dr],
\end{split}
\end{equation}
where we take $\tilde{\beta}_r=\tilde{v}(\zeta_{2,k-1}^{\pi_n})$, $r\in[s_n,t_{k}^n]$, $\tilde{v}$ is a $V$-valued  measurable function.
Define $\rho_n=(t_k^n-s_n)^2$, from (\ref{equ 3.26})
there exists
 a $\rho_n$-optimal strategy $\alpha^n$ (depending on $\tilde{\beta}$) such that
\begin{eqnarray}\label{equ 3.27}
\begin{split}
&-(t^n_{k}-s_n)\big(\frac{\partial \varphi}{\partial r}(s_n,x_n)+m(C|t_{k}^n-s_n|)+(t_k^n-s_n)\big)
\leq
E[\int_{s_n}^{t_{k}^n}f(x_n,\alpha^n_r,\tilde{\beta}_r)\cdot D\varphi(s_n,x_n)dr]\\
&= \int_{s_n}^{t_{k}^n}E[f(x_n,\alpha^n_r(\zeta_{1,k-1}^{\pi_n},\tilde{v}),\tilde{v}(\zeta_{2,k-1}^{\pi_n}))\cdot D\varphi(s_n,x_n)]dr.
\end{split}
\end{eqnarray}
Notice that on the interval $[s_n,t_k^n]$, $\alpha^n$ does not depend on the control $\tilde{v}$ due to the delay property. Then
thanks to the independence between $\zeta_{1,k-1}^{\pi_n}$ and $\zeta_{2,k-1}^{\pi_n}$, from (\ref{equ 3.27}) we get
\begin{eqnarray}\label{equ 3.28}
\begin{split}
&-(t^n_{k}-s_n)\big(\frac{\partial \varphi}{\partial r}(s_n,x_n)+m(C|t_{k}^n-s_n|)+(t_{k}^n-s_n)\big)\\
&\leq \int_{s_n}^{t_{k}^n}\sup_{\mu\in\mathcal{P}(U)}\int_{U}E[f(x_n,u,\tilde{v}(\zeta_{2,k-1}^{\pi_n}))\cdot D\varphi(s_n,x_n)]\mu(du)dr.
\end{split}
\end{eqnarray}
From the arbitrariness of $\tilde{v}$, from (\ref{equ 3.28}) we get
\begin{eqnarray}\label{equ 3.1.21.3}
\begin{split}
&-(t^n_{k}-s_n)\big(\frac{\partial \varphi}{\partial r}(s_n,x_n)+m(C|t_{k}^n-s_n|)+(t_{k}^n-s_n)\big)\\
&\leq \int_{s_n}^{t_{k}^n}\inf_{\nu\in\mathcal{P}(V)}\sup_{\mu\in\mathcal{P}(U)}\int_{U\times V}f(x_n,u,v)\cdot D\varphi(s_n,x_n)\mu(du)\nu(dv)dr\\
&= (t_k^n-s_n)\inf_{\nu\in\mathcal{P}(V)}\sup_{\mu\in\mathcal{P}(U)}\int_{U\times V}E[f(x_n,u,v)\cdot D\varphi(s_n,x_n)]\mu(du)\nu(dv),
\end{split}
\end{eqnarray}
which means that
\begin{equation}\label{equ 3.30}
\begin{split}
&-\big(\frac{\partial \varphi}{\partial r}(s_n,x_n)+m(C|t_{k}^n-s_n|)+(t_{k}^n-s_n)\big)\\
&\leq \inf_{\nu\in\mathcal{P}(V)}\sup_{\mu\in\mathcal{P}(U)}\int_{U\times V}E[f(x_n,u,v)\cdot D\varphi(s_n,x_n)]\mu(du)\nu(dv).
\end{split}
\end{equation}
Recall that $(s_n,x_n)\rightarrow (t,x)$ and $0\leq (t_k^n-s_n)\leq (t_k^n-t_{k-1}^n)\leq |\pi_n|$, when $n\rightarrow\infty$ we get
\begin{equation}\label{equ 3.31}
\frac{\partial\varphi}{\partial t}(t,x)+\inf_{\nu\in\mathcal{P}(V)}\sup_{\mu\in\mathcal{P}(U)}\int_{U\times V}
f(x,u,v)\cdot D\varphi(t,x)\mu(du)\nu(dv)\geq 0.
\end{equation}
\end{proof}
Now we want to prove $\tilde{W}$ is a viscosity supersolution of the HJI equation (\ref{equ 3.20}).
Notice that
\begin{equation}\label{equ 3.32}
-W_1^\pi(t,x,p,q)=\sup_{(\alpha_i)\in(\mathcal{A}_1^\pi(t,T))^I}
\inf_{(\beta_j)\in(\mathcal{B}_1^\pi(t,T))^J}
\sum_{i=1}^I\sum_{j=1}^Jp_iq_jE[-g_{ij}(X_T^{t,x,\alpha_i,\beta_j})].
\end{equation}
Then $-W_1^\pi(t,x,p,q)$ has the same form as $V_1^\pi$, only change the role
of players. Thus, the convex conjugate $-W_1^\pi(t,x,p,q)$ with respect to $q$,
i.e., $-(W_1^{\pi\#}(t,x,p,-\bar{q}))$ satisfies a sub-dynamic programming principle. Then
similar to Lemma \ref{le 3.5} and Theorem \ref{th 3.1}
we have the following result.
\begin{lemma}\label{le 3.7}
For any $(t,x,p,\bar{q})\in[0,T]\times\mathbb{R}^n\times\Delta(I)\times
\mathbb{R}^J$, and for all $l$ $(k\leq l\leq n)$, we have
\begin{equation}\label{equ 3.33}
W_1^{\pi\#}(t,x,p,\bar{q})\geq \sup_{\alpha\in\mathcal{A}_1^\pi(t,t_l)}\inf_{\beta\in\mathcal{B}_1^\pi(t,t_l)}
E[W_1^{\pi\#}(t_l,X_{t_l}^{t,x,\alpha,\beta},p,\bar{q})],
\end{equation}
and $\tilde{W}$ $($the limit of $(W_1^{\pi_n\#})$ on compacts$)$ is a supersolution of the HJI equation $(\ref{equ 3.20})$.
\end{lemma}
We now give the definition of dual solutions for the following HJI equation
\begin{equation}\label{equ 3.34}
\left\{
\begin{array}{ll}
\frac{\partial V}{\partial t}(t,x)+H(x,D{V}(t,x))=0,&(t,x)\in[0,T]\times\mathbb{R}^n,\\
V(T,x)=\sum_{i,j}p_iq_jg_{ij}(x),& (p,q)\in\Delta(I)\times\Delta(J),
\end{array}
\right.
\end{equation}
where $H(x,\xi)=\inf_{\mu\in\mathcal{P}(U)}\sup_{\nu\in\mathcal{P}(V)}
\big(\int_{U\times V}f(x,u,v)\mu(du)\nu(dv)\cdot\xi\big).$
\begin{definition}
A function $w:[0,T]\times\mathbb{R}^n\times\Delta(I)\times\Delta(J)\mapsto\mathbb{R}$ is called a dual viscosity
subsolution of the equation $(\ref{equ 3.34})$ if, firstly, $w$ is Lipschitz continuous with all its variables, convex with respect to $p$ and
concave with respect to $q$, secondly, for any $(p,\bar{q})\in\Delta(I)\times\mathbb{R}^J$,
$w^\#(t,x,p,\bar{q})$ is a viscosity supersolution of the dual HJI equation
\begin{equation}\label{equ 3.35}
\frac{\partial V}{\partial t}(t,x)+H^*(x,D{V}(t,x))=0,\ (t,x)\in[0,T]\times\mathbb{R}^n,
\end{equation}
where $H^*(x,\xi)=-H(x,-\xi)$.

A function $w:[0,T]\times\mathbb{R}^n\times\Delta(I)\times\Delta(J)\mapsto\mathbb{R}$ is called a dual viscosity
supersolution of the equation $(\ref{equ 3.34})$ if, firstly, $w$ is Lipschitz continuous with all its variables, convex with respect to $p$ and
concave with respect to $q$, secondly, for any $(\bar{p},q)\in\mathbb{R}^I\times\Delta(J)$,
$w^*(t,x,\bar{p},q)$ is a viscosity subsolution of the dual HJI equation $(\ref{equ 3.35})$.

The function $w$ is called the dual viscosity solution of the equation $(\ref{equ 3.34})$ if $w$ is a dual viscosity subsolution and
a dual viscosity supersolution of the equation $(\ref{equ 3.34})$.
\end{definition}
\begin{lemma}\label{le 3.8}
Let $w_1,w_2:[0,T]\times\mathbb{R}^n\times\Delta(I)\times\Delta(J)\mapsto\mathbb{R}$ be a
dual viscosity subsolution and a dual viscosity supersolution of the HJI equation $(\ref{equ 3.34})$, respectively. If, for all $(x,p,q)\in\mathbb{R}^
n\times\Delta(I)\times\Delta(J)$, $w_1(T,x,p,q)\leq w_2(T,x,p,q)$, then we have $w_1\leq w_2$ on
$[0,T]\times\mathbb{R}^n\times\Delta(I)\times\Delta(J)$.
\end{lemma}
The proof of Lemma \ref{le 3.8} is referred to Theorem 5.1 in \cite{C2007}.
\begin{theorem}\label{th 3.2}
The functions $(V_1^{\pi_n})$ and $(W_1^{\pi_n})$ converge uniformly
on compacts to a same Lipschitz continuous function $U$ when the mesh of the partition $\pi_n$ tends to $0$. Moreover, the function $U$ is  the unique dual viscosity  solution
of the HJI equation $(\ref{equ 3.34})$.
\end{theorem}
For this we first prove the following proposition, then we get Theorem \ref{th 3.2} directly.
\begin{proposition}\label{prop 3.1}
There exists a subsequence of partitions $\pi_n$ with $|\pi_n|\rightarrow 0$, still denoted by $(\pi_n)_{n\geq 1}$
such that $(V_1^{\pi_n})$ and $(W_1^{\pi_n})$ converges uniformly on compacts to the same function $U$, and the function $U$ is the unique dual viscosity solution
of the HJI equation $(\ref{equ 3.34})$.
\end{proposition}
\begin{remark}\label{re 3.2}
If Proposition \ref{prop 3.1} holds, then there exists a sub-subsequence $(\pi_{n_l})$ such that $(V_1^{\pi_{n_l}},W_1^{\pi_{n_l}})$
converges uniformly to the function $(U,U)$ and the limit $U$ is the unique
dual solution of the HJI equation $(\ref{equ 3.34})$. Therefore, the limits of all converging sub-subsequences are the same, then
Theorem \ref{th 3.2} holds.
\end{remark}
Now we prove (of Proposition \ref{prop 3.1}).
\begin{proof}
From Lemma \ref{le 3.1}, using the Arzel\`a-Ascoli Theorem we know there exist two bounded Lipschitz functions
$V_1$ and $W_1:$ $[0,T]\times\mathbb{R}^n\times\Delta(I)\times\Delta(J)\mapsto\mathbb{R}$ such that $(V_1^{\pi_n},W_1^{\pi_n})\rightarrow
(V_1,W_1)$ uniformly on compacts in $[0,T]\times\mathbb{R}^n\times\Delta(I)\times\Delta(J)$, and $(V_1,W_1)$ are convex in $p$,
concave in $q$.

From Lemma \ref{le 3.6}, $\tilde{W}=\lim\limits_{n\rightarrow\infty}W_1^{\pi_n\#},\ \tilde{V}=\lim\limits_{n\rightarrow\infty}V_1^{\pi_n*}$.
We know $\tilde{V}^*$ and $\tilde{W}^\#$ are a dual viscosity supersolution and a dual viscosity subsolution of HJI equation
$(\ref{equ 3.34})$, respectively, and the terminal value $\tilde{V}^*(T,x,p,q)=\tilde{W}^\#(T,x,p,q)=\sum_{ij}p_iq_jg_{ij}(x)$. Then from
Lemma \ref{le 3.8}, we have
\begin{equation}\label{equ 3.36}
\tilde{V}^*\geq\tilde{W}^\#,\ \text{on}\ [0,T]\times\mathbb{R}^n\times\Delta(I)\times\Delta(J).
\end{equation}
Since $V_1(t,x,p,q)=\lim\limits_{n\rightarrow\infty}V_1^{\pi_n}(t,x,p,q)$ for any $M>0$ and $(t,x,p,q)\in[0,T]\times\bar{B}_M(0)\times\Delta(I)\times\Delta(J)$. Then for any $\rho>0$, we know there exists a positive
integer $N_{\rho,M}$, such that for any $(t,x,p,q)\in[0,T]\times\bar{B}_M(0)\times\Delta(I)\times\Delta(J)$, it holds
$|V_1^{\pi_n}(t,x,p,q)-V_1(t,x,p,q)|\leq \rho.$
Thus, from the definition of convex
conjugate we have
\begin{eqnarray*}
&&|V_1^{\pi_n*}(t,x,\bar{p},q)-V_1^{*}(t,x,\bar{p},q)|
=|\sup_{p\in\Delta(I)}\{\bar{p}\cdot p-V_1^{\pi_n}(t,x,p,q)\}-\sup_{p\in\Delta(I)}\{\bar{p}\cdot p-V_1(t,x,p,q)\}|\\
&&\leq \sup_{p\in\Delta(I)}|V_1^{\pi_n}(t,x,p,q)-V_1(t,x,p,q)|\leq \rho.
\end{eqnarray*}
Therefore, $V_1^*(t,x,\bar{p},$ $q)=\lim\limits_{n\rightarrow\infty}V_1^{\pi_n*}(t,x,\bar{p},q)$. Therefore, $\tilde{V}=\lim\limits_{n\rightarrow\infty}V_1^{\pi_n*}=V_1^*$,
since $V_1$ is convex in $p$, we have $V_1=V_1^{**}=\tilde{V}^*$. Similarly, we have $W_1=\tilde{W}^\#$. From $(\ref{equ 3.36})$
 we have
\begin{equation}\label{equ 3.37}
 W_1\leq V_1,\ \text{on}\ [0,T]\times\mathbb{R}^n\times\Delta(I)\times\Delta(J).
\end{equation}
On the other hand, knowing that $W_1^{\pi_n}\geq V_1^{\pi_n}$,  we have
\begin{equation}\label{equ 3.38}
 W_1\geq V_1,\ \text{on}\ [0,T]\times\mathbb{R}^n\times\Delta(I)\times\Delta(J).
 \end{equation}
From $(\ref{equ 3.37})$ and $(\ref{equ 3.38})$, we know that $U:=V_1=W_1$ on $[0,T]\times\mathbb{R}^n\times\Delta(I)\times\Delta(J)$. Furthermore,
from the above proof, we also know that $U$ is the unique dual viscosity solution of HJI equation $(\ref{equ 3.34})$.
\end{proof}
From Theorem \ref{le 3.1.19} and Theorem \ref{th 3.2}, we obtain the following result directly.
\begin{theorem}\label{th 3.3}
The functions $(V^{\pi_n})$ and $(W^{\pi_n})$ converge uniformly
on compacts to a same Lipschitz continuous function $U$ when the mesh of the partition $\pi_n$ tends to $0$. Moreover, the function $U$ is  the unique dual viscosity  solution
of the HJI equation $(\ref{equ 3.34})$.
\end{theorem}

\section{{\protect \large Nash equilibrium payoffs for nonzero-sum differential games with symmetric information and without Isaacs condition}}
In this section we consider the existence of Nash equilibrium payoffs for nonzero-sum differential games with symmetric information $(\text{i.e.},\
I=J=1)$ and without Isaacs condition. From Theorem \ref{le 3.1.19}, we only need to consider the strategies $\alpha\in\mathcal{A}_1^\pi(t,T)$ and
$\beta\in\mathcal{B}_1^\pi(t,T)$ for our nonzero-sum games.
Let $g_1:\mathbb{R}^n\mapsto\mathbb{R}$ and $g_2:\mathbb{R}^n\mapsto\mathbb{R}$ be two bounded Lipschitz continuous functions. For $(t,x)\in[0,T]\times\mathbb{R}^n$, $(u,v)\in\mathcal{U}_{t,T}^{\pi,1}\times\mathcal{V}^{\pi,1}_{t,T}$ (the definition of $\mathcal{U}_{t,T}^{\pi,1}$ and
 $\mathcal{V}^{\pi,1}_{t,T}$ refer to Remark \ref{re 2.2}), we define
\begin{equation}\label{equ 5.1}
J_1(t,x,u,v)=E[g_1(X_T^{t,x,u,v})]\ \text{and}\ J_2(t,x,u,v)=E[g_2(X_T^{t,x,u,v})],
\end{equation}
where $X^{t,x,u,v}$ is the solution of the equation $(\ref{equ 2.1})$.
From Remark \ref{re 2.1}, we know for any $(\alpha,\beta)\in\mathcal{A}_1^\pi(t,T)\times\mathcal{B}_1^\pi(t,T)$, there exists
$(u,v)\in\mathcal{U}_{t,T}^{\pi,1}\times\mathcal{V}_{t,T}^{\pi,1}$, such that $\alpha(v)=u,\beta(u)=v$. Thus, we have
$J_m(t,x,\alpha,\beta)=J_m(t,x,u,v),\ m=1,2$, respectively.

 Here, for the nonzero-sum differential games
Player \Rmnum {1} wants to maximize $J_1(t,x,\alpha,\beta)$, while Player \Rmnum{2} wants to maximize $J_2(t,x,\alpha,\beta)$.
In general, a Nash equilibrium point is a couple strategies $(\bar{\alpha},\bar{\beta})$ such that for any other couples of
strategies $(\alpha,\beta)$, it holds
\begin{equation}\label{equ 5.2}
J_1(t,x,\bar{\alpha},\bar{\beta})\geq J_1(t,x,\alpha,\bar{\beta}),\ \text{and}\ J_2(t,x,\bar{\alpha},\bar{\beta})\geq J_2(t,x,\bar{\alpha},\beta),
\end{equation}
and the pair $(J_1(t,x,\bar{\alpha},\bar{\beta}),J_2(t,x,\bar{\alpha},\bar{\beta}))$ is called a Nash equilibrium payoff.

In our paper, we only concern the existence of the Nash equilibrium payoff which can be approximated by
$(J_1(t,x,\bar{\alpha}^\epsilon,\bar{\beta}^\epsilon),J_2(t,x,\bar{\alpha}^\epsilon,\bar{\beta}^\epsilon))$ when $\epsilon$ tends to $0$.
 Now we first give the definition of a Nash equilibrium payoff for our
nonzero-sum differential games.
\begin{definition}
A couple $(e_1,e_2)\in\mathbb{R}^2$ is called a Nash equilibrium payoff (NEP, for short) at the position $(t,x)$, if
for any $\epsilon>0$, there exists $\delta_\epsilon$ small enough satisfying that for any partition $\pi$ of the interval $[0,T]$ with $|\pi|\leq \delta_\epsilon$, there exist $(\alpha^\epsilon,\beta^\epsilon)\in\mathcal{A}_1
^\pi(t,T)\times\mathcal{B}_1^\pi(t,T)$ such that for all $(\alpha,\beta)\in\mathcal{A}_1^\pi(t,T)\times\mathcal{B}_1^\pi(t,T)$
\begin{equation}\label{equ 5.3}
J_1(t,x,\alpha^\epsilon,\beta^\epsilon)\geq J_1(t,x,\alpha,\beta^\epsilon)-\epsilon\ \text{and}\
 J_2(t,x,\alpha^\epsilon,\beta^\epsilon)\geq J_2(t,x,\alpha^\epsilon,\beta)-\epsilon,
\end{equation}
 and
 \begin{equation}\label{equ 5.4}
 \text{for}\ m=1,2,\ |J_m(t,x,\alpha^\epsilon,\beta^\epsilon)-e_m|\leq\epsilon,\ \text{respectively}.
 \end{equation}
\end{definition}
The following lemma gives an equivalent condition of assumption $(\ref{equ 5.3})$ which will be frequently
used in this section.
\begin{lemma}\label{le 5.1}
We assume $\epsilon>0$ and $(\alpha^\epsilon,\beta^\epsilon)\in\mathcal{A}_1^\pi(t,T)\times\mathcal{B}_1^\pi(t,T)$. Assumption (\ref{equ 5.3})
holds if and only if for any $(u,v)\in\mathcal{U}^{\pi,1}_{t,T}\times\mathcal{V}^{\pi,1}_{t,T}$,
\begin{equation}\label{equ 5.5}
J_1(t,x,\alpha^\epsilon,\beta^\epsilon)\geq J_1(t,x,u,\beta^\epsilon(u))-\epsilon\ \text{and}\
 J_2(t,x,\alpha^\epsilon,\beta^\epsilon)\geq J_2(t,x,\alpha^\epsilon(v),v)-\epsilon.
\end{equation}
\end{lemma}
\begin{proof}
We assume (\ref{equ 5.3}) holds, then for any fixed $u\in\mathcal{U}^{\pi,1}_{t,T}$, we define $\alpha(v)\equiv u$,
for all $v\in\mathcal{V}^{\pi,1}_{t,T}$, then we know $\alpha\in\mathcal{A}_1^\pi(t,T)$.
Thus, from condition (\ref{equ 5.3}),
we have $J_1(t,x,\alpha^\epsilon,\beta^\epsilon)\geq J_1(t,x,u,\beta^\epsilon(u))-\epsilon.$
Similarly, for any $v\in\mathcal{V}_{t,T}^{\pi,1}$, we obtain $J_2(t,x,\alpha^\epsilon,\beta^\epsilon)\geq J_2(t,x,\alpha^\epsilon(v),v)-\epsilon$, then
condition (\ref{equ 5.5}) holds.

Conversely now (\ref{equ 5.5}) holds, for any $\alpha\in\mathcal{A}_1^\pi(t,T)$, from Remark \ref{re 2.2}, there exists $(u,v)\in\mathcal{U}_{t,T}^{\pi,1}
\times\mathcal{V}_{t,T}^{\pi,1}$ such that, $\alpha(v)=u,\beta^\epsilon(u)=v$. Then we know
$$J_1(t,x,\alpha,\beta^\epsilon)-\epsilon =J_1(t,x,u,\beta^\epsilon(u))-\epsilon \leq J_1(t,x,\alpha^\epsilon,\beta^\epsilon).$$
Similarly to $J_2$, then we get condition (\ref{equ 5.3}).
\end{proof}
From Theorem \ref{th 3.2} we know the upper value function $W_1^{\pi}$ and the lower value function $V_1^{\pi}$ converge to
the same function without Isaacs condition. Thus we can denote the following functions $U_1(t,x)$ and $U_2(t,x)$:
\begin{equation}\label{equ 5.6}
U_1(t,x)=\lim_{|\pi|\rightarrow 0}\inf_{\beta\in\mathcal{B}_1^\pi(t,T)}\sup_{\alpha\in\mathcal{A}_1^\pi(t,T)}J_1(t,x,\alpha,\beta)
=\lim_{|\pi|\rightarrow 0}\sup_{\alpha\in\mathcal{A}_1^\pi(t,T)}\inf_{\beta\in\mathcal{B}_1^\pi(t,T)}J_1(t,x,\alpha,\beta),
\end{equation}
and similarly,
\begin{equation}\label{equ 5.7}
U_2(t,x)=\lim_{|\pi|\rightarrow 0}\sup_{\beta\in\mathcal{B}_1^\pi(t,T)}\inf_{\alpha\in\mathcal{A}_1^\pi(t,T)}J_2(t,x,\alpha,\beta)
=\lim_{|\pi|\rightarrow 0}\inf_{\alpha\in\mathcal{A}_1^\pi(t,T)}\sup_{\beta\in\mathcal{B}_1^\pi(t,T)}J_2(t,x,\alpha,\beta).
\end{equation}
Now we announce the following two important results for our nonzero-sum differential games.
\begin{theorem}\label{th 5.1}(Characterization)
A couple $(e_1,e_2)\in\mathbb{R}^2$ is a NEP at the position $(t,x)$ if and only
if for any $\epsilon>0$, there exists $\delta_\epsilon$ satisfying that for any partition $\pi=\{0=t_0<t_1<\cdots<t_N=T\}$ with $|\pi|<\delta_\epsilon$ and $t=t_{k-1}$, there exists $(u^\epsilon,v^\epsilon)
\in\mathcal{U}_{t,T}^{\pi,1}\times\mathcal{V}_{t,T}^{\pi,1}$ such that for $i=k,\ldots,N$ and $m=1,2$, respectively,
\begin{equation}\label{equ 5.8}
P\big\{E[g_m(X_T^{t,x,u^\epsilon,v^\epsilon})|\mathcal{F}_{t_{k-2},t_{i-2}}]\geq U_m(t_{i-1},X_{t_{i-1}}^{t,x,u^\epsilon,v^\epsilon})-\epsilon\big\}\geq 1-\epsilon,
\end{equation}
and
\begin{equation}\label{equ 5.9}
|E[g_m(X_T^{t,x,u^\epsilon,v^\epsilon})]-e_m|\leq \epsilon.
\end{equation}
\end{theorem}
\begin{theorem}\label{th 5.2}
For any initial position $(t,x)\in[0,T]\times\mathbb{R}^n$, there exists some NEP at the position $(t,x)$.
\end{theorem}
The rest of this section mainly gives the proof for the above theorems, we first prove Theorem \ref{th 5.1} and then from this characterization, we prove the existence result (Theorem \ref{th 5.2}).
First of all, we give the following lemma which will be used to prove Theorem \ref{th 5.1} and \ref{th 5.2}.
\begin{lemma}\label{le 5.2}
$a)$ Fix $(t,x)\in[0,T]\times\mathbb{R}^n$. For any $\epsilon>0$, for any partition $\pi=\{0=t_0<t_1<\cdots< t_N=T\}$
 with $|\pi|<\delta_\epsilon$ (small enough)  and $t=t_{k-1}$,
any fixed $u'\in\mathcal{U}_{t,T}^{\pi,1}$, there exist  strategies
 $\alpha^i\in\mathcal{A}_1^\pi(t,T)$, $i=k,\ldots,N$, such that for any $v\in\mathcal{V}_{t,T}^{\pi,1}$,
\begin{equation}\label{equ 5.10}
\begin{array}{l}
\alpha^i(v)\equiv u',\ P\text{-a.s.},\ \text{on}\ [t,t_{i-1}],\\
E[g_2(X_T^{t,x,\alpha^i(v),v})|\mathcal{F}_{t_{k-2},t_{i-2}}]\leq U_2(t_{i-1},X_{t_{i-1}}^{t,x,\alpha^i(v),v})+\epsilon,\ P\text{-a.s.}
\end{array}
\end{equation}
$b)$ Fix $(t,x)\in[0,T]\times\mathbb{R}^n$. For any $\epsilon>0$,
 for any partition $\pi=\{0=t_0<t_1<\cdots< t_N=T\}$
 with $|\pi|<\delta_\epsilon$ (small enough) and $t=t_{k-1}$,
any fixed $u'\in\mathcal{U}_{t,T}^{\pi,1}$, there exist strategies
$\alpha^i\in\mathcal{A}_1^\pi(t,T)$, $i=k,\ldots,N$,  such that for any $v\in\mathcal{V}_{t,T}^{\pi,1}$,
\begin{equation}\label{equ 5.11}
\begin{array}{l}
\alpha^i(v)\equiv u',\ P\text{-a.s.},\ \text{on}\ [t,t_{i-1}],\\
E[g_1(X_T^{t,x,\alpha^i(v),v})|\mathcal{F}_{t_{k-2},t_{i-2}}]\geq U_1(t_{i-1},X_{t_{i-1}}^{t,x,\alpha^i(v),v})-\epsilon,\ P\text{-a.s.}
\end{array}
\end{equation}
\end{lemma}
\begin{proof}
We just give the proof for $a)$, the proof of $b)$ is analogous.

For any $\epsilon>0$, $y\in\mathbb{R}^n$, any fixed $i$, from the definition of the value function $U_2$, there
exists a strategy $\alpha^i_y\in\mathcal{A}_1^\pi(t_{i-1},T)$ such that
\begin{eqnarray}\label{equ 5.12}
\begin{split}
&U_2(t_{i-1},y)
=\lim_{|\pi|\rightarrow 0}\inf_{\alpha\in\mathcal{A}_1^\pi(t_{i-1},T)}\sup_{\beta\in\mathcal{B}_1^\pi(t_{i-1},T)}
E[g_2(X_T^{t_{i-1},y,\alpha,\beta})]\\
&\geq \inf_{\alpha\in\mathcal{A}_1^\pi(t_{i-1},T)}\sup_{\beta\in\mathcal{B}_1^\pi(t_{i-1},T)}
E[g_2(X_T^{t_{i-1},y,\alpha,\beta})]-\frac{\epsilon}{4}\ \ (\text{since}\ |\pi|<\delta_\epsilon) \\
&\geq \inf_{\alpha\in\mathcal{A}_1^\pi(t_{i-1},T)}\sup_{v\in\mathcal{V}_{t_{i-1},T}^{\pi,1}}
E[g_2(X_T^{t_{i-1},y,\alpha(v),v})]-\frac{\epsilon}{4}
\geq \sup_{v\in\mathcal{V}_{t_{i-1},T}^{\pi,1}}
E[g_2(X_T^{t_{i-1},y,\alpha^i_y(v),v})]-\frac{\epsilon}{2}.
\end{split}
\end{eqnarray}
Since the coefficient $f$ is bounded, for any $(u,v)\in\mathcal{U}_{t,T}^{\pi,1}\times\mathcal{V}_{t,T}^{\pi,1}$,
there exists a constant $R>0$ such that $|X^{t,x,u,v}|\leq R$. Then there exists a finite partition $(O_l)_{l=1,2,\ldots,n}$
of the closed ball $\bar{B}_R(0)$ with diam$(O_l)\leq\epsilon\diagup (4C)$. For any $l$, from $(\ref{equ 5.12})$ there is some $y_l\in O_l$ with
\begin{equation}\label{equ 5.13}
\forall z\in O_l,\ \sup_{v\in\mathcal{V}_{t_{i-1},T}^{\pi,1}}E[g_2(X_T^{t_{i-1},z,\alpha^i_{y_l}(v),v})]\leq U_2(t_{i-1},z)+\epsilon,
\end{equation}
since $U_2(t_{i-1},z)$ and $\sup_{v\in\mathcal{V}_{t_{i-1},T}^{\pi,1}}E[g_2(X_T^{t_{i-1},z,\alpha^i_{y_l}(v),v})]$
are Lipschitz continuous with respect to $z$.

For any $v\in\mathcal{V}_{t,T}^{\pi,1}$, from the definition of the control $v$, it has the following form (refer to Remark \ref{re 2.2})
\begin{equation}\label{equ 5.14}
v(\omega,s)=v^k(s,\zeta_{2,k-1}^{\pi})I_{[t,t_k)}(s)+\sum_{l=k+1}^Nv^l(s,\zeta_{k-1}^{\pi},\ldots,\zeta_{l-2}^{\pi},\zeta_{2,l-1}^{\pi})I_{[t_{l-1},t_l)}(s),
\end{equation}
then for $s\in[t_{i-1},T]$ we define
\begin{eqnarray}\label{equ 5.15}
v''(\omega,Q,s)=\sum_{l=i}^Nv^l(s,Q,\zeta_{2,i-1}^{\pi},\zeta_{1,i-1}^{\pi},\zeta_{2,i}^{\pi},\ldots,
\zeta_{2,l-1}^{\pi})I_{[t_{l-1},t_l)}(s),
\end{eqnarray}
where $Q$ is a ${2(i-k)}$-dimensional constant vector.
Therefore, $v''\in\mathcal{V}_{t_{i-1},T}^{\pi,1}$, and we define the following strategy $\alpha^i$,
\begin{equation}\label{equ 5.16}
\forall v\in\mathcal{V}_{t,T}^{\pi,1},\ \alpha^i(v):=
\left\{
\begin{array}{ll}
u',&\text{on}\ [t,t_{i-1}],\\
\alpha^i_{y_l}(v''(Q_0,s)),&\text{on}\ (t_{i-1},T]\times\{X_{t_{i-1}}^{t,x,u',v}\in O_l\},
\end{array}
\right.
\end{equation}
where $Q_0=(\zeta_{2,k-1}^{\pi},\zeta_{1,k-1}^{\pi},\ldots,
\zeta_{2,i-2}^{\pi},\zeta_{1,i-2}^{\pi})$.
Then, $\alpha^i\in\mathcal{A}_1^\pi(t,T)$. Notice that
$Q_0$ and $X^{t,x,u',v}_{t_{i-1}}$ are all $\mathcal{F}_{t_{k-2},t_{i-2}}$-measurable,
$v''(Q,s)$ and $\alpha^i_{y_l}(v''(Q,s))$ are all $\mathcal{F}_{t_{i-2},T}$-measurable.
Therefore, for all $v\in\mathcal{V}_{t,T}^{\pi,1}$ from (\ref{equ 5.13}) we have, $P$-a.s.,
\begin{equation}\label{equ 5.18}
\begin{split}
&E[g_2(X_T^{t,x,\alpha^i(v),v})|\mathcal{F}_{t_{k-2},t_{i-2}}]=
\sum_{l=1}^nE[g_2(X_T^{t_{i-1},z,\alpha^i_{y_l}(v''(Q,s)),v''(Q,s)})]_{Q=Q_0,z=X^{t,x,u',v}_{t_{i-1}}}
\cdot I_{\{X_{t_{i-1}}^{t,x,u',v}\in O_l\}}\\
&\leq \sum_{l=1}^nU_2(t_{i-1},X_{t_{i-1}}^{t,x,u',v})\cdot I_{\{X_{t_{i-1}}^{t,x,u',v}\in O_l\}}+\epsilon
= U_2({t_{i-1},X_{t_{i-1}}^{t,x,\alpha^i(v),v}})+\epsilon.
\end{split}
\end{equation}
\end{proof}
Now with the help of Lemma \ref{le 5.2}, we will prove Theorem \ref{th 5.1}.
\begin{proof}
\emph{Sufficient condition}.

Let us assume that $(e_1,e_2)$ satisfies condition (\ref{equ 5.8}) and (\ref{equ 5.9}) of Theorem \ref{th 5.1}, namely,
for any $\epsilon>0$, there exists $\delta_\epsilon$ small enough satisfying that for  any partition $\pi=\{0=t_0<t_1<\cdots<t_N=T\}$ with $|\pi|<\delta_\epsilon$ and $t=t_{k-1}$, there exists $(u^{\epsilon},v^{\epsilon})
\in\mathcal{U}_{t,T}^{\pi,1}\times\mathcal{V}_{t,T}^{\pi,1}$ such that for $i=k,\ldots,N$ and $m=1,2$,
\begin{equation}\label{equ 5.19}
P\big\{E[g_m(X_T^{t,x,u^{\epsilon},v^{\epsilon}})|\mathcal{F}_{t_{k-2},t_{i-2}}]\geq U_m(t_{i-1},X_{t_{i-1}}^{t,x,u^{\epsilon},v^{\epsilon}})-\epsilon\big\}\geq 1-\epsilon,
\end{equation}
and
\begin{equation}\label{equ 5.20}
|E[g_m(X_T^{t,x,u^{\epsilon},v^{\epsilon}})]-e_m|\leq \epsilon.
\end{equation}
Then we will prove that $(e_1,e_2)$ is a NEP for the initial position $(t,x)$. For this, we construct
$(\alpha^\epsilon,\beta^\epsilon)\in\mathcal{A}_1^\pi(t,T)\times\mathcal{B}_1^\pi(t,T)$ satisfying $(\ref{equ 5.3})$ and $(\ref{equ 5.4})$.

 Since $g_m, m=1,2,$ is bounded, without loss of generality, we assume
$g_m\geq 0$, which means $W_m\geq 0$, for $m=1,2$, respectively. Suppose $\epsilon_0=\frac{\epsilon}{8+4NC}$ and $(\bar{u},\bar{v})=(u^{\epsilon_0},v^{\epsilon_0})$, then $(\ref{equ 5.19})$ and $\ref{equ 5.20}$ also hold for $\epsilon=\epsilon_0$.
From Lemma \ref{le 5.2} $a)$, let $u':=\bar{u}$,  there
exist  strategies $\alpha_i\in\mathcal{A}_1^\pi(t,T)$, $i=k,\ldots,N$, such that for any
$v\in\mathcal{V}_{t,T}^{\pi,1}$,
\begin{equation}\label{equ 5.21}
\begin{array}{l}
\alpha_i(v)\equiv \bar{u}, P\text{-a.s.},\ \text{on}\ [t,t_{i-1}],\\
E[g_2(X_T^{t,x,\alpha_i(v),v})|\mathcal{F}_{t_{k-2},t_{i-2}}]\leq U_2(t_{i-1},X_{t_{i-1}}^{t,x,\alpha_i(v),v})+\frac{\epsilon}{8},\ P\text{-a.s.}
\end{array}
\end{equation}
For any $v\in\mathcal{V}_{t,T}^{\pi,1}$, we introduce the stopping times $S^v=\inf\{s|v_{s}\neq \bar{v}_{s},\ t\leq s\leq T\}\wedge T$,
$\tau^v=\inf\{t_{i-1}|t_{i-1}>S^v, k+1\leq i\leq N\}\wedge T$.
Now we define $\alpha^\epsilon$ as follows:
\begin{equation}\label{equ 5.22}
\forall v\in\mathcal{V}_{t,T}^{\pi,1},\ \alpha^\epsilon(v)=
\left\{
\begin{array}{ll}
\bar{u},& \text{on}\ [[t,\tau^v]],\\
\alpha_i(v), & \text{on}\ (t_{i-1},T]\times\{\tau^v=t_{i-1}\}.
\end{array}
\right.
\end{equation}
Then $\alpha^\epsilon\in\mathcal{A}_1^{\pi}(t,T)$. Furthermore, for any $v\in\mathcal{V}_{t,T}^{\pi,1}$,
\begin{equation}\label{equ 5.23}
X^{t,x,\alpha^\epsilon(v),v}=
\left\{
\begin{array}{ll}
X^{t,x,\bar{u},v}, &\text{on}\ [[t,\tau^v]],\ P\text{-a.s.},\\
\sum_{i=k+1}^NX^{t,x,\alpha_i(v),v}\cdot I_{\{\tau^v=t_{i-1}\}},&  \text{on}\ [[\tau^v,T]],\ P\text{-a.s.}
\end{array}
\right.
\end{equation}
Then, since $\{\tau^v=t_{i-1}\}\in\mathcal{F}_{t_{k-2},t_{i-2}}$ from (\ref{equ 5.21}) we get
\begin{equation}\label{equ 5.24}
E[g_2(X_T^{t,x,\alpha^\epsilon(v),v})|\mathcal{F}_{t_{k-2},\tau^v}]\leq U_2(\tau^v,X_{\tau^v}^{t,x,\alpha^\epsilon(v),v})+\frac{\epsilon}{8},\ P\text{-a.s.}
\end{equation}
Taking expectation on both side we have
\begin{equation}\label{equ 5.25}
J_2(t,x,\alpha^\epsilon(v),v)\leq E[U_2(\tau^v,X_{\tau^v}^{t,x,\alpha^\epsilon(v),v})]+\frac{\epsilon}{8}.
\end{equation}
Since $X_{S^v}^{t,x,\alpha^\epsilon(v),v}=X_{S^v}^{t,x,\bar{u},\bar{v}}$ and the coefficient $f$ is bounded, for  $\rho:=|\pi|>0$, we have
$$E[\sup_{0\leq r\leq \rho}|X_{(S^v+r)\wedge T}^{t,x,\alpha^\epsilon(v),v}-X_{(S^v+r)\wedge T}^{t,x,\bar{u},\bar{v}}|]\leq C\rho.$$
Moreover, since $U_2(s,x)$ is Lipschitz in $x$, and $S^v\leq \tau^v\leq S^v+\rho$,  then we have
\begin{equation}\label{equ 999}
E[|U_2(\tau^v,X_{\tau^v}^{t,x,\alpha^\epsilon(v),v})-U_2(\tau^v,X_{\tau^v}^{t,x,\bar{u},\bar{v}})|]\leq C\rho\leq \frac{\epsilon}{8}.
\end{equation}
From (\ref{equ 5.25}) and (\ref{equ 999}), we have
\begin{equation}\label{equ 9999}
J_2(t,x,\alpha^\epsilon(v),v)\leq E[U_2(\tau^v,X_{\tau^v}^{t,x,\bar{u},\bar{v}})]+\frac{\epsilon}{4}.
\end{equation}
Now we denote
\begin{equation}\label{equ 5.26}
\Omega_{i}:=\big\{E[g_2(X_T^{t,x,\bar{u},\bar{v}})|\mathcal{F}_{t_{k-2},t_{i-2}}]\geq U_2(t_{i-1},X_{t_{i-1}}^{t,x,\bar{u},\bar{v}})-\epsilon_0\big\},
\end{equation}
and from (\ref{equ 5.19}), we have $P(\Omega_i)\geq 1-\epsilon_0$.
Thus, from (\ref{equ 9999}), (\ref{equ 5.26}) and (\ref{equ 5.20}),
we have
\begin{equation}\label{equ 5.27}
\begin{split}
&J_2(t,x,\alpha^\epsilon(v),v)\\
&\leq\sum_{i=k+1}^N E[U_2(t_{i-1},X_{t_{i-1}})\cdot I_{\{\tau^v=t_{i-1}\}}\cdot I_{\Omega_i}]+
\sum_{i=k+1}^N E[U_2(t_{i-1},X_{t_{i-1}})\cdot I_{\{\tau^v=t_{i-1}\}}\cdot I_{\Omega_{i}^c}]+\frac{\epsilon}{4}\\
&\leq \sum_{i=k+1}^N E[(E[g_2(X_T)|\mathcal{F}_{t_{k-2},t_{i-2}}]+\epsilon_0)\cdot I_{\{\tau^v=t_{i-1}\}}\cdot 1_{\Omega_{i}}]+
\sum_{i=k+1}^N CP(\Omega_{i}^c\cap\{\tau^v=t_{i-1}\})+\frac{\epsilon}{4}\\
& \leq E[g_2(X_T)]+\epsilon_0+\sum_{i=k+1}^NCP(\Omega_{i}^c)+\frac{\epsilon}{4}
 \leq e_2+(2+NC) \epsilon_0+\frac{\epsilon}{4}=e_2+\frac{\epsilon}{2},
\end{split}
\end{equation}
where $X_.:=X_.^{t,x,\bar{u},\bar{v}}$. Then
 from (\ref{equ 5.27}) and (\ref{equ 5.22}) we obtain
\begin{equation}\label{equ 5.28}
\forall v\in\mathcal{V}_{t,T}^{\pi,1},\ J_2(t,x,\alpha^\epsilon(v),v)\leq e_2+\frac{\epsilon}{2},\ \text{and}\
 \alpha^\epsilon(\bar{v})=\bar{u}.
\end{equation}
Similarly, we can construct $\beta^\epsilon\in\mathcal{B}_1^\pi(t,T)$ such that
\begin{equation}\label{equ 5.29}
\forall u\in\mathcal{U}_{t,T}^{\pi,1},\ J_1(t,x,u,\beta^\epsilon(u))\leq e_1+\frac{\epsilon}{2},\
\text{and}\ \beta^\epsilon(\bar{u})=\bar{v}.
\end{equation}
From (\ref{equ 5.28}), (\ref{equ 5.29}) and (\ref{equ 5.20}), we have, for $m=1,2,$ respectively,
\begin{equation}\label{equ 5.29.1}
|J_m(t,x,\alpha^\epsilon,\beta^\epsilon)-e_m|=|J_m(t,x,\bar{u},\bar{v})-e_m|\leq \frac{\epsilon}{2},
\end{equation}
namely, we obtain (\ref{equ 5.4}).
From (\ref{equ 5.29.1}) we know, for $m=1,2,$ respectively,
\begin{equation}\label{equ 5.29.2}
e_m\leq J_m(t,x,\alpha^\epsilon,\beta^\epsilon)+\frac{\epsilon}{2}.
\end{equation}
From (\ref{equ 5.28}), (\ref{equ 5.29}) and (\ref{equ 5.29.2}), we have
\begin{eqnarray*}
J_2(t,x,\alpha^\epsilon(v),v)\leq e_2+\frac{\epsilon}{2}\leq J_2(t,x,\alpha^\epsilon,\beta^\epsilon)+\epsilon,\\
J_1(t,x,u,\beta^\epsilon(u))\leq e_1+\frac{\epsilon}{2}\leq J_1(t,x,\alpha^\epsilon,\beta^\epsilon)+\epsilon.
\end{eqnarray*}
From Lemma \ref{le 5.1}, we know (\ref{equ 5.3}) holds.\\

\noindent\emph{Necessary condition}.

We assume there exists a NEP $(e_1,e_2)\in\mathbb{R}^2$ at the position $(t,x)$, i.e., for any $\epsilon>0$,
there exists $\delta_\epsilon$ small enough satisfying that for any  partition
$\pi=\{0=t_0<\ldots<t_N=T\}$ with $|\pi|<\delta_\epsilon$ and $t=t_{k-1}$, there exists $(\alpha^\epsilon,\beta^\epsilon)\in\mathcal{A}_1^\pi(t,T)\times\mathcal{B}_1^\pi(t,T)$ be
such that for any $(u,v)\in\mathcal{U}^{\pi,1}_{t,T}\times\mathcal{V}^{\pi,1}_{t,T}$, the following inequalities hold:
\begin{equation}\label{equ 5.30}
J_1(t,x,\alpha^\epsilon,\beta^\epsilon)\geq J_1(t,x,u,\beta^\epsilon(u))-\frac{\epsilon^2}{2}\ \text{and}\
 J_2(t,x,\alpha^\epsilon,\beta^\epsilon)\geq J_2(t,x,\alpha^\epsilon(v),v)-\frac{\epsilon^2}{2},
\end{equation}
and for $m=1,2$, respectively,
\begin{equation}\label{equ 5.31}
 |J_m(t,x,\alpha^\epsilon,\beta^\epsilon)-e_m|\leq \frac{\epsilon^2}{2}.
\end{equation}
From Remark \ref{re 2.1}, we know there exist $(u^\epsilon,v^\epsilon)\in\mathcal{U}_{t,T}^{\pi,1}\times\mathcal{V}_{t,T}^{\pi,1}$,
such that $\alpha^\epsilon(v^\epsilon)=u^\epsilon$, $\beta^\epsilon(u^\epsilon)=v^\epsilon$. Now we see (\ref{equ 5.9}) holds obviously.
We suppose (\ref{equ 5.8}) doesn't hold, then we assume that there is some $j\in\{k,\ldots,N\}$, without loss of generality,
we consider the case $m=1$ such that
\begin{equation}\label{equ 5.32}
P\big\{E[g_1(X_T^{t,x,u^\epsilon,v^\epsilon})|\mathcal{F}_{t_{k-2},t_{j-2}}]<U_1(t_{j-1},X_{t_{j-1}}^{t,x,u^\epsilon,v^\epsilon})-\epsilon\big\}>\epsilon.
\end{equation}
Define
\begin{equation}\label{equ 5.33}
A=\big\{E[g_1(X_T^{t,x,u^\epsilon,v^\epsilon})|\mathcal{F}_{t_{k-2},t_{j-2}}]<U_1(t_{j-1},X_{t_{j-1}}^{t,x,u^\epsilon,v^\epsilon})-\epsilon\big\}.
\end{equation}
From Lemma \ref{le 5.2} $b)$, let $u':=u^\epsilon\in\mathcal{U}_{t,T}^{\pi,1}$, then
there exist a strategy $\alpha\in\mathcal{A}_1^\pi(t,T)$ such that,
for any $v\in\mathcal{V}_{t,T}^{\pi,1}$, $\alpha(v)=u^\epsilon$, on $[t,t_{j-1}]$, $P$-a.s., and
\begin{equation}\label{equ 5.34}
 E[g_1(X_T^{t,x,\alpha(v),v})|\mathcal{F}_{t_{k-2},t_{j-2}}]\geq U_1(t_{j-1},X_{t_{j-1}}^{t,x,\alpha(v),v})-\frac{\epsilon}{2}.
\end{equation}
 For $(\alpha,\beta^\epsilon)\in\mathcal{A}_1^\pi(t,T)\times\mathcal{B}_1^\pi(t,T)$, there exists a pair  $(u,v)\in\mathcal{U}_{t,T}^{\pi,1}\times\mathcal{V}_{t,T}^{\pi,1}$,
 such that $\alpha(v)=u$, $\beta^\epsilon(u)=v$. Notice that $u\equiv u^\epsilon, v\equiv v^\epsilon$, on $[t,t_{j-1}]$.
Define $\bar{u}$ by setting:
$$\bar{u}=
\left\{
\begin{array}{ll}
u^\epsilon,&\text{on}\ \big([t,t_{j-1}]\times\Omega\big)\cup\big([t_{j-1},T]\times A^c\big),\\
u,&\text{on}\ [t_{j-1},T]\times A.
\end{array}
\right.
$$
Obviously, $\bar{u}\in\mathcal{U}_{t,T}^{\pi}$.\\ And we know $\beta^\epsilon(\bar{u})\equiv v^\epsilon$,
on $[t,t_{j-1})$, and for $s\in[t_{j-1},T]$, $\beta^\epsilon(\bar{u})_s=\left\{
\begin{array}{ll}
v_s,& \text{on}\ A,\\
v_s^\epsilon,& \text{on}\ A^c.
\end{array}
\right.$\\
Then, we have $X^{t,x,\bar{u},\beta^\epsilon(\bar{u})}\equiv X^{t,x,u^\epsilon,v^\epsilon}$, on $[t,t_{j-1}]$.
For $s\in[t_{j-1},T]$, $X_s^{t,x,\bar{u},\beta^\epsilon(\bar{u})}=\left\{
\begin{array}{ll}
X_s^{t,x,\alpha(v),v}, &\text{on}\ A,\\
X_s^{t,x,u^\epsilon,v^\epsilon}, &\text{on}\ A^c.
\end{array}
\right.
$
Furthermore, we have
\begin{equation}\label{equ 5.35}
\begin{split}
J_1(t,x,\bar{u},\beta^\epsilon(\bar{u}))&=E[g_1(X_T^{t,x,u^\epsilon,v^\epsilon})\cdot I_{A^c}]
+E[g_1(X_T^{t,x,\alpha(v),v})\cdot I_{A}]\\
&=E[g_1(X_T^{t,x,u^\epsilon,v^\epsilon})\cdot I_{A^c}]
+E[E[g_1(X_T^{t,x,\alpha(v),v})|\mathcal{F}_{t_{k-2},t_{j-2}}]\cdot I_{A}]\\
&\geq E[g_1(X_T^{t,x,u^\epsilon,v^\epsilon})\cdot I_{A^c}]
+E[U_1(t_{j-1},X_{t_{j-1}}^{t,x,\alpha(v),v})\cdot I_{A}]-\frac{\epsilon}{2}P(A)\ \ (\text{from}\ (\ref{equ 5.34}))\\
&\geq  E[g_1(X_T^{t,x,u^\epsilon,v^\epsilon})]+\frac{\epsilon}{2}P(A)\ \ (\text{from}\ (\ref{equ 5.33}))\\
&>J_1(t,x,\alpha^\epsilon,\beta^\epsilon)+\frac{\epsilon^2}{2},\ \ (\text{from}\ (\ref{equ 5.32})\ \text{and}\ (\ref{equ 5.33}) )
\end{split}
\end{equation}
which is in contradiction with (\ref{equ 5.30}). Therefore, (\ref{equ 5.8}) holds.
\end{proof}
To prove Theorem \ref{th 5.2} we only need to prove that for any $\epsilon>0$, there exists $\delta_\epsilon$ small enough satisfying that for any partition $\pi=\{0=t_0<t_1<\cdots< t_N=T\}$
 with $|\pi|<\delta_\epsilon$ and $t=t_{k-1}$, there is a pair $(u^\epsilon,v^\epsilon)$
satisfying the conditions of Theorem \ref{th 5.1}. For this we show a stronger result.
\begin{proposition}\label{prop 5.1}
For any $\epsilon>0$, there exists $\delta_\epsilon$ small enough satisfying that for any partition $\pi=\{0=t_0<t_1<\cdots< t_N=T\}$
 with $|\pi|<\delta_\epsilon$ and $t=t_{k-1}$,
there exist a pair $(u^\epsilon,v^\epsilon)\in\mathcal{U}_{t,T}^{\pi,1}\times\mathcal{V}_{t,T}^{\pi,1}$,
such that, for any $k\leq i\leq l\leq N$, and $m=1,2$, respectively,
\begin{equation}\label{equ 5.36}
E[U_m(t_l,X_{t_l})|\mathcal{F}_{t_{k-2},t_{i-2}}]
\geq U_m(t_{i-1},X_{t_{i-1}})-\epsilon,\ P\text{-a.s.},
\end{equation}
where $X_.=X_.^{t,x,u^\epsilon,v^\epsilon}$.
\end{proposition}
\begin{remark}\label{re 5.1}
If Proposition \ref{prop 5.1} holds, then we set $l=N$, we have
$U_m(T,x)=g_m(x)$, i.e., $U_m(T,X_T^{t,x,u^\epsilon,v^\epsilon})=g_m(X_T^{t,x,u^\epsilon,v^\epsilon})$, then we know
the pair $(u^\epsilon,v^\epsilon)$
satisfy the conditions of Theorem \ref{th 5.1}, let $\epsilon\rightarrow 0$, we obtain the NEP $(e_1,e_2)$.
\end{remark}
We first give the following lemma.
\begin{lemma}\label{le 5.3}
For any $\epsilon>0$, there exists $\delta_\epsilon$ small enough satisfying that for any partition $\pi=\{0=t_0<t_1<\cdots< t_N=T\}$
 with $|\pi|<\delta_\epsilon$ and $t=t_{k-1}$, there exists a pair $(u^\epsilon,v^\epsilon)\in\mathcal{U}_{t,T}^{\pi,1}\times\mathcal{V}_{t,T}^{\pi,1}$,
such that, for $m=1,2$, respectively,
\begin{equation}\label{equ 5.37}
E[U_m(t_k,X_{t_k}^{t,x,u^\epsilon,v^\epsilon})]\geq U_m(t,x)-\epsilon.
\end{equation}
\end{lemma}
\begin{proof}
From the definition of $U_1(t,x)$ and $U_2(t,x)$ (refer to (\ref{equ 5.6}) and (\ref{equ 5.7})),
there is some $\delta_\epsilon$ such that when $|\pi|<\delta_\epsilon$
$$U_1(t,x)=\lim_{|\pi|\rightarrow 0}\sup_{\alpha\in\mathcal{A}_1^\pi(t,T)}
\inf_{\beta\in\mathcal{B}_1^\pi(t,T)}J_1(t,x,\alpha,\beta)\leq \sup_{\alpha\in\mathcal{A}_1^\pi(t,T)}
\inf_{\beta\in\mathcal{B}_1^\pi(t,T)}J_1(t,x,\alpha,\beta)+\frac{\epsilon}{4}, $$
$$U_2(t,x)=\lim_{|\pi|\rightarrow 0}\sup_{\beta\in\mathcal{B}_1^\pi(t,T)}
\inf_{\alpha\in\mathcal{A}_1^\pi(t,T)}J_2(t,x,\alpha,\beta)\leq \sup_{\beta\in\mathcal{B}_1^\pi(t,T)}
\inf_{\alpha\in\mathcal{A}_1^\pi(t,T)}J_2(t,x,\alpha,\beta)+\frac{\epsilon}{4}.$$
Then we choose $\alpha^\epsilon\in\mathcal{A}_1^\pi(t,T)$ and $\beta^\epsilon\in\mathcal{B}_1^\pi(t,T)$
such that
\begin{equation}\label{equ 5.38}
\begin{split}
U_1(t,x)\leq \inf_{\beta\in\mathcal{B}_1^\pi(t,T)}J_1(t,x,\alpha^\epsilon,\beta)+\frac{\epsilon}{2}\leq
\inf_{v\in\mathcal{V}_{t,T}^{\pi,1}}J_1(t,x,\alpha^\epsilon(v),v)+\frac{\epsilon}{2},\\
U_2(t,x)\leq \inf_{\alpha\in\mathcal{A}_1^\pi(t,T)}J_2(t,x,\alpha,\beta^\epsilon)+\frac{\epsilon}{2}\leq
\inf_{u\in\mathcal{U}_{t,T}^{\pi,1}}J_2(t,x,u,\beta^\epsilon(u))+\frac{\epsilon}{2}.
\end{split}
\end{equation}
For $(\alpha^\epsilon,\beta^\epsilon)$, from Remark \ref{re 2.1} there exists a unique pair
$(u^\epsilon,v^\epsilon)$ such that, $\alpha^\epsilon(v^\epsilon)=u^\epsilon,\
\beta^\epsilon(u^\epsilon)=v^\epsilon$.\\
Now we want to prove that $(u^\epsilon,v^\epsilon)$ satisfy (\ref{equ 5.37}). For this, we suppose (\ref{equ 5.37})
doesn't hold, i.e., for $m=2$ ($m=1$, similar)
such that
\begin{equation}\label{equ 5.39}
E[U_2(t_k,X_{t_k}^{t,x,u^\epsilon,v^\epsilon})]<U_2(t,x)-\epsilon.
\end{equation}
From Lemma \ref{le 5.2} $a)$, for $u':=u^\epsilon$, there exists  a NAD strategy $\alpha\in\mathcal{A}_1
^\pi(t,T)$, for any $v\in\mathcal{V}_{t,T}^{\pi,1}$,  such that $\alpha(v)=u^\epsilon$, $P$-a.s.,
on $[t,t_k]$, and
\begin{equation}\label{equ 5.40}
E[g_2(X_T^{t,x,\alpha(v),v})|\mathcal{F}_{t_{k-2},t_{k-1}}]\leq U_2(t_k,X_{t_k}^{t,x,\alpha(v),v})+\frac{\epsilon}{2},\ P\text{-a.s.}
\end{equation}
From Remark \ref{re 2.1} we know there exists a couple $(\bar{u},\bar{v})\in\mathcal{U}_{t,T}^{\pi,1}\times\mathcal{V}_{t,T}^{\pi,1}$
such that, $\alpha(\bar{v})=\bar{u}$, $\beta^\epsilon(\bar{u})=\bar{v}$. Since
$\bar{u}\equiv u^\epsilon, \bar{v}\equiv v^\epsilon$, on $[t,t_k]$, we know
$X_{t_k}^{t,x,\bar{u},\bar{v}}=X_{t_k}^{t,x,\alpha(\bar{v}),\bar{v}}=X_{t_k}^{t,x,u^{\epsilon},v^\epsilon}$, $P$-a.s.\\
From (\ref{equ 5.40}) and (\ref{equ 5.39}), it follows that
\begin{equation}\label{equ 5.41}
\begin{split}
&J_2(t,x,\bar{u},\beta^\epsilon(\bar{u}))=J_2(t,x,\alpha(\bar{v}),\bar{v})=E[E[g_2(X_T^{t,x,\alpha(\bar{v}),\bar{v}})|\mathcal{F}_{t_{k-2},t_{k-1}}]]\\
\leq& E[U_2({t_k},X_{t_k}^{t,x,\alpha(\bar{v}),\bar{v}})]+\frac{\epsilon}{2}
<U_2(t,x)-\frac{\epsilon}{2},
\end{split}
\end{equation}
which is contradictory to (\ref{equ 5.38}). Hence, (\ref{equ 5.37}) holds.
\end{proof}
We now give the proof of Proposition \ref{prop 5.1}.
\begin{proof}
Firstly, we show that when $l=i$, Proposition \ref{prop 5.1} holds. \\
Similar to Lemma \ref{le 5.3}, we know for any $\epsilon>0$, there exists $\delta_\epsilon$ small enough satisfying that for any partition $\pi=\{0=t_0<t_1<\cdots< t_N=T\}$
 with $|\pi|<\delta_\epsilon$ and $t=t_{k-1}$, for any $y\in\mathbb{R}^n$ there exist
 $(u_j^{\epsilon,y},v_j^{\epsilon,y})\in\mathcal{U}_{t_j,T}^{\pi,1}\times\mathcal{V}_{t_j,T}^{\pi,1}$, $j=k-1,\ldots,N-1$,
such that for $m=1,2$, respectively,
\begin{equation}\label{equ 4.1.23.1}
E[U_m(t_{j+1},X_{t_{j+1}}^{t_j,y,u_j^{\epsilon,y},v_j^{\epsilon,y}})]\geq U_m(t_j,y)-\epsilon.
\end{equation}
For the partition $\pi$ with $|\pi|<\delta_\epsilon$, we now give the construction of $(u^\epsilon,v^\epsilon)\in\mathcal{U}_{t,T}^{\pi,1}\times\mathcal{V}_{t,T}^{\pi,1}$ by induction on $[t_{i-1},t_i)$
satisfying, for $i=k,\ldots,N$,
\begin{equation}\label{equ 4.1.23.2}
E[U_m(t_i,X_{t_i}^{t,x,u^\epsilon,v^\epsilon})|\mathcal{F}_{t_{k-2},t_{i-2}}]
\geq U_m(t_{i-1},X_{t_{i-1}}^{t,x,u^\epsilon,v^\epsilon})-\epsilon,\ P\text{-a.s.}.
\end{equation}
For $i=k$, from (\ref{equ 4.1.23.1}) we know there is $(u_{k-1}^{\epsilon,x},v_{k-1}^{\epsilon,x})$ satisfying (\ref{equ 4.1.23.2}). We
define $u^\epsilon|_{[t_{k-1},t_k)}:=u_{k-1}^{\epsilon,x}$, $v^\epsilon|_{[t_{k-1},t_k)}:=v_{k-1}^{\epsilon,x}$. \\
For $i=k+1$, from (\ref{equ 4.1.23.1}) we know for any $y\in\mathbb{R}^n$, there is $(u_k^{\epsilon,y},v_k^{\epsilon,y})\in\mathcal{U}_{t_k,T}^{\pi,1}\times\mathcal{V}_{t_k,T}^{\pi,1}$ such that, for $m=1,2$, respectively,
 \begin{equation}\label{equ 4.1.23.3}
E[U_m(t_{k+1},X_{t_{k+1}}^{t_k,y,u_k^{\epsilon,y},v_k^{\epsilon,y}})]\geq U_m(t_k,y)-\frac{\epsilon}{2}.
\end{equation}
Since the coefficient $f$ is bounded, there is some $R>0$ such that $|X_{t_k}^{t,x,u^\epsilon,v^\epsilon}|<R$. Then
there exists a finite Borel partition $(O_l)_{l=1}^n$ of $\bar{B}_R(0)$. From (\ref{equ 4.1.23.3}),
we have for any $z\in O_l$, there is some $y_l\in O_l$, such that
 \begin{equation}\label{equ 4.1.23.4}
E[U_m(t_{k+1},X_{t_{k+1}}^{t_k,z,u_k^{\epsilon,y_l},v_k^{\epsilon,y_l}})]\geq U_m(t_k,z)-\epsilon.
\end{equation}
Now we define $u^\epsilon|_{[t_{k},t_{k+1})}:=\sum\limits_{l=1}^n u_{k}^{\epsilon,y_l} I_{\{X^{t,x,u^\epsilon,v^\epsilon}_{t_k}\in O_l\}}$, $v^\epsilon|_{[t_{k},t_{k+1})}:=\sum\limits_{l=1}^nv_{k}^{\epsilon,y_l} I_{\{X^{t,x,u^\epsilon,v^\epsilon}_{t_k}\in O_l\}}$. Then from (\ref{equ 4.1.23.4}) we have
\begin{equation}\label{equ 4.1.23.5}
\begin{split}
&E[U_m(t_{k+1},X_{t_{k+1}}^{t,x,u^{\epsilon},v^{\epsilon}})|\mathcal{F}_{t_{k-2},t_{k-1}}]=
\sum\limits_{l=1}^nE[U_m(t_{k+1},X_{t_{k+1}}^{t_k,z,u_k^{\epsilon,y_l},v_k^{\epsilon,y_l}})]_{z=X^{t,x,u^\epsilon,v^\epsilon}_{t_k}}
I_{\{X^{t,x,u^\epsilon,v^\epsilon}_{t_k}\in O_l\}}\\
&\geq \sum\limits_{l=1}^n[U_m(t_k,z)-\epsilon]_{z=X^{t,x,u^\epsilon,v^\epsilon}_{t_k}} I_{\{X^{t,x,u^\epsilon,v^\epsilon}_{t_k}\in O_l\}}=U_m(t_k,X_{t_k}^{t,x,u^\epsilon,v^\epsilon})-\epsilon,\ P\text{-a.s.}
\end{split}
\end{equation}
Repeating the above step, we can get $(u^\epsilon,v^\epsilon)\in\mathcal{U}_{t,T}^{\pi,1}\times\mathcal{V}_{t,T}^{\pi,1}$ satisfying (\ref{equ 4.1.23.2}).

Next for $l>i$, from (\ref{equ 4.1.23.2}) with using $\epsilon:=\frac{\epsilon}{N}$ we get
\begin{equation}\label{equ 4.1.23.6}
\begin{split}
&E[U_m(t_{l},X_{t_{l}}^{t,x,u^{\epsilon},v^{\epsilon}})|\mathcal{F}_{t_{k-2},t_{i-2}}]=
E[E[U_m(t_{l},X_{t_{l}}^{t,x,u^{\epsilon},v^{\epsilon}})|\mathcal{F}_{t_{k-2},t_{l-2}}]|\mathcal{F}_{t_{k-2},t_{i-2}}]\\
&\geq E[U_m(t_{l-1},X_{t_{l-1}}^{t,x,u^{\epsilon},v^{\epsilon}})|\mathcal{F}_{t_{k-2},t_{i-2}}]-\frac{\epsilon}{N}
\cdots=U_m(t_{i-1},X_{t_{i-1}}^{t,x,u^\epsilon,v^\epsilon})-\epsilon,\ P\text{-a.s.}
\end{split}
\end{equation}
\end{proof}

\section{ {\protect \large  Characterization for the functions $W(t,x,p,q)$ and $V(t,x,p,q)$}}
This section mainly gives a characterization for $W(t,x,p,q)$ and $V(t,x,p,q)$. Under some equivalent Isaacs
condition, we prove that $W(t,x,p,q)=V(t,x,p,q)$. This characterization guarantees that we can
consider the discrete case (with the strategies along the partition $\pi$) for some indiscrete zero-sum
differential games with asymmetric information. With this property, we provide a new method to calculate the value
of the zero-sum differential games through considering all the partitions.

For simplicity, we only consider the case that
 Player \Rmnum {1} and \Rmnum {2} have no private  information in a small time from beginning, then they observe each other
  and only know the opponent's probability, i.e., along with
the partition $\pi=\{0=t_0<t_1<\ldots<t_N=T\}$, $t_{k-1}\leq t<t_k$, the strategy $\alpha:\Omega\times[t,T]\times\mathcal{V}_{t,T}\rightarrow\mathcal{U}_{t,T}$
of Player \Rmnum {1} has the following form
$$\alpha(\omega,v)(s)=\alpha_k(v)(s)I_{[t,t_k)}(s)+
\sum_{l=k+1}^N\alpha_l((\zeta_{1,k}^\pi,\zeta_{2,k}^\pi,\zeta_{1,k+1}^\pi,\ldots,\zeta_{1,l-1}^\pi)(\omega),v)I_
{[t_{l-1},t_l)}(s),\ s\in[t,T],$$
where $\alpha_k: [t,t_k)\times\mathcal{V}_{t,T}\mapsto\mathcal{U}_{t,T}$, $\alpha_l: \mathbb{R}^{2(l-k)-1}\times[t_{l-1},t_l)\times\mathcal{V}_{t,T}\mapsto\mathcal{U}_{t,T},\ k+1\leq l\leq N$, are  Borel measurable functions satisfying:
For all $v,v'\in\mathcal{V}_{t,T}$, it holds that, whenever $v=v'$ a.e. on $[t,t_{l-1}]$, we have for all $x\in\mathbb{R}^{2(l-k)-1}$, $\alpha_l(x,v)(s)=
\alpha_l(x,v')(s)$, a.e. on $[t_{l-1},t_l],\  k+1\leq l\leq N$.\\ Obviously, the strategy $\alpha$ such defined is a special case of Definition \ref{def 2.2},
still denoted by $\mathcal{A}^\pi(t,T)$ for the set of the strategy $\alpha$ that have the above form. Similarly, we have the definition for the strategy $\beta$ and for the set we still
denoted by $\mathcal{B}^\pi(t,T)$. Obviously, for $\pi'\subset\pi$, we have $\mathcal{A}^{\pi'}(t,T)\subset\mathcal{A}^\pi(t,T)$.
$\mathcal{A}(t,T)$ and $\mathcal{B}(t,T)$ are the union of $\mathcal{A}^\pi(t,T)$ and $\mathcal{B}^\pi(t,T)$ with all partition $\pi$, respectively.
It is noticed that the strategies used in this section have the above forms, the rest corresponding definitions are the
same with that defined in Section 2.

To give the characterization, we introduce the following upper and lower value functions as follows
\begin{eqnarray}
\bar{W}^\pi(t,x,p,q)&=&\inf_{\hat{\alpha}\in(\mathcal{A}^\pi(t,T))^I}\sup_{\hat{\beta}\in(\mathcal{B}(t,T))^J}
J(t,x,\hat{\alpha},\hat{\beta},p,q),\\
\bar{V}^\pi(t,x,p,q)&=&\sup_{\hat{\beta}\in(\mathcal{B}(t,T))^J}\inf_{\hat{\alpha}\in(\mathcal{A}^\pi(t,T))^I}
J(t,x,\hat{\alpha},\hat{\beta},p,q),\\
\bar{\bar{W}}^\pi(t,x,p,q)&=&\inf_{\hat{\alpha}\in(\mathcal{A}(t,T))^I}\sup_{\hat{\beta}\in(\mathcal{B}^\pi(t,T))^J}
J(t,x,\hat{\alpha},\hat{\beta},p,q),\\
\bar{\bar{V}}^\pi(t,x,p,q)&=&\sup_{\hat{\beta}\in(\mathcal{B}^\pi(t,T))^J}\inf_{\hat{\alpha}\in(\mathcal{A}(t,T))^I}
J(t,x,\hat{\alpha},\hat{\beta},p,q).
\end{eqnarray}
Next we first prove $(\bar{W}^\pi(t,x,p,q),\bar{V}^\pi(t,x,p,q))$ and
$(\bar{\bar{W}}^\pi(t,x,p,q),\bar{\bar{V}}^\pi(t,x,p,q))$ converge
uniformly on compacts to the same couple $(U(t,x,p,q),U(t,x,p,q))$,
as $|\pi|\rightarrow0$, under the condition
\begin{equation}\label{equ 4.1}
\inf_{u\in U}\sup_{\nu\in\mathcal{P}(V)}f(x,u,\nu)\cdot \xi=\sup_{\nu\in\mathcal{P}(V)}\inf_{u\in U}
f(x,u,\nu)\cdot \xi,
\end{equation}
\begin{equation}\label{equ 4.2}
\sup_{v\in V}\inf_{\mu\in\mathcal{P}(U)}f(x,\mu,v)\cdot \xi=\inf_{\mu\in\mathcal{P}(U)}\sup_{v\in V}
f(x,\mu,v)\cdot \xi,
\end{equation}
respectively, where $f(x,\mu,v):=\int_Uf(x,u,v)\mu(du)$, $f(x,u,\nu):=\int_Vf(x,u,v)\nu(dv)$,
and the function $U(t,x,p,q)$ is the unique solution of the HJI equation (\ref{equ 3.34}).
Then we show that the functions $W(t,x,p,q)=U(t,x,p,q)=V(t,x,p,q)$ under the conditions (\ref{equ 4.1}) and (\ref{equ 4.2}).
\begin{remark}\label{re 4.1}
The assumptions $(\ref{equ 4.1})$ and $(\ref{equ 4.2})$ hold, if and only if the following classical Isaacs condition holds:
\begin{equation}\label{equ 4.3}
 \inf\limits_{u\in U}\sup\limits_{v\in V}f(x,u,v)\cdot\xi=\sup\limits_{v\in V}\inf\limits_{u\in U}f(x,u,v)\cdot\xi.
\end{equation}
Indeed, we have
$$\inf_{\mu\in\mathcal{P}(U)}f(x,\mu,v)\cdot\xi=\inf_{\mu\in\mathcal{P}(U)}
\int_{U}f(x,u,v)\cdot\xi d\mu(u)\geq \inf_{u\in U}
f(x,u,v)\cdot\xi\geq \inf_{\mu\in\mathcal{P}(U)}f(x,\mu,v)\cdot\xi,$$
hence, $\inf\limits_{\mu\in\mathcal{P}(U)}f(x,\mu,v)
\cdot\xi=\inf\limits_{u\in U}f(x,u,v)\cdot\xi$.
 Similarly,
$\sup\limits_{\nu\in\mathcal{P}(V)}f(x,u,\nu)\cdot\xi=\sup\limits_{v\in V}
f(x,u,v)\cdot\xi$.\\
If $(\ref{equ 4.1})$ and $(\ref{equ 4.2})$ hold, then we have
$$\inf_{u\in U}\sup_{v\in V}f(x,u,v)\cdot \xi=
\inf_{u\in U}\sup_{\nu\in\mathcal{P}(V)}f(x,u,\nu)\cdot \xi=
\sup_{\nu\in\mathcal{P}(V)}\inf_{u\in U}f(x,u,\nu)\cdot \xi=
\sup_{\nu\in\mathcal{P}(V)}\inf_{\mu\in \mathcal{P}(U)}f(x,\mu,\nu)\cdot \xi.$$
$$\sup_{v\in V}\inf_{u\in U}f(x,u,v)\cdot \xi=
\sup_{v\in V}\inf_{\mu\in\mathcal{P}(U)}f(x,\mu,v)\cdot \xi
=\inf_{\mu\in\mathcal{P}(U)}\sup_{v\in V}f(x,\mu,v)\cdot \xi=
\inf_{\mu\in\mathcal{P}(U)}\sup_{\nu\in \mathcal{P}(V)}f(x,\mu,\nu)\cdot \xi.$$
Then, we get classical Isaacs condition $(\ref{equ 4.3})$ holds.\\
If $(\ref{equ 4.3})$ holds, then we have
$$\inf_{u\in U} \sup_{\nu\in\mathcal{P}(V)} f(x,u,\nu)\cdot \xi=\inf_{u\in U}\sup_{v\in V}f(x,u,v)\cdot \xi=
\sup_{v\in V}\inf_{u\in U}f(x,u,v)\cdot \xi=\sup_{\nu\in\mathcal{P}(V)}\inf_{u\in U}f(x,u,\nu)\cdot \xi,$$
then we know $(\ref{equ 4.1})$ holds.
Similarly, we get $(\ref{equ 4.2})$.
\end{remark}
Similar to the proof of Theorem \ref{th 3.2}, we get the following result.
\begin{theorem}\label{th 5.1.24.1}
The functions $(V^{\pi_n})$ and $(W^{\pi_n})$ converge uniformly
on compacts to a same Lipschitz continuous function $U$ when the mesh of the partition $\pi_n$ tends to $0$. Moreover, the function $U$ is  the unique dual viscosity  solution
of the HJI equation $(\ref{equ 3.34})$.
\end{theorem}
Now we only prove the convergence of $(\bar{W}^\pi(t,x,p,q),\bar{V}^\pi(t,x,p,q))$ and the proof of
$(\bar{\bar{W}}^\pi$ $(t,x,p,q),\bar{\bar{V}}^\pi(t,x,p,q))$ is similar.

Notice that for any fixed $\beta\in\mathcal{B}(t,T)$, there exists some partition $\bar{\pi}$ such that
$\beta\in\mathcal{B}^{\bar{\pi}}(t,T)$. Using this technique and the  method which have been used in Section 3, we have
the following lemmas.
\begin{lemma}\label{le 4.1}
The functions $\bar{W}^\pi$ and $\bar{V}^\pi$ are Lipschitz continuous with respect to $(t,x,p,q)$, uniformly with respect to $\pi$.
\end{lemma}
\begin{lemma}\label{le 4.2}
For any $(t,x)\in[0,T]\times\mathbb{R}^n$, the functions $\bar{W}^\pi(t,x,p,q)$ and $\bar{V}^\pi(t,x,p,q)$
are convex in $p$ and concave in $q$ on $\Delta(I)\times\Delta(J)$, respectively.
\end{lemma}
\begin{lemma}\label{le 4.3}
For all $(t,x,\bar{p},q)\in[0,T]\times\mathbb{R}^n\times\mathbb{R}^I\times\Delta(J)$, we have
\begin{equation}\label{equ 4.4}
\bar{V}^{\pi*}(t,x,\bar{p},q)=\inf_{(\beta_j)\in(\mathcal{B}(t,T))^J}\sup_{\alpha\in\mathcal{A}_0^\pi(t,T)}
\max_{i\in\{1,...,I\}}\{\bar{p}_i
-\sum_{j=1}^Jq_jE[g_{ij}(X_T^{t,x,\alpha,\beta_j})]\}.
\end{equation}
\end{lemma}
\begin{proof}
Define $\bar{V}_1^{\pi*}(t,x,\bar{p},q)=\inf_{(\beta_j)\in(\mathcal{B}(t,T))^J}\sup_{\alpha\in\mathcal{A}_0^\pi(t,T)}
\max_{i\in\{1,...,I\}}\{\bar{p}_i
-\sum_{j=1}^Jq_jE[g_{ij}(X_T^{t,x,\alpha,\beta_j})]\}.$
Similar to the proof of Lemma \ref{le 3.3}, we have
\begin{equation}\label{equ 1.15.1}
\bar{V}^{\pi*}(t,x,\bar{p},q)=\inf_{(\beta_j)\in(\mathcal{B}(t,T))^J}\sup_{\alpha\in\mathcal{A}^\pi(t,T)}
\max_{i\in\{1,...,I\}}\{\bar{p}_i
-\sum_{j=1}^Jq_jE[g_{ij}(X_T^{t,x,\alpha,\beta_j})]\}.
\end{equation}
Since $\mathcal{A}_0^\pi(t,T)\subset\mathcal{A}^\pi(t,T)$, we have $\bar{V}_1^{\pi*}(t,x,\bar{p},q)\leq \bar{V}^{\pi*}(t,x,\bar{p},q)$.
Now we prove the $\bar{V}_1^{\pi*}(t,x,\bar{p},q)\geq \bar{V}^{\pi*}(t,x,\bar{p},q)$.  For any $\alpha\in\mathcal{A}^\pi(t,T)$ , for any $y=(y_1,y_2,...,y_{2(N-k)-1})\in\mathbb{R}^{2(N-k)-1}$, it holds $\alpha(y,\cdot)\in\mathcal{A}_0^\pi(t,T)$. For any $(\beta_j)\in(\mathcal{B}(t,T))^J$,
 we have the following inequalities
\begin{eqnarray}\label{equ 3.15}
\begin{split}
&\sup_{\alpha\in\mathcal{A}^\pi(t,T)}\max_{i\in\{1,...,I\}}
\{\bar{p}_i-\sum_{j=1}^Jq_jE[g_{ij}(X_T^{t,x,\alpha,\beta_j})]\}\\
\leq&\sup_{\alpha\in\mathcal{A}^\pi(t,T)}\int_{[0,1]^{2(N-k)-1}}
\max_{i\in\{1,...,I\}}\{\bar{p}_i-\sum_{j=1}^Jq_jE[g_{ij}(X_T^{t,x,\alpha
((y_1,y_2,...,y_{2(N-k)-1})
,\cdot),\beta_j})]\}dy_1...dy_{2(N-k)-1}\\
\leq&\sup_{\alpha\in\mathcal{A}_0^\pi(t,T)}\sup_{y\in[0,1]^{2(N-k)-1}}\max_{i\in\{1,...,I\}}
\{\bar{p}_i-\sum_{j=1}^Jq_jE[g_{ij}(X_T^{t,x,\alpha
((y_1,y_2,...,y_{2(N-k)-1})
,\cdot),\beta_j})]\}\\
\leq&\sup_{\alpha\in\mathcal{A}_0^\pi(t,T)}\max_{i\in\{1,...,I\}}
\{\bar{p}_i-\sum_{j=1}^Jq_jE[g_{ij}(X_T^{t,x,\alpha,\beta_j})]\}.
\end{split}
\end{eqnarray}
Then taking infimum over ${(\beta_j)\in(\mathcal{B}(t,T))^J}$ on both side we get the desired result.
\end{proof}
\begin{lemma}\label{le 4.4}
For all the partition $\pi$ of the interval $[0,T]$,
 the convex conjugate function $\bar{V}^{\pi*}(t,x,\bar{p},q)$ is Lipschitz with respect to $(t,x,\bar{p},q)$,
 the concave conjugate function $\bar{W}^{\pi\#}(t,x,p,\bar{q})$ is Lipschitz with respect to $(t,x,p,\bar{q})$.
\end{lemma}
\begin{lemma}\label{le 4.5}
For any $(t,x,\bar{p},q)\in[t_{k-1},t_k)\times\mathbb{R}^n\times\mathbb{R}^I\times\Delta(J)$, and for all $l$ $(k\leq l \leq N)$,
we have
\begin{equation}\label{equ 4.6}
\begin{split}
\bar{V}^{\pi*}(t,x,\bar{p},q)&\leq \inf_{\beta\in\mathcal{B}(t,t_l)}\sup_{\alpha\in\mathcal{A}_0^\pi(t,t_l)}
E[\bar{V}^{\pi*}(t_l,X_{t_l}^{t,x,\alpha,\beta},\bar{p},q)]\\
&\leq \inf_{\beta\in\mathcal{B}(t,t_l)}\sup_{\alpha\in\mathcal{A}^\pi(t,t_l)}
E[\bar{V}^{\pi*}(t_l,X_{t_l}^{t,x,\alpha,\beta},\bar{p},q)].
\end{split}
\end{equation}
\end{lemma}
For the proof of this lemma, we give the following remarks.
\begin{remark}
The proof of the first inequality is similar to Lemma \ref{le 3.5} with the help of Lemma \ref{le 4.3} and one should be noticed that $\alpha\in\mathcal{A}_0^\pi(t,t_l)$ means
$\alpha$ is a deterministic strategy. The second inequality is obviously since $\mathcal{A}_0^\pi(t,t_l)\subset\mathcal{A}^\pi(t,t_l)$.
\end{remark}
\begin{lemma}\label{le 4.6}
There exists a subsequence of partitions $(\pi_n)_{n\geq 1}$, still denoted by $(\pi_n)_{n\geq 1}$, and two
 functions $\tilde{V}:[0,T]\times\mathbb{R}^n\times\mathbb{R}^I\times\Delta(J)\mapsto\mathbb{R}$
and $\tilde{W}:[0,T]\times\mathbb{R}^n\times\Delta(I)\times\mathbb{R}^J\mapsto\mathbb{R}$ such that $(\bar{V}^{\pi_n*},
\bar{W}^{\pi_n\#})\rightarrow(\tilde{V},\tilde{W})$ uniformly on compacts in $[0,T]\times\mathbb{R}^n\times\Delta(I)
\times\Delta(J)\times\mathbb{R}^I\times\mathbb{R}^J$. Furthermore, the functions $\tilde{V}$ and $\tilde{W}$ are Lipschitz continuous
with respect to all their variables.
\end{lemma}
\begin{lemma}\label{le 4.7}
The limit function $\tilde{V}(t,x,\bar{p},q)$ is
a viscosity subsolution of the same HJI equation $(\ref{equ 3.20})$.
\end{lemma}
Notice that
\begin{equation}\label{equ 4.6}
-\bar{W}^\pi(t,x,p,q)=\sup_{(\alpha_i)\in(\mathcal{A}^\pi(t,T))^I}
\inf_{(\beta_j)\in(\mathcal{B}(t,T))^J}
\sum_{i=1}^I\sum_{j=1}^Jp_iq_jE[-g_{ij}(X_T^{t,x,\alpha_i,\beta_j})].
\end{equation}
Hence, the convex conjugate of $(-\bar{W}^\pi)$ with respect to $q$, i.e.,
$-(\bar{W}^{\pi\#}(t,x,p,-\bar{q}))$ satisfying a sub-dynamic programming principle. Then, we have
the following lemma.
\begin{lemma}\label{le 4.8}
For any $(t,x,p,\bar{q})\in[t_{k-1},t_k)\times\mathbb{R}^n\times\Delta(I)\times\mathbb{R}^J$, and for all $l$ $(k\leq l \leq N)$,
we have
\begin{equation}\label{equ 4.7}
\begin{split}
\bar{W}^{\pi\#}(t,x,p,\bar{q})&\geq \sup_{\alpha\in\mathcal{A}^\pi(t,t_l)}\inf_{\beta\in\mathcal{B}(t,t_l)}
E[\bar{W}^{\pi\#}(t_l,X_{t_l}^{t,x,\alpha,\beta},p,\bar{q})].
\end{split}
\end{equation}
\end{lemma}
\begin{proposition}\label{prop 4.1}
The limit function $\tilde{W}(t,x,p,\bar{q})$ is
a viscosity supersolution of the same HJI equation (\ref{equ 3.20}) under the condition $(\ref{equ 4.1})$.
\end{proposition}
\begin{proof}
For simplicity, we denote $\tilde{W}(t,x,p,\bar{q})$ by $\tilde{W}(t,x)$, for fixed $(p,\bar{q})\in\Delta(I)\times\mathbb{R}^J$.
For any fixed $(t,x)\in[0,T]\times\mathbb{R}^n$, since the coefficient $f$ is bounded, there is some $M>0$ such that,
$\bar{B}_M(x)\supset\{X_r^{s,y,\alpha,\beta},\ (s,y)\in[0,T]\times\bar{B}_1(x), (\alpha,\beta)\in\mathcal{A}^\pi(s,T)\times\mathcal{B}(s,T), r\in[s,T] \}$, where $\bar{B}_M(x)$ is the closed ball with the center $x$ and the radius $M$. From Lemma \ref{le 4.6}, we know $\bar{W}^{\pi_n\#}$ converge to $\tilde{W}$ over $[0,T]\times\bar{B}_M(x)$.
Let $\varphi\in C_b^1([0,T]\times\mathbb{R}^n)$
 be a test function such that
\begin{equation}\label{equ 6.1}
(-\tilde{W}-(-\varphi))(t,x)>(-\tilde{W}-(-\varphi))(s,y),\ \text{for\ all}\ (s,y)\in[0,T]\times\bar{B}_M(x)\setminus\{(t,x)\}.
\end{equation}
Let $(s_n,x_n)\in[0,T]\times\bar{B}_M(x)$ be the maximum point of $-\bar{W}^{\pi_n\#}-(-\varphi)$ over $[0,T]\times\bar{B}_M(x)$, then
there exists a subsequence of $(s_n,x_n)$ still denoted by $(s_n,x_n)$, such that $(s_n,x_n)$ converges to $(t,x)$.

Indeed, since $[0,T]\times\bar{B}_M(x)$ is a compact set, there exists a subsequence $(s_n,x_n)$ and $(\bar{s},\bar{x})\in[0,T]\times\bar{B}_M(x)$ such that
$(s_n,x_n)\rightarrow(\bar{s},\bar{x})$. Due to $(-\bar{W}^{\pi_n\#}-(-\varphi))(s_n,x_n)\geq (-\bar{W}^{\pi_n\#}-(-\varphi))(t,x)$, for $n\geq 1$, we have
\begin{equation}\label{equ 6.2}
(-\tilde{W}-(-\varphi))(\bar{s},\bar{x})\geq (-\tilde{W}-(-\varphi))(t,x).
\end{equation}
From (\ref{equ 6.1}) and (\ref{equ 6.2}), we have $(\bar{s},\bar{x})=(t,x)$.\\
For the partition $\pi_n$, we assume $t^n_{k_{n-1}}\leq s_n<t^n_{k_n}$, for simplicity, we write $t^n_{k-1}\leq s_n<t^n_{k}$. Since $x_n\rightarrow x$, there
is a positive integer $N$ such that for all $n\geq N$, we have $|x_n-x|\leq 1$. Then from Lemma \ref{le 4.8}, we get
\begin{eqnarray}\label{equ 6.3}
\begin{split}
-\varphi(s_n,x_n)=- \bar{W}^{\pi_n\#}(s_n,x_n)
&\leq \inf_{\alpha\in\mathcal{A}^{\pi_n}(s_n,t^n_{k})}\sup_{\beta\in\mathcal{B}(s_n,t^n_{k})}
E[-\bar{W}^{\pi_n\#}(t^n_{k},X^{s_n,x_n,\alpha,\beta}_{t_{k}^n})]\\
&\leq \inf_{\alpha\in\mathcal{A}^{\pi_n}(s_n,t^n_{k})}\sup_{\beta\in\mathcal{B}(s_n,t^n_{k})}
E[-\varphi(t^n_{k},X^{s_n,x_n,\alpha,\beta}_{t_{k}^n})].
\end{split}
\end{eqnarray}
Thus we get
\begin{eqnarray}\label{equ 6.4}
\begin{split}
0\leq&\inf_{\alpha\in\mathcal{A}^{\pi_n}(s_n,t^n_{k})}\sup_{\beta\in\mathcal{B}(s_n,t^n_{k})}
E[-\varphi(t^n_{k},X^{s_n,x_n,\alpha,\beta}_{t_{k}^n})-(-\varphi(s_n,x_n))]\\
=& \inf_{\alpha\in\mathcal{A}^{\pi_n}(s_n,t^n_{k})}\sup_{\beta\in\mathcal{B}(s_n,t^n_{k})}
E[-\int_{s_n}^{t_{k}^n}(\frac{\partial \varphi}{\partial r}(r,X_r^{s_n,x_n,\alpha,\beta})+
f(X_r^{s_n,x_n,\alpha,\beta},\alpha_r,\beta_r)\cdot D\varphi(r,X_r^{s_n,x_n,\alpha,\beta}))dr].
\end{split}
\end{eqnarray}
For $(u,v)\in\mathcal{U}_{t,T}\times\mathcal{V}_{t,T}$, we introduce the following continuity modulus,
\begin{equation}\label{equ 6.5}
m(\delta):=\sup_{\mbox{\tiny
$\begin{array}{c}
|r-s|+|y-\bar{x}|\leq \delta,\\
u\in U,v\in V, \bar{x},y\in\bar{B}_M(x)
\end{array}$}
}\big|(\frac{\partial \varphi}{\partial r}(r,y)+
f(y,u,v)\cdot D\varphi(r,y))-(\frac{\partial \varphi}{\partial r}(s,\bar{x})+
f(\bar{x},u,v)\cdot D\varphi(s,\bar{x}))\big|.
\end{equation}
Obviously, $m(\delta)$ is nondecreasing in $\delta$ and $m(\delta)\rightarrow 0$, as $\delta\downarrow 0$. From (\ref{equ 2.1.1}), considering
that $|X_r^{s_n,x_n,\alpha,\beta}-x_n|\leq C|r-s_n|\leq C|t^n_{k}-s_n|,\ r\in[s_n,t^n_{k}]$ and
from $(\ref{equ 6.5})$
we know that
\begin{eqnarray}\label{equ 6.6}
\begin{array}{l}
|(\frac{\partial \varphi}{\partial r}(r,X_r^{s_n,x_n,\alpha,\beta})+
f(X_r^{s_n,x_n,\alpha,\beta},\alpha_r,\beta_r)\cdot D\varphi(r,X_r^{s_n,x_n,\alpha,\beta}))-\\
(\frac{\partial \varphi}{\partial r}(s_n,x_n)+
f(x_n,\alpha_r,\beta_r)\cdot D\varphi(s_n,x_n))|
\leq m(C|t_{k}^n-s_n|),\ r\in[s_n,t^n_{k}].
\end{array}
\end{eqnarray}
It follows from (\ref{equ 6.4}) and (\ref{equ 6.6}) that
\begin{equation}\label{equ 6.7}
\begin{split}
&-(t^n_{k}-s_n)\big(-\frac{\partial \varphi}{\partial r}(s_n,x_n)+m(C|t_{k}^n-s_n|)\big)\\
\leq&
\inf_{\alpha\in\mathcal{A}^{\pi_n}(s_n,t^n_{k})}\sup_{\beta\in\mathcal{B}(s_n,t^n_{k})}
E[\int_{s_n}^{t_{k}^n}(-f)(x_n,\alpha_r,\beta_r)\cdot D\varphi(s_n,x_n)dr]\\
\leq& \sup_{\beta\in\mathcal{B}(s_n,t^n_{k})}
E[\int_{s_n}^{t_{k}^n}(-f)(x_n,\tilde{\alpha}_r,\beta_r)\cdot D\varphi(s_n,x_n)dr],
\end{split}
\end{equation}
where we take $\tilde{\alpha}_r=\tilde{u}_k$, $r\in[t,t_k^n]$, $\tilde{u}_k\in U$.
Define $\rho_n=(t_k^n-s_n)^2$, then from (\ref{equ 6.7}) there exists a $\rho_n$-optimal strategy $\beta^n\in\mathcal{B}(s_n,t_k^n)$ (depending on $\tilde{\alpha}_r$) such that
\begin{equation}\label{equ 6.8}
-(t^n_{k}-s_n)\big(-\frac{\partial \varphi}{\partial r}(s_n,x_n)+m(C|t_{k}^n-s_n|)+(t_{k}^n-s_n)\big)\leq
E[\int_{s_n}^{t_{k}^n}(-f)(x_n,\tilde{\alpha}_r,\beta_r^n)\cdot D\varphi(s_n,x_n)dr].
\end{equation}
Since $\beta^n\in\mathcal{B}(s_n,t_k^n)$ there is some partition $\pi^0$, such that
$\beta^n\in\mathcal{B}^{\pi^0}(s_n,t_k^n)$, without loss of generality, $\pi^0\supset\pi_n$. Assume
$\{s_n=\theta_0<\theta_1<\ldots<\theta_m=t_k^n\}\subset\pi^0$. Therefore,
\begin{equation}\label{equ 6.9}
E[\int_{s_n}^{t_{k}^n}(-f)(x_n,\tilde{\alpha}_r,\beta_r^\rho)\cdot D\varphi(s_n,x_n)dr]=
\sum_{l=1}^m\int_{\theta_{l-1}}^{\theta_l}E[(-f)(x_n,\tilde{\alpha}_r,\beta_r^\rho)\cdot D\varphi(s_n,x_n)]dr,
\end{equation}
where $\beta^\rho$ depends on $(\zeta^{\pi^0}_{2,1},\ldots,\zeta^{\pi^0}_{2,l-1})$ on $[\theta_{l-1},\theta_l]$. Then
for  $r\in[\theta_{l-1},\theta_l]$, we have,
\begin{equation}\label{equ 6.10}
\begin{split}
&E[(-f)(x_n,\tilde{\alpha}_r,\beta_r^\rho)\cdot D\varphi(s_n,x_n)]
=E[(-f)(x_n,\tilde{u}_k,\beta_r^\rho
(\zeta^{\pi^0}_{2,1},\ldots,\zeta^{\pi^0}_{2,l-1}))\cdot D\varphi(s_n,x_n)]\\
&=\int_V(-f)(x_n,\tilde{u}_k,v)\cdot D\varphi(s_n,x_n)P_{\beta_r^\rho
(\zeta^{\pi^0}_{2,1},\ldots,\zeta^{\pi^0}_{2,l-1})}(dv)
\leq \sup_{\nu\in\mathcal{P}(V)}\int_V(-f)(x_n,\tilde{u}_k,v)\cdot D\varphi(s_n,x_n)\nu(dv).
\end{split}
\end{equation}
W e define $I(\tilde{u}_k):=\sup_{\nu\in\mathcal{P}(V)}\int_V(-f)(x_n,\tilde{u}_k,v)\cdot D\varphi(s_n,x_n)\nu(dv)$, from the arbitrariness of $\tilde{u}_k$, we can choose $\tilde{u}_k:=u^*$, such that
$I(u^*)=\min\limits_{\tilde{u}_k\in U} I(\tilde{u}_k)$. Then, for all $u\in U$, from (\ref{equ 6.10}), we have
\begin{eqnarray}\label{equ 6.11}
\begin{split}
\int_{\theta_{l-1}}^{\theta_l}I(\tilde{u}_k)dr&= \int_{\theta_{l-1}}^{\theta_l}I(u^*)dr
\leq \int_{\theta_{l-1}}^{\theta_l}\sup_{\nu\in\mathcal{P}(V)}\int_{V}(-f)(x_n,u,v)
\cdot D\varphi(s_n,x_n)d\nu(v)dr.
\end{split}
\end{eqnarray}
From (\ref{equ 6.11}) and the condition (\ref{equ 4.1}) we obtain
\begin{eqnarray}\label{equ 6.12}
\begin{split}
\int_{\theta_{l-1}}^{\theta_l}I(\tilde{u}_k)dr
\leq& (\theta_l-\theta_{l-1})\inf_{u\in U}\sup_{\nu\in\mathcal{P}(V)}\int_{V}(-f)(x_n,u,v)
\cdot D\varphi(s_n,x_n)d\nu(v)\\
=&(\theta_l-\theta_{l-1})\sup_{\nu\in\mathcal{P}(V)}\inf_{u\in U}\int_{V}(-f)(x_n,u,v)
\cdot D\varphi(s_n,x_n)d\nu(v)\\
=& (\theta_l-\theta_{l-1})\sup_{\nu\in\mathcal{P}(V)}\inf_{\mu\in\mathcal{P}(U)}\int_{U\times V}(-f)(x_n,u,v)
\cdot D\varphi(s_n,x_n)d\mu(u)d\nu(v).
\end{split}
\end{eqnarray}
From (\ref{equ 6.8}), (\ref{equ 6.9}), (\ref{equ 6.10}) and (\ref{equ 6.12}), we have
\begin{equation}\label{equ 6.13}
\begin{split}
&-(t^n_{k}-s_n)\big(-\frac{\partial \varphi}{\partial r}(s_n,x_n)+m(C|t_{k}^n-s_n|)+(t_{k}^n-s_n)\big)\\
 \leq&
(t_k^n-s_n)\sup_{\nu\in\mathcal{P}(V)}\inf_{\mu\in\mathcal{P}(U)}\int_{U\times V}(-f)(x_n,u,v)
\cdot D\varphi(s_n,x_n)\mu(du)\nu(dv).
\end{split}
\end{equation}
Then we have
\begin{equation}\label{equ 6.14}
\frac{\partial \varphi}{\partial r}(s_n,x_n)-m(C|t_{k}^n-s_n|)-(t_{k}^n-s_n)\big)
 \leq
\sup_{\nu\in\mathcal{P}(V)}\inf_{\mu\in\mathcal{P}(U)}\int_{U\times V}(-f)(x_n,u,v)
\cdot D\varphi(s_n,x_n)\mu(du)\nu(dv).
\end{equation}
Recall that $(s_n,x_n)\rightarrow (t,x)$ and $0\leq (t_k^n-s_n)\leq (t_k^n-t_{k-1}^n)\leq |\pi_n|$, when $n\rightarrow\infty$ we get
\begin{equation}\label{equ 3.31}
\frac{\partial\varphi}{\partial t}(t,x)+\inf_{\nu\in\mathcal{P}(V)}\sup_{\mu\in\mathcal{P}(U)}\int_{U\times V}
f(x,u,v)\cdot D\varphi(t,x)\mu(du)\nu(dv)\leq 0.
\end{equation}
Therefore, $\tilde{W}(t,x,p,\bar{q})$ is a viscosity supersolution of the HJI equation (\ref{equ 3.20}).
\end{proof}
Similar to Section 3 (Proposition \ref{prop 3.1} and Theorem \ref{th 3.2}), we have the following results.
\begin{proposition}\label{prop 4.2}
If condition (\ref{equ 4.1}) holds, then for all sequences of partitions $(\pi_n)$ of the interval $[0,T]$ with
$|\pi_n|\rightarrow 0$, as $n\rightarrow\infty$, there exists a subsequence of partitions, still denoted by $(\pi_n)_{n\geq 1}$
such that $(\bar{V}^{\pi_n})$ and $(\bar{W}^{\pi_n})$ converges uniformly on compacts to a couple $(U,U)$,
and the function $U$ is the unique dual solution
of the HJI equation $(\ref{equ 3.34})$.
\end{proposition}
\begin{theorem}\label{th 4.1}
Suppose condition (\ref{equ 4.1}) holds, then for all sequences of partitions $(\pi_n)$ with
$|\pi_n|\rightarrow 0$, the sequences $(\bar{V}^{\pi_n})$ and $(\bar{W}^{\pi_n})$ converge uniformly
on compacts to the same Lipschitz continuous function $U$. Moreover, the function $U$ is the unique dual solution
of the HJI equation $(\ref{equ 3.34})$.
\end{theorem}
Similar to $(\bar{W}^{\pi_n},\bar{V}^{\pi_n})$, we obtain the following theorem.
\begin{theorem}\label{th 4.2}
Suppose condition (\ref{equ 4.2}) holds, then for all sequences of partitions $(\pi_n)$ with
$|\pi_n|\rightarrow 0$, the sequences $(\bar{\bar{V}}^{\pi_n})$ and $(\bar{\bar{W}}^{\pi_n})$ converge uniformly
on compacts to the same Lipschitz continuous function $U$. Moreover, the function $U$ is the unique dual solution
of the HJI equation $(\ref{equ 3.34})$.
\end{theorem}
Now from Theorem \ref{th 5.1.24.1}, Theorem \ref{th 4.1} and Theorem \ref{th 4.2}, we get the following result.
\begin{theorem}\label{th 4.3}
Under the conditions $(\ref{equ 4.1})$ and $(\ref{equ 4.2})$, the function $W(t,x,p,q)$ is equal to $V(t,x,p,q)$, for all compacts in $[0,T]
\times\mathbb{R}^n\times\Delta(I)\times\Delta(J)$.
\end{theorem}
\begin{proof}
We have shown that the value functions $W^\pi(t,x,p,q),V^\pi(t,x,p,q)$, $\bar{W}^\pi(t,x,p,q),\bar{V}^\pi(t,x,p,q)$,
$\bar{\bar{W}}^\pi(t,x,p,q),\bar{\bar{V}}^\pi(t,x,p,q)$  converges uniformly on compacts to the function
 $U(t,x,p,q)$, as $|\pi|\rightarrow0$, under the assumptions $(\ref{equ 4.1})$ and $(\ref{equ 4.2})$,  where the function $U(t,x,p,q)$ is the unique solution of  the HJI equation (\ref{equ 3.34}).

Then, from the definition of $W(t,x,p,q)$, for any $\varepsilon>0$, there exist
$\hat{\alpha}^\varepsilon\in(\mathcal{A}(t,T))^I$, such that
\begin{equation}\label{equ 4.23}
\varepsilon+W(t,x,p,q)\geq\sup_{\hat{\beta}\in(\mathcal{B}(t,T))^J}J(t,x,\hat{\alpha}^\varepsilon,\hat{\beta},p,q).
\end{equation}
For $\hat{\alpha}^\varepsilon\in(\mathcal{A}(t,T))^I$, there exist a partition $\pi^\varepsilon$, such that $\hat{\alpha}^\varepsilon\in(\mathcal{A}^{\pi^\varepsilon}(t,T))^I\subset(\mathcal{A}^\pi(t,T))^I$, for $\pi\supset\pi^\varepsilon$.
Thus for all $\pi\supset\pi^\varepsilon$, it holds that
\begin{equation}\label{equ 4.24}
\varepsilon+W(t,x,p,q)\geq\sup_{\hat{\beta}\in(\mathcal{B}^\pi(t,T))^J}J(t,x,\hat{\alpha}^\varepsilon,\hat{\beta},p,q)\geq W^\pi(t,x,p,q).
\end{equation}
From the arbitrariness of $\varepsilon$, we have $W(t,x,p,q)\geq W^{\pi}(t,x,p,q)$, then let $|\pi|\rightarrow 0$,
we have
\begin{equation}\label{equ 4.25}
W(t,x,p,q)\geq U(t,x,p,q).
\end{equation}
Similarly, we have
\begin{equation}\label{equ 4.26}
U(t,x,p,q)\geq V(t,x,p,q).
\end{equation}

On the other hand, since $W(t,x,p,q)\leq\inf\limits_{\hat{\alpha}\in(\mathcal{A}^\pi(t,T))^I}\sup\limits_{\hat{\beta}\in(\mathcal{B}(t,T))^J}
J(t,x,\hat{\alpha},\hat{\beta},p,q)=\bar{W}^\pi(t,x,p,q)$, let $|\pi|\rightarrow0$, we have
\begin{equation}\label{equ 4.27}
W(t,x,p,q)\leq U(t,x,p,q).
\end{equation}
Similarly, since $V(t,x,p,q)\geq \bar{\bar{V}}^\pi(t,x,p,q)$,
let $|\pi|\rightarrow0$, we have
\begin{equation}\label{equ 4.28}
V(t,x,p,q)\geq U(t,x,p,q).
\end{equation}

Now from $(\ref{equ 4.25})$, $(\ref{equ 4.26})$, $(\ref{equ 4.27})$ and $(\ref{equ 4.28})$, we get $W(t,x,p,q)=U(t,x,p,q)= V(t,x,p,q)$.
\end{proof}
Now we give a example to explain that the conditions (\ref{equ 4.1}) and (\ref{equ 4.2}) are necessary even for the games with symmetric information.
\begin{example}
We assume $U=V=[-1,1]$, $I=J=1$, $g(x)=x$, $f(x,u,v)=|u-v|^2$. For any given $(t,x)\in[0,T]\times\mathbb{R}^n$, the dynamic is
$$X_s=x+\int_t^s|u_s-v_s|^2ds,\ s\in[t,T].$$
The payoffs $$J(t,x,\alpha,\beta)=E[g(X_T^{t,x,\alpha,\beta})]=E[X_T^{t,x,\alpha,\beta}].$$
From Remark \ref{re 2.1}, we know there exists the unique $(u,v)\in\mathcal{U}_{t,T}\times\mathcal{V}_{t,T}$, such that
$\alpha(v)=u$, $\beta(u)=v$, then we have
\begin{equation}\label{equ 123321}
J(t,x,\alpha,\beta)=x+E[\int_t^T|\alpha_r-\beta_r|^2dr].
\end{equation}
For any $x\in\mathbb{R}$, $(u,v)\in U\times V$, $p\in\mathbb{R}$, the Hamiltonian function
$H(x,u,v,p)=|u-v|^2p$. Moreover, we define
$$\tilde{H}^+(x,p)\triangleq\inf_{u\in U}\sup_{v\in V}H(x,u,v,p)=\inf_{u\in U}(1+|u|)^2p^+=p^+;$$
$$\tilde{H}^-(x,p)\triangleq\sup_{v\in V}\inf_{u\in U}H(x,u,v,p)=\sup_{v\in V}\big(-(1+|v|)^2p^-\big)=-p^-.$$
Obviously, for any $p\neq0$, $\tilde{H}^-(x,p)=-p^-<p^+=\tilde{H}^+(x,p)$. For the measure-valued controls
$\mu\in\mathcal{P}(U), \nu\in\mathcal{P}(V)$ and $(x,p)\in\mathbb{R}^2$, the Hamiltonian function
$H(x,\mu,\nu,p)=\int_U\int_V|u-v|^2p\nu(dv)\mu(du).$
We define
$$H^+(x,p)=\inf_{\mu\in\mathcal{P}(U)}\sup_{\nu\in\mathcal{P}(V)}H(x,\mu,\nu,p);\ \
H^-(x,p)=\sup_{\nu\in\mathcal{P}(V)}\inf_{\mu\in\mathcal{P}(U)}H(x,\mu,\nu,p).$$
Then we know $H(x,p):=H^+(x,p)=H^-(x,p)$. Now we compute $H^-(x,p)$.
For the case $p\geq 0$, since
\begin{equation}
\begin{split}
&H(x,\mu,\nu,p)=\int_V\int_U|u-v|^2p\mu(du)\nu(dv)\\
&=\int_V\int_U\big(|v-\int_Vv\nu(dv)|^2+|u-\int_Vv\nu(dv)|^2\big)p\mu(du)\nu(dv)\\
&\geq \int_V|v-\int_Vv\nu(dv)|^2p\nu(dv)=H(x,\delta_{\int_Vv\nu(dv)},\nu,p),
\end{split}
\end{equation}
we know
$\inf_{\mu\in\mathcal{P}(U)}H(x,\mu,\nu,p)=H(x,\delta_{\int_Vv\nu(dv)},\nu,p)=\big(\int_V|v-\int_Vv\nu(dv)|^2\nu(dv)\big)p.$ Therefore,
$H^-(x,p)=\sup_{\nu\in\mathcal{P}(V)}(\int_V|v-\int_Vv\nu(dv)|^2\nu(dv))p=\sup_{\nu\in\mathcal{P}(V)}
\Big(\int_Vv^2\nu(dv)-(\int_Vv\nu(dv))^2\Big)p\leq p$, and when $\nu=\frac{1}{2}(\delta_1+\delta_{-1})$ it attains the maximum value.\\
For the case $p<0$, $H^-(x,p)=-\Big(\inf_{\nu\in\mathcal{P}(V)}\sup_{\mu\in\mathcal{P}(U)}\int_U\int_V|u-v|^2\mu(du)\nu(dv)(-p)\Big)=-(-p)^+=p.$
Thus, we get $H(x,p)=H^-(x,p)=p$.\\
The corresponding HJI equation is as follows:
\begin{equation}
\left\{
\begin{array}{l}
\partial_tV+H(x,\partial_xV)=0,\\
V(T,x)=g(x)=x.
\end{array}
\right.
\end{equation}
Notice that $V(t,x)=x+T-t,$ $t\in[0,T],x\in\mathbb{R}$ is the solution of this equation.\\
If both players use the same partition $\pi$, as $|\pi|\rightarrow 0$, from Theorem \ref{th 3.2} we have $(W^\pi(t,x),V^\pi(t,x))$
converge to the same function $V(t,x)$.
If not, for example, we calculate\\  $$\bar{W}^\pi(t,x)=\inf_{\alpha\in\mathcal{A}^\pi(t,T)}\sup_{\beta\in\mathcal{B}(t,T)}E[g(X_T^{t,x,\alpha,\beta})].$$
For Player \Rmnum {1}, we assume the partition $\pi=\{0=t_0<t_1<\cdots<t_N=T\}$, $t\in[t_{k-1},t_k)$, without loss of generality, we assume $t=t_{k-1}$ and
$\alpha_r=\alpha_k(\zeta^\pi_{1,k-1})_r$, $r\in[t_k,t_{k+1}]$.\\
For Player \Rmnum {2}, we consider the partition $\pi_n=\{0=(t_0=)S_0<S_1<\ldots<S_{2^n}(=t_1)<S_{2^n+1}<\ldots
<S_{2^n+2^n}(=t_2)<\ldots<S_{(l-1)2^n+j}<\ldots<S_{(N-1)2^n+2^n}(=t_N)=T\}$, where
$S_{(l-1)2^n+j}=t_{l-1}+j(t_l-t_{l-1})2^{-n}$, $0\leq j\leq 2^n$, $1\leq l\leq N.$
On the subinterval $\Delta_m^{l,n}:=[S_{(l-1)2^n+m-1},S_{(l-1)2^n+m}]$, Player \Rmnum {2}
uses the strategy $\beta_r=\beta_{S_{(l-1)2^n+m}}(\zeta^{\pi_n}_{2,(l-1)2^n+m-1},\alpha|_{[t,S_{(l-1)2^n+m-1}]})_r$, $1\leq m\leq 2^n$,
$k\leq l\leq N.$
Obviously, $\beta_r(u)\triangleq u_{r-|\pi_n|}$ satisfy the above situation.\\
Now let $\beta_r(u)\triangleq\left\{
\begin{array}{ll}
-\mathop{\rm sgn}(u_{r-|\pi_n|}),\  r\in[t+|\pi_n|,T],\\
0,\ r\in[t,t+|\pi_n|],
\end{array}
\right.
$  and from the equation (\ref{equ 123321}), we get
\begin{equation}\label{equ 123333}
\begin{split}
&J(t,x,\alpha,\beta)=x+\int_t^{t+|\pi_n|}E[|\alpha_r|^2]dr+E[\int_{t+|\pi_n|}^T|\alpha_r+\mathop{\rm sgn}(\alpha_{r-\pi_n})|^2dr]\\
&=x+E[\int_t^T|\alpha_r+\mathop{\rm sgn}(\alpha_r)|^2dr]+R_n=x+E[\int_t^T(1+|\alpha_r|)^2dr]+R_n,
\end{split}
\end{equation}
where $R_n=\int_t^{t+|\pi_n|}E[|\alpha_r|^2]dr-\int_t^{t+|\pi_n|}E[|\alpha_r+\mathop{\rm sgn}(\alpha_{r-|\pi_n|})|^2]dr+
E[\int_t^T(|\alpha_r+\mathop{\rm sgn}(\alpha_{r-|\pi_n|})|^2-|\alpha_r+\mathop{\rm sgn}(\alpha_{r})|^2)dr]\leq |\pi_n|+4|\pi_n|+
4E[\int_t^T|\mathop{\rm sgn}(\alpha_r)-\mathop{\rm sgn}(\alpha_{r-|\pi_n|})|dr]$. Now we give the following lemma to explain
$R^n\rightarrow 0$, as $n\rightarrow 0$.
\begin{lemma}\label{le e.g.1}
For all $u\in L^2([0,T])$, it holds $\lim_{\varepsilon\rightarrow 0}\int_0^T|u_s-u_{s-\varepsilon}|^2ds=0.$
\end{lemma}
\begin{proof}
For any fixed $u\in L^2([0,T])$, since $C^1([0,T])$ is dense in $L^2([0,T])$, for all $\rho>0$, there exists
$u^\rho\in C^1([0,T])$, such that $(\int_0^T|u_s-u_s^\rho|^2ds)^{\frac{1}{2}}\leq \rho$. Then, we have\\
$(\int_0^T|u_s-u_{s-\varepsilon}|^2ds)^{\frac{1}{2}}\leq C\rho+C(\int_0^T|u_s^\rho-u^\rho_{s-\varepsilon}|^2ds)^{\frac{1}{2}}\leq
C\rho+C\varepsilon\sqrt{T}\rightarrow 0,$ when $\varepsilon$ and $\rho\rightarrow 0$.
\end{proof}
From this lemma we know that when $n\rightarrow\infty$, $R_n\rightarrow 0.$ Thus, from our choice of $\beta\in\mathcal{B}^{\pi_n}(t,T)$
and the equation (\ref{equ 123333})
we have
\begin{equation}\label{equ e.g.1}
J(t,x,\alpha,\beta)=x+E[\int_t^T(1+|\alpha_r|)^2dr]+R_n,\ \text{and}\ R_n\rightarrow 0, \text{as}\ n\rightarrow\infty.
\end{equation}
On the other hand, for any $\beta\in\mathcal{B}(t,T)\triangleq \cup_{\pi'\supset\pi}\mathcal{B}^{\pi'}(t,T)$, we have
\begin{equation}\label{equ e.g.2}
J(t,x,\alpha,\beta)=x+E[\int_t^T|\alpha_r-\beta_r|^2dr]\leq x+E[\int_t^T(1+|\alpha_r|)^2dr].
\end{equation}
From (\ref{equ e.g.1}) and (\ref{equ e.g.2}), we have
$\sup_{\beta\in\mathcal{B}(t,T)}J(t,x,\alpha,\beta)=x+E[\int_t^T(1+|\alpha_r|)^2dr].$ Then
\begin{equation}
\bar{W}^\pi(t,x)=\inf_{\alpha\in\mathcal{A}^\pi(t,T)}\sup_{\beta\in\mathcal{B}(t,T)}E[g(X_T^{t,x,\alpha,\beta})]=x+(T-t)=V(t,x).
\end{equation}
We assume $\pi$ and $\pi_n$ as before, choose $\alpha_r=
\left\{
\begin{array}{ll}
0,&\ [t,t+|\pi_n|];\\
\beta_{r-|\pi_n|},&\ [t+|\pi_n|,T];
\end{array}
\right.
$,  then we have
$J(t,x,\alpha,\beta)=x+E[\int_t^{t+|\pi_n|}|\beta_r|^2dr]+E[\int_{t+|\pi_n|}^T|\beta_r-\beta_{r-\pi_n}|^2dr]$,
when $n\rightarrow \infty$, from Lemma \ref{le e.g.1}, we have
$$\inf_{\alpha\in\mathcal{A}(t,T)}E[g(X_T^{t,x,\alpha,\beta})]=x.$$ Then, we know
\begin{equation}\label{equ e.g.3}
\bar{\bar{V}}^\pi(t,x)=\sup_{\beta\in\mathcal{B}^\pi(t,T)}\inf_{\alpha\in\mathcal{A}(t,T)}E[g(X_T^{t,x,\alpha,\beta})]=x.
\end{equation}
Using the same argument we have
\begin{equation}\label{equ e.g.4}
\begin{split}
W(t,x)&=\inf_{\alpha\in\mathcal{A}(t,T)}\sup_{\beta\in\mathcal{B}(t,T)}E[g(X_T^{t,x,\alpha,\beta})]=x+(T-t);\\
V(t,x)&=\sup_{\beta\in\mathcal{B}(t,T)}\inf_{\alpha\in\mathcal{A}(t,T)}E[g(X_T^{t,x,\alpha,\beta})]=x.
\end{split}
\end{equation}
Obviously, the upper value function $W(t,x)$ is not equal to the lower value function $V(t,x)$ if we do not consider the conditions (\ref{equ 4.1}) and (\ref{equ 4.2}).
\end{example}


 \end{document}